\documentclass[
a4paper 
,11pt 
,leqno 
,english
]{scrartcl} 
\usepackage{my-preamble-article} 
\usepackage{my-macros} 
\usepackage{my-specific-macros} 
\usepackage{my-counters-related-preprints} 
\title{Global behaviour of bistable solutions for hyperbolic gradient systems in one unbounded spatial dimension}
\author{Emmanuel \textsc{Risler}}
\begin{document}
\maketitle
\begin{abstract}
This paper is concerned with damped hyperbolic gradient systems of the form 
\[
\alpha u_{tt} + u_t = -\nabla V(u) + u_{xx}\,,
\]
where the spatial domain is the whole real line, the state variable $u$ is multidimensional, $\alpha$ is a positive quantity, and the potential $V$ is coercive at infinity. For such systems, under generic assumptions on the potential, the asymptotic behaviour of every \emph{bistable solution} --- that is, every solution close at both ends of space to stable homogeneous equilibria --- is described. Every such solution approaches, far to the left in space a stacked family of bistable fronts travelling to the left, far to the right in space a stacked family of bistable fronts travelling to the right, and in between a pattern of profiles of stationary solutions homoclinic or heteroclinic to stable homogeneous equilibria, going slowly away from one another. In the absence of maximum principle, the arguments are purely variational. This extends previous results obtained in companion papers for damped wave equations or parabolic gradient systems, in the spirit of the program initiated in the late seventies by Fife and McLeod on the global asymptotic behaviour of bistable solutions for parabolic equations.
\end{abstract}
\nnfootnote{%
}
\thispagestyle{empty} 
\pagestyle{empty}
\hypersetup{pageanchor=false} 
\newpage
\tableofcontents
\newpage
\hypersetup{pageanchor=true} 
\pagestyle{plain}
\setcounter{page}{1}
\section{Introduction}
This paper deals with the global dynamics of nonlinear hyperbolic systems of the form
\begin{equation}
\label{hyp_syst}
\alpha u_{tt} + u_t = -\nabla V(u) + u_{xx}
\,,
\end{equation}
where the time variable $t$ and the space variable $x$ are real, the spatial domain is the whole real line, the function $(x,t)\mapsto u(x,t)$ takes its values in $\rr^d$ with $d$ a positive integer, $\alpha$ is a positive quantity, and the nonlinearity is the gradient of a scalar \emph{potential} function $V:\rr^d\to\rr$, which is assumed to be regular (of class $\ccc^2$) and coercive at infinity (see hypothesis \cref{hyp_coerc} in \vref{subsec:coerc_glob_exist}).

The aim of this paper is to extend to hyperbolic systems of the form \cref{hyp_syst} the results describing the global asymptotic behaviour of bistable solutions obtained in \cite{Risler_globalRelaxation_2016,Risler_globalBehaviour_2016} for parabolic systems of the form 
\begin{equation}
\label{parab_syst}
u_t = -\nabla V(u) + u_{xx}
\,.
\end{equation}
As was already observed by several authors, the long-time asymptotics of solutions of the two systems \cref{hyp_syst,parab_syst} present strong similarities, see \cite{GallayJoly_globStabDampedWaveBistable_2009} and references therein. The common feature of theses two systems that will be extensively used in this paper is the existence --- at least formally --- of an energy functional, not only for solutions considered in the laboratory frame (at rest), but also for solutions considered in every frame travelling at a constant speed. 

If $(v,w)$ is a pair of vectors of $\rr^d$, let $v\cdot w$ and $\abs{v} =\sqrt{v\cdot v}$ denote the usual Euclidean scalar product and the usual Euclidean norm, respectively, and let us write simply $v^2$ for $\abs{v}^2$. If $(x,t)\mapsto u(x,t)$ is a solution of system~\cref{hyp_syst}, the (formal) \emph{energy} of the solution reads
\begin{equation}
\label{form_en}
\eee[u(\cdot,t)] = \int_\rr\Bigl( \frac{\alpha}{2}u_t(x,t)^2 + \frac{1}{2}u_x(x,t)^2 + V\bigl(u(x,t)\bigr)\Bigr) \, dx
\,,
\end{equation}
and its time derivative reads, at least formally,
\begin{equation}
\label{dt_form_en}
\frac{d}{d t}\eee[u(\cdot,t)] = -\int_{\rr} u_t(x,t)^2\, dx \le 0
\,.
\end{equation}
In the parabolic case $\alpha=0$, the same properties hold with the same expression for the energy (the inertial term involving $\alpha$ vanishes); by the way, an additional feature in this case is the fact that the parabolic system \cref{parab_syst} is nothing but the (formal) gradient of energy functional \cref{form_en} (this does not hold for hyperbolic system \cref{hyp_syst}). 

A striking feature of both systems \cref{hyp_syst,parab_syst} is the fact that a formal (Lyapunov) energy functional exists not only in the laboratory frame, but also in every frame travelling at a constant speed (see \vref{subsubsec:energy_L2_trav} and specifically equality \cref{ddt_formal_en_tf}). In the parabolic case, this is known for long and was in particular used by P. C. Fife and J. B. McLeod to prove  global convergence towards bistable fronts and to study the global behaviour of bistable solutions in the scalar case $d$ equals $1$, \cite{FifeMcLeod_approachTravFront_1977,Fife_longTimeBistable_1979,FifeMcLeod_phasePlaneDisc_1981}. More recently, this property received a detailed attention from several authors (among which S. Heinze, C. B. Muratov, Th. Gallay, and the author \cite{Heinze_variationalApproachTW_2001,Muratov_globVarStructPropagation_2004,GallayRisler_globStabBistableTW_2007,Risler_globCVTravFronts_2008}), and it was shown that this structure is sufficient (in itself, that is without the use of the maximum principle) to prove results of global convergence towards travelling fronts. In the hyperbolic case, a similar strategy was successfully applied by Th. Gallay and R. Joly in the scalar case $d$ equals $1$ to prove global stability of travelling fronts for a bistable potential \cite{GallayJoly_globStabDampedWaveBistable_2009}. These ideas have been applied since in different contexts, to prove either global convergence or just existence results, see for instance \cite{Chapuisat_existenceCurvedFront_2007,ChapuisatJoly_asymptProfilesTravFrontBiolEqu_2010,MuratovNovaga_frontPropIVariational_2008,MuratovNovaga_frontPropIISharpReaction_2008,MuratovNovaga_globExpConvTW_2012,AlikakosKatzourakis_heteroclinicTW_2011,AlikakosFusco_ellipticSystemsPhaseTransType_2018,Luo_globStabDampedWaveEqu_2013,BouhoursNadin_variationalApproachRDForcedSpeedDim1_2015,BouhoursGiletti_extinctSpreadClimateAllee_2016,BouhoursGiletti_spreadVanishMonStabRDEqu_2018,OliverBonafoux_heteroclinicTW1dParabSystDegenerate_2021,OliverBonafoux_TWparabAllenCahn_2021,ChenChienHuang_varApproach3PhaseTWgradSyst_2021,ChenCotiZelati_TWSolAllenCahnEqu_2022,OliverBonafouxRisler_globCVPushedTravFronts_2023}. 
Using the same strategy, a full description of the global asymptotic behaviour of every bistable solution was recently obtained for parabolic systems \cite{Risler_globalRelaxation_2016,Risler_globalBehaviour_2016}. Roughly speaking, such a solution must approach: 
\begin{itemize}
\item far to the right a stacked family of fronts travelling to the right, 
\item far to the left a stacked family of fronts travelling to the left, 
\item in between a pattern made of bistable stationary solutions (possibly a singe homogeneous stable equilibrium) getting slowly away from one another.
\end{itemize} 
The aim of this paper is to extend this result to the case of hyperbolic systems of the form \cref{hyp_syst} (\vref{thm:1}). This will also provide an extension of the global stability result obtained par Gallay and Joly in the scalar case $d$ equals $1$ \cite{GallayJoly_globStabDampedWaveBistable_2009}. 
\section{Assumptions, notation, and statement of the results}
\subsection{Semi-flow in uniformly local Sobolev space and coercivity hypothesis}
\label{subsec:coerc_glob_exist}
Let us assume that the potential function $V:\rr^d\to\rr$ is of class $\ccc^2$ and that this potential function is strictly coercive at infinity in the following sense: 
\begin{gather}
\tag{$\text{H}_\text{coerc}$}
\lim_{R\to+\infty}\quad  \inf_{\abs{u}\ge R}\ \frac{u\cdot \nabla V(u)}{\abs{u}^2} >0
\label{hyp_coerc}
\end{gather}
(or in other words there exists a positive quantity $\varepsilon$ such that the quantity $u\cdot \nabla V(u)$ is greater than or equal to $\varepsilon\abs{u}^2$ as soon as $\abs{u}$ is large enough). 

System \cref{hyp_syst} defines a local semi-flow on the uniformly local energy space
\[
\HulofR{1} \times \LtwoulofR
\,,
\]
and, according to hypothesis \cref{hyp_coerc}, this semi-flow is actually global (see \vref{prop:exist_sol_att_ball}). Let us denote by $(S_t)_{t\ge0}$ this semi-flow. 

In the following, a \emph{solution of system \cref{hyp_syst}} will refer to a function 
\[
\rr\times[0,+\infty)\to\rr^d\,, \quad (x,t)\mapsto u(x,t)
\,,
\]
such that the function $u_0:x\mapsto u(x,t=0)$  is in $\HulofR{1}$, the function $\tilde{u}_0:x\mapsto u_t(x,t=0)$) is in $\LtwoulofR$, and $\bigl(u(\cdot,t),u_t(\cdot,t)\bigr)$ equals $S_t(u_0,\tilde{u}_0)$ for every nonnegative time $t$. 
\subsection{Minimum points and bistable solutions}
\subsubsection{Minimum points}
Everywhere in this paper, the term ``minimum point'' denotes a point where a function --- namely the potential $V$ --- reaches a local \emph{or} global minimum. 
\begin{notation}
Let $\mmm$ denote the set of nondegenerate minimum points of $V$:
\[
\mmm=\{u\in\rr^d: \nabla V(u)=0 
\quad\text{and}\quad 
D^2V(u)\text{ is positive definite}\}
\,.
\]
\end{notation}
\subsubsection{Bistable solutions}
Let us recall the following definition, already stated in \cite{Risler_globalRelaxation_2016}.
\begin{definition}[bistable solution]
\label{def_bist}
A solution $(x,t)\mapsto u(x,t)$ of system~\cref{hyp_syst} is called a \emph{bistable solution} if there are two (possibly equal) points $m_-$ and $m_+$ in $\mmm$ such that the quantities
\[
\limsup_{x\to-\infty} \abs{u(x,t)-m_-}
\quad\text{and}\quad
\limsup_{x\to+\infty} \abs{u(x,t)-m_+}
\]
both approach $0$ as time goes to $+\infty$. More precisely, such a solution is called a \emph{bistable solution connecting $m_-$ to $m_+$} (see \cref{fig:bist_sol}). 
\begin{figure}[!htbp]
	\centering
    \includegraphics[width=0.6\textwidth]{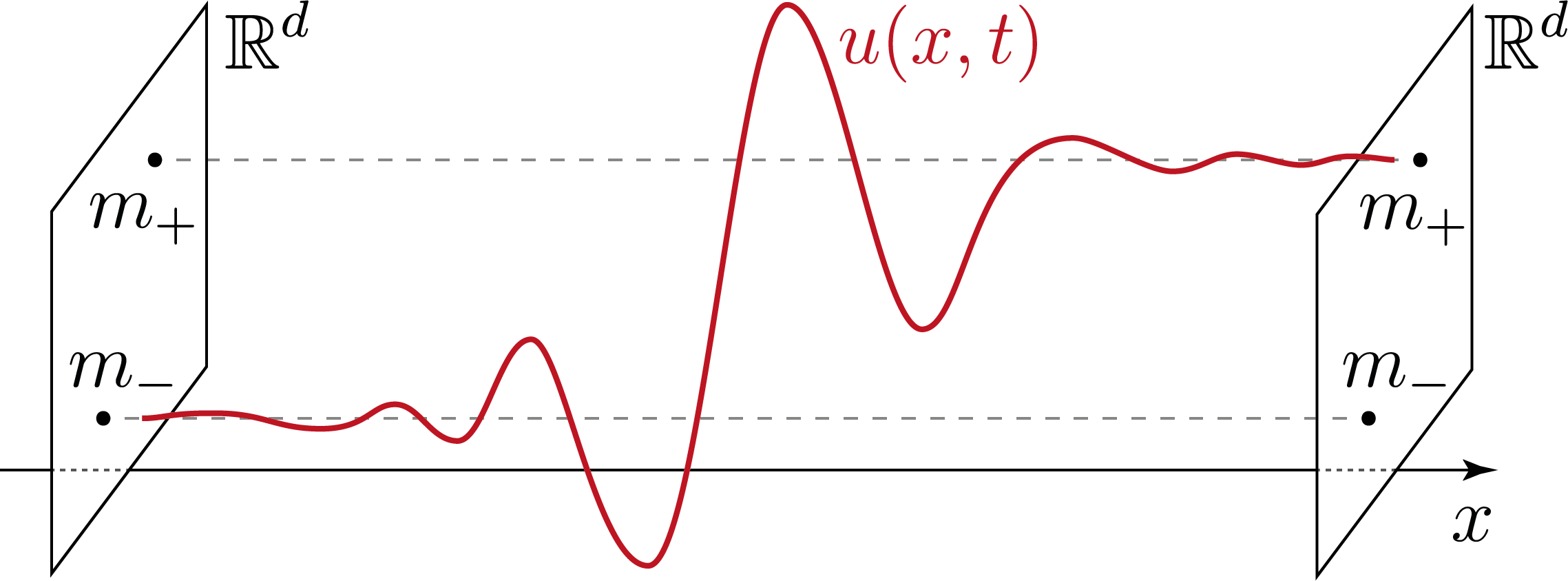}
    \caption{A bistable solution connecting $m_-$ to $m_+$.}
    \label{fig:bist_sol}
\end{figure}
\end{definition}
\subsection{Stationary solutions, travelling fronts, terraces, asymptotic pattern}
\subsubsection{Stationary solutions and travelling fronts}
\label{subsubsec:stat_sol_trav_fronts}
Let $c$ be a real quantity. A function
\[
\phi:\rr\to\rr^d,
\quad \xi\mapsto\phi(\xi)
\]
is the profile of a wave travelling at the speed $c$ (or is a stationary solution if $c$ vanishes) for the parabolic system \cref{parab_syst} if the function  $(x,t)\mapsto \phi(x-ct)$ is a solution of this system, that is if $\phi$ is a solution of the differential system
\begin{equation}
\label{syst_trav_front}
\phi''=-c\phi'+\nabla V(\phi) 
\,.
\end{equation}
In this case, for every real quantity $x_0$, the function
\[
(x,t)\mapsto \phi\bigl(\sqrt{1+\alpha c^2}\, x - ct - x_0\bigr)
\]
is a solution of the hyperbolic system \cref{hyp_syst}, more precisely a wave travelling at the \emph{physical speed} $\sigma$ related to the \emph{parabolic speed} $c$ by 
\[
\sigma = \frac{c}{\sqrt{1+\alpha c^2}} 
\iff
c = \frac{\sigma}{\sqrt{1-\alpha \sigma^2}}
\,.
\]
System \cref{syst_trav_front} can be viewed as a damped oscillator (or a conservative oscillator if $c$ vanishes) in the potential $-V$, the speed $c$ playing the role of the damping coefficient. 
\begin{notation}
If $m_-$ and $m_+$ are critical points of $V$ and $c$ is a real quantity, let $\Phi_c(m_-,m_+)$ denote the set of \emph{nonconstant} global solutions of system \cref{syst_trav_front} connecting $m_-$ to $m_+$. With symbols, 
\[
\begin{aligned}
\Phi_c(m_-,m_+) = \bigl\{ &
\phi:\rr\to\rr^d : 
\phi \text{ is a \emph{nonconstant} global solution of system \cref{syst_trav_front}}
\\
& \text{and}\quad\phi(\xi)\xrightarrow[\xi\to -\infty]{} m_-
\quad\text{and}\quad
\phi(\xi)\xrightarrow[\xi\to +\infty]{} m_+
\bigr\}
\,.
\end{aligned}
\]
And, if the quantity $c$ is positive, let $\Phi_c(m_+)$ denote the set of \emph{nonconstant} global and bounded solutions of system \cref{syst_trav_front} converging to $m_+$ at the right end of space. With symbols, 
\[
\begin{aligned}
\Phi_c(m_+) = \bigl\{ &
\phi:\rr\to\rr^d : 
\phi \text{ is a \emph{nonconstant} global solution of system \cref{syst_trav_front}}
\\
& \text{and}\quad
\sup_{\xi\in\rr}\abs{\phi(\xi)}<+\infty
\quad\text{and}\quad
\phi(\xi)\xrightarrow[\xi\to +\infty]{} m_+
\bigr\}
\,.
\end{aligned}
\]
\end{notation}
If $\phi$ is an element of some set $\Phi_c(m_-,m_+)$, then it follows from system \cref{syst_trav_front} that
\begin{equation}
\label{V_of_u_plus_minus_V_of_u_minus}
V(m_+)-V(m_-) = c \int_{\rr}\phi'(\xi)^2 \, d\xi
\,.
\end{equation}
\subsubsection{Propagating terrace of bistable travelling fronts}
\label{subsubsec:def_prop_terrace}
This \namecref{subsubsec:def_prop_terrace} is devoted to several definitions. Their purpose is to enable a compact formulation of the main result of this paper (\cref{thm:1} below). Some comments on the terminology and related references are given at the end of this \namecref{subsubsec:def_prop_terrace}.
\begin{figure}[!htbp]
	\centering
    \includegraphics[width=.8\textwidth]{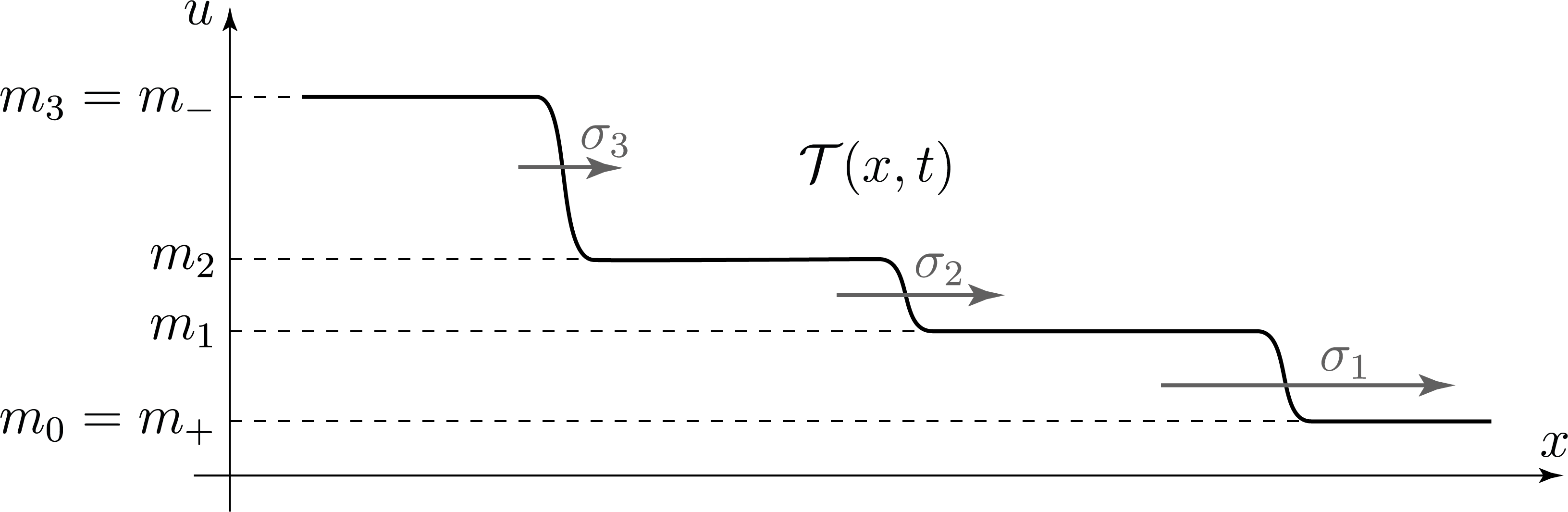}
    \caption{Propagating terrace of (bistable) fronts travelling to the right ($\sigma_i$ denotes the ``physical'' speed corresponding to $c_i$, that is: $\sigma_i = c_i/\sqrt{1+\alpha c_i^2}$).}
    \label{fig:prop_terrace}
\end{figure}
\begin{definition}[propagating terrace of bistable travelling fronts, \cref{fig:prop_terrace}]
Let $m_-$ and $m_+$ be two points of $\mmm$ (satisfying $V(m_-)\le V(m_+)$). A function 
\[
\ttt : \rr\times[0,+\infty)\to\rr^d,\quad (x,t)\mapsto \ttt(x,t)
\]
is called a \emph{propagating terrace of bistable fronts travelling to the right, connecting $m_-$ to $m_+$,} if there exists a nonnegative integer $q$ such that:
\begin{enumerate}
\item if $q$ equals $0$, then $m_-=m_+$ and, for every real quantity $x$ and every nonnegative time $t$, 
\[
\ttt(x,t)=m_-=m_+
\,;
\]
\item if $q$ equals $1$, then there exist
\begin{itemize}
\item a positive quantity $c_1$,
\item and a function $\phi_1$ in $\Phi_{c_1}(m_-,m_+)$ (that is, the profile of a bistable front travelling at parabolic speed $c_1$ and connecting $m_-$ to $m_+$),
\item and a $\ccc^1$-function $t\mapsto x_1(t)$, defined on $[0,+\infty)$, and such that $x_1'(t)$ goes to the quantity $c_1/\sqrt{1+\alpha c_1^2}$ (the corresponding physical speed) as time goes to $+\infty$,
\end{itemize} 
such that, for every real quantity $x$ and every nonnegative time $t$, 
\[
\ttt(x,t)=\phi_1\Bigl[\sqrt{1+\alpha c_1^2}\bigl(x-x_1(t)\bigr)\Bigr]
\,;
\]
\label{item:def_propagating_terrace_q_equals_one}
\item if $q$ is not smaller than $2$, then there exists $q-1$ points $m_1,\dots,m_{q-1}$ in $\mmm$, satisfying (if $m_+$ is denoted by $m_0$ and $m_-$ by $m_q$)
\[
V(m_0)>V(m_1)>\dots>V(m_q)
\,,
\]
and there exist $q$ positive quantities $c_1$, …, $c_q$ satisfying
\[
c_1\ge\dots\ge c_q
\,,
\]
and for each integer $i$ in $\{1,\dots,q\}$, there exist:
\begin{itemize}
\item a function $\phi_i$ in $\Phi_{c_i}(m_i,m_{i-1})$ (that is, the profile of a bistable front travelling at parabolic speed $c_i$ and connecting $m_i$ to $m_{i-1}$),
\item and a $\ccc^1$-function $t\mapsto x_i(t)$, defined on $[0,+\infty)$, and such that $x_i'(t)$ goes to the quantity $c_i/\sqrt{1+\alpha c_i^2}$ (the corresponding physical speed) as time goes to $+\infty$,
\end{itemize}
such that, for every integer $i$ in $\{1,\dots,q-1\}$, 
\[
x_{i+1}(t)-x_i(t)\to +\infty 
\quad\text{as}\quad
t\to +\infty
\,,
\]
and such that, for every real quantity $x$ and every nonnegative time $t$, 
\[
\ttt(x,t) = m_0 + \sum_{i=1}^q \biggl(\phi_i\Bigl[\sqrt{1+\alpha c_i^2}\bigl(x-x_i(t)\bigr)\Bigr]-m_{i-1}\biggr)
\,.
\]
\label{item:def_propagating_terrace_q_larger_than_one}
\end{enumerate}
\end{definition}
\begin{remark}
Item \cref{item:def_propagating_terrace_q_equals_one} may have been omitted in this definition, since it boils down to item \cref{item:def_propagating_terrace_q_larger_than_one} with $q$ equals $1$. 
\end{remark}
A \emph{propagating terrace of bistable fronts travelling to the left} may be defined similarly. 
\subsubsection{Standing terrace of bistable stationary solutions}
\label{subsubsec:def_stand_terrace}
The next three definitions deal with stationary solutions. They are exactly identical to those of \cite{Risler_globalRelaxation_2016,Risler_globalBehaviour_2016}.
\begin{figure}[!htbp]
\centering
\includegraphics[width=\textwidth]{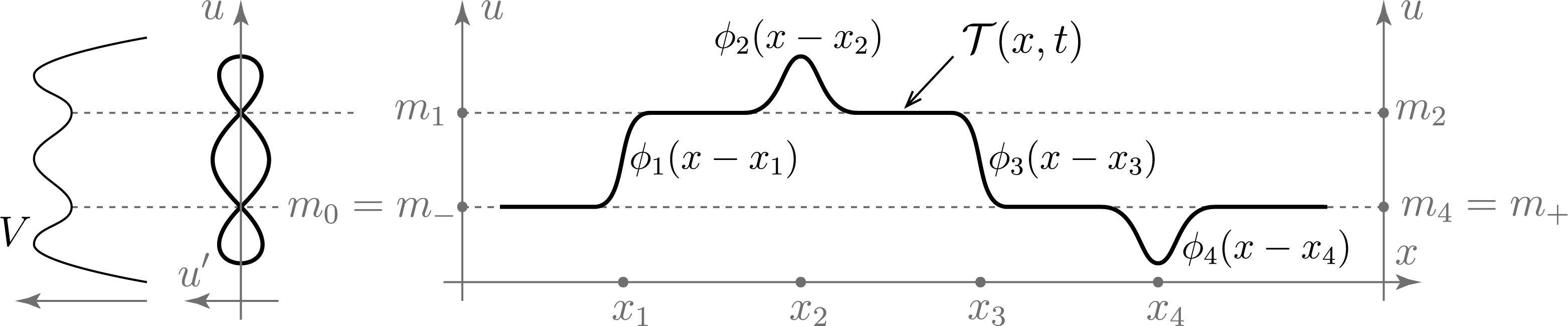}
\caption{Standing terrace (with four items, $q=4$).}
\label{fig:standing_terrace}
\end{figure}
\begin{definition}[standing terrace of bistable stationary solutions, \cref{fig:standing_terrace}]
Let $\valueOfV$ be a real quantity and let $m_-$ and $m_+$ be two points of $\mmm$ such that both quantities $V(m_-)$ and $V(m_+)$ are equal to $\valueOfV$. A function
\[
\ttt : \rr\times[0,+\infty)\to\rr^d,\quad (x,t)\mapsto \ttt(x,t)
\]
is called a \emph{standing terrace of bistable stationary solutions, connecting $m_-$ to $m_+$,} if there exists a nonnegative integer $q$ such that:
\begin{enumerate}
\item if $q$ equals $0$, then $m_-=m_+$ and, for every real quantity $x$ and every nonnegative time $t$, 
\[
\ttt(x,t)=m_-=m_+
\,;
\]
\item if $q=1$, then there exist:
\begin{itemize}
\item a bistable stationary solution $\phi_1$ connecting $m_-$ to $m_+$,
\item and a $\ccc^1$-function $t\mapsto x_1(t)$ defined on $[0,+\infty)$ and satisfying $x_1'(t)\to0$ as time goes to $+\infty$,
\end{itemize}
such that, for every real quantity $x$ and every nonnegative time $t$, 
\[
\ttt(x,t) = \phi_1\bigl( x-x_1(t)\bigr)
\,;
\]
\label{item:def_standing_terrace_q_equals_one}
\item if $q$ is not smaller than $2$, then there exist $q-1$ (not necessarily distinct) points $m_1,\dots,m_{q-1}$ in $\mmm$, all in the level set $V^{-1}(\{\valueOfV\})$, and if $m_-$ is denoted by $m_0$ and $m_+$ by $m_q$, then for each integer $i$ in $\{1,\dots,q\}$, there exist:
\begin{itemize}
\item a bistable stationary solution $\phi_i$ connecting $m_{i-1}$ to $m_i$,
\item and a $\ccc^1$-function $t\mapsto x_i(t)$ defined on $[0,+\infty)$ and satisfying $x_i'(t)\to0$ as time goes to $+\infty$,
\end{itemize}
such that, for every integer $i$ in $\{1,\dots,q-1\}$,
\[
x_{i+1}(t)-x_i(t)\to +\infty 
\quad\text{as}\quad
t\to +\infty
\,,
\]
and such that, for every real quantity $x$ and every nonnegative time $t$, 
\[
\ttt(x,t) = m_0 + \sum_{i=1}^q \Bigl[\phi_i\bigl(x-x_i(t)\bigr)-m_{i-1}\Bigr]
\,.
\]
\label{item:def_standing_terrace_q_larger_than_one}
\end{enumerate}
\end{definition}
\begin{remark}
Once again item \cref{item:def_standing_terrace_q_equals_one} may have been omitted in this definition, since it boils down to item \cref{item:def_standing_terrace_q_larger_than_one} with $q$ equals $1$. 
\end{remark} 
The terminology ``propagating terrace'' was introduced by A. Ducrot, T. Giletti, and H. Matano in \cite{DucrotGiletti_existenceConvergencePropagatingTerrace_2014} (and subsequently used by several other authors \cite{Polacik_propagatingTerracesAsymptOneDimSym_2017,Polacik_propTerracesProofGibbonsConj_2016,GilettiRossi_pulsatingSolMultBistMultiStab_2019,MatanoPolacik_dynNonnegSolOneDimRDII_2020,Polacik_propagatingTerracesDynFrontLikeSolRDEquationsR_2020,GilettiMatano_existenceUniquenessPropTerr_2020,PauthierRademacherU_WeakStrongInteractKinks_2021}) to denote a stacked family (a layer) of travelling fronts in a (scalar) reaction-diffusion equation. This led the author to keep the same terminology in the present context. This terminology is convenient to denote objects that would otherwise require a long description. It is also used in the companion papers \cite{Risler_globalBehaviour_2016,Risler_globalBehaviourRadiallySymmetric_2017}. Additional comments on this terminological choice can be found in \cite{Risler_globalBehaviour_2016}.
\subsubsection{Energy of a bistable stationary solution and of a standing terrace}
\label{subsubsec:def_energy_stat_sol_stand_terrace}
\begin{definition}[energy of a bistable stationary solution]
Let $x\mapsto u(x)$ be a bistable stationary solution connecting two points $m_-$ and $m_+$ of $\mmm$, and let $\valueOfV$ denote the quantity $V(m_+)$ (which is equal to $V(m_-)$). The quantity 
\[
\eee[u] = \int_{\rr}\Bigl(\frac{1}{2}\abs{u'(x)} ^2+V\bigl(u(x)\bigr)- \valueOfV\Bigr)\, dx
\]
is called the \emph{energy of the (bistable) stationary solution $u$}. Observe that this integral converges, since $u(x)$ approaches its limits $m_-$ and $m_+$ at both ends of space at an exponential rate. 
\end{definition}
\begin{definition}[energy of a standing terrace]
Let $\valueOfV$ denote a real quantity and let $\ttt$ denote a standing terrace of bistable stationary solutions connecting two points of $\mmm$ in the level set $V^{-1}(\{\valueOfV\})$. With the notation of the two definitions above, the quantity $\eee[\ttt]$ defined as
\begin{enumerate}
\item if $q$ equals $0$, then $\eee[\ttt]=0$,
\item if $q$ equals $1$, then $\eee[\ttt]=\eee[\phi_1]$,
\item if $q$ is not smaller than $2$, then $\eee[\ttt]=\sum_{i=1}^q\eee[\phi_i]$,
\end{enumerate}
is called the \emph{energy of the standing terrace $\ttt$}. 
\end{definition}
\subsubsection{Bistable asymptotic pattern}
\label{subsubsec:def_asympt_patt}
The next definition is identical to the one of \cite{Risler_globalBehaviour_2016}. 
\begin{figure}[!htbp]
\centering
\includegraphics[width=\textwidth]{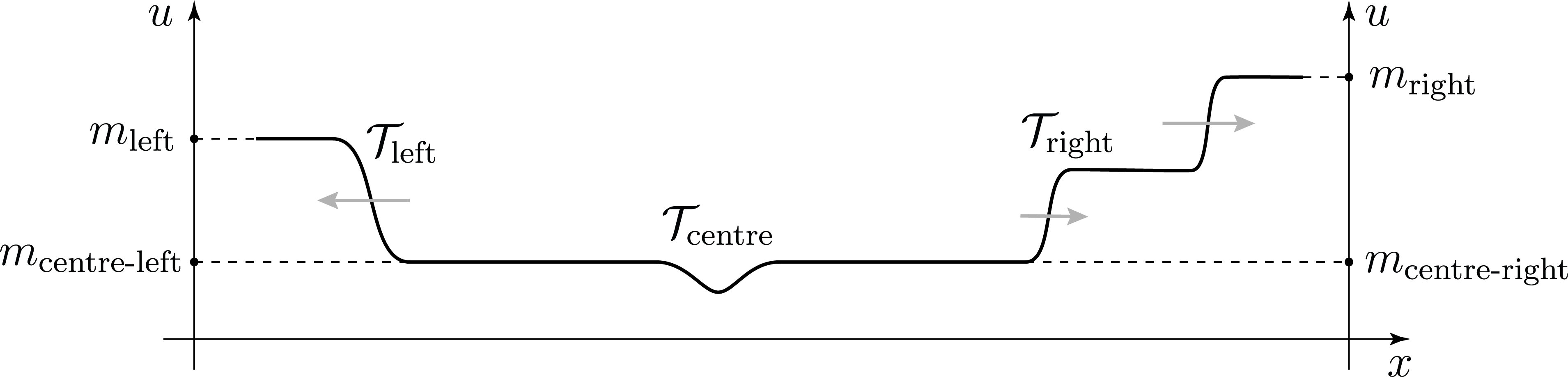}
\caption{Bistable asymptotic pattern.}
\label{fig:bist_asympt_pattern}
\end{figure}
\begin{definition}[bistable asymptotic pattern, \cref{fig:bist_asympt_pattern}]
Let $\mLeft$ and $\mRight$ be two points of $\mmm$. A function
\[
\ppp : \rr\times[0,+\infty)\to\rr^d,\quad (x,t)\mapsto \ppp(x,t)
\]
is called a \emph{bistable asymptotic pattern connecting $\mLeft$ to $\mRight$} if there exist:
\begin{itemize}
\item two points $\mLeftBehind$ and $\mRightBehind$ in $\mmm$, belonging to the same level set of $V$,
\item and a propagating terrace $\tttLeft$ of bistable fronts travelling to the left, connecting $\mLeft$ to $\mLeftBehind$, 
\item and a standing terrace $\tttCentre$ of bistable stationary solutions, connecting $\mLeftBehind$ to $\mRightBehind$, 
\item and a propagating terrace $\tttRight$ of bistable fronts travelling to the right, connecting $\mRightBehind$ to $\mRight$, 
\end{itemize} 
such that, for every real quantity $x$ and for every nonnegative time $t$, 
\[
\ppp(x,t) = \bigl[ \tttLeft(x,t) - \mLeftBehind \bigr] + \tttCentre(x,t) + \bigl[ \tttRight(x,t) - \mRightBehind \bigr]
\,.
\]
\end{definition}
\subsection{Generic hypotheses on the potential}
\label{subsec:generic_assupmt_pot}
\subsubsection{Escape distance}
\label{subsubsec:Escape_dist}
\begin{notation}
For every $u$ in $\rr^d$, let $\sigma\bigl(D^2V(u)\bigr)$ denote the spectrum (the set of eigenvalues) of the Hessian matrix of $V$ at $u$, and let $\eigVmin(u)$ denote the minimum of this spectrum:
\begin{equation}
\label{def_eigVmin_of_u}
\eigVmin(u) = \min \Bigl(\sigma\bigl(D^2V(u)\bigr)\Bigr)
\,.
\end{equation}
\end{notation}
\begin{definition}[Escape distance of a nondegenerate minimum point]
For every $m$ in $\mmm$, let us call \emph{Escape distance of $m$}, and let us denote by $\dEsc(m)$, the supremum of the set
\begin{equation}
\label{set_for_definition_Escape_distance}
\Bigl\{\delta \in[0,1]: \text{ for all } u \text{ in } \rr^d \text{ satisfying } \abs{u-m}\le \delta, \quad\eigVmin(u) \ge\frac{1}{2} \eigVmin(m) \Bigr\}
\,.
\end{equation}
\end{definition}
Since the quantity $\eigVmin(u)$ varies continuously with $u$, this Escape distance $\dEsc(m)$ is positive (thus in $(0,1]$). In addition, for all $u$ in $\rr^d$ such that $\abs{u-m}$ is not larger than $\dEsc(m)$, the following inequality holds:
\begin{equation}
\label{property_dEsc}
\eigVmin(u) \ge\frac{1}{2} \eigVmin(m)
\,.
\end{equation}
\subsubsection{Breakup of space translation invariance for stationary solutions and travelling fronts}
%
For every real quantity $c$, for every ordered pair $(m_-,m_+)$ of points of $\mmm$, and for every function $\phi$ in $\Phi_c(m_-,m_+)$, 
\[
\sup_{\xi \in\rr}\abs{\phi(\xi)-m_-}>\dEsc(m_-)
\quad\text{and}\quad
\sup_{\xi \in\rr}\abs{\phi(\xi)-m_+}>\dEsc(m_+)
\]
(assertion \cref{item:escape_spatial_asymptotics_tw} of \vref{lem:asympt_behav_tw_2}). See \cref{fig:esc_distance}. 
\begin{figure}[!htbp]
\centering
\includegraphics[width=0.6\textwidth]{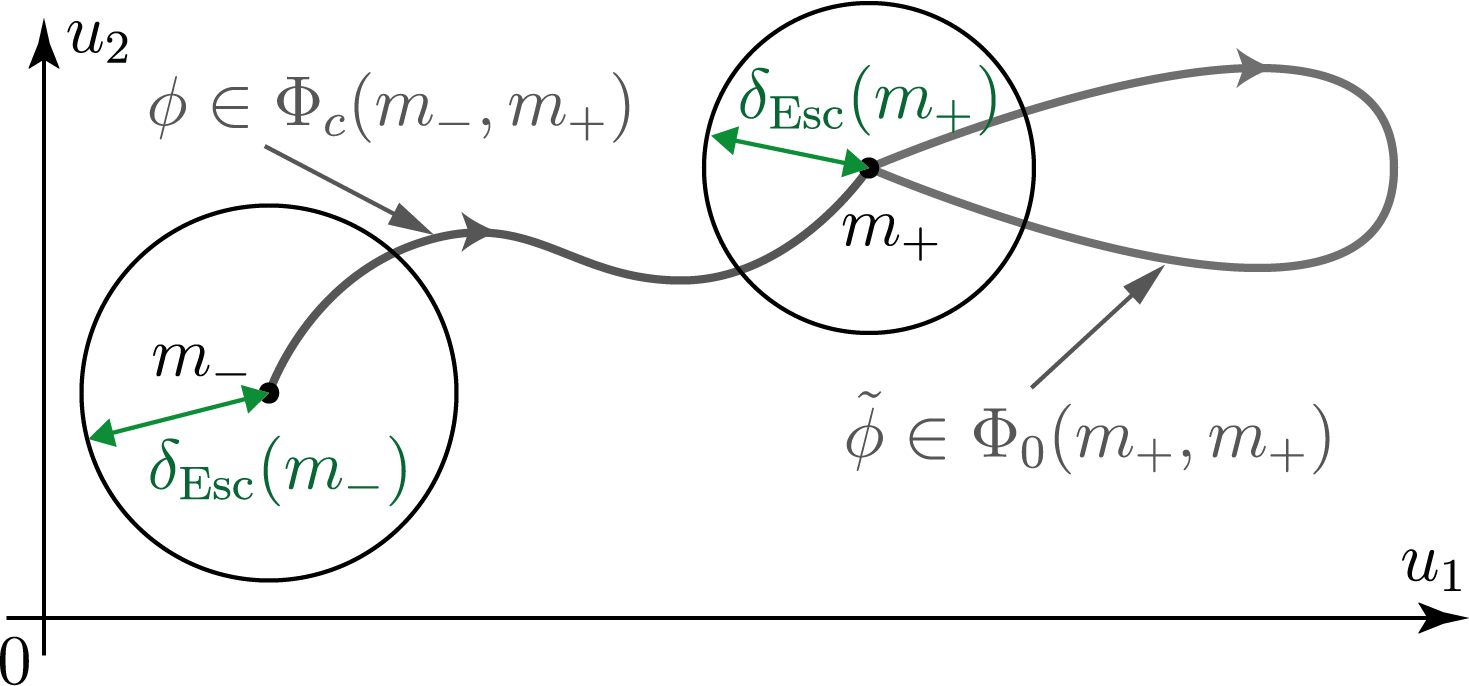}
\caption{Every function in $\Phi_{c}(m_-,m_+)$ escapes at least at distance $\dEsc(m_-)$ of $m_-$ and at distance $\dEsc(m_+)$ of $m_+$; every function in $\Phi_0(m_+,m_+)$ escapes at least at distance $\dEsc(m_+)$ of $m_+$.}
\label{fig:esc_distance}
\end{figure}
Thus, for $c$ in $\rr$ and $(m_-, m_+)$ in $\mmm^2$, let us introduce the set of \emph{normalized profiles of bistable fronts travelling at the parabolic speed $c$/stationary solutions connecting $m_-$ to $m_+$}, defined as
\begin{figure}[!htbp]
\centering
\includegraphics[width=0.75\textwidth]{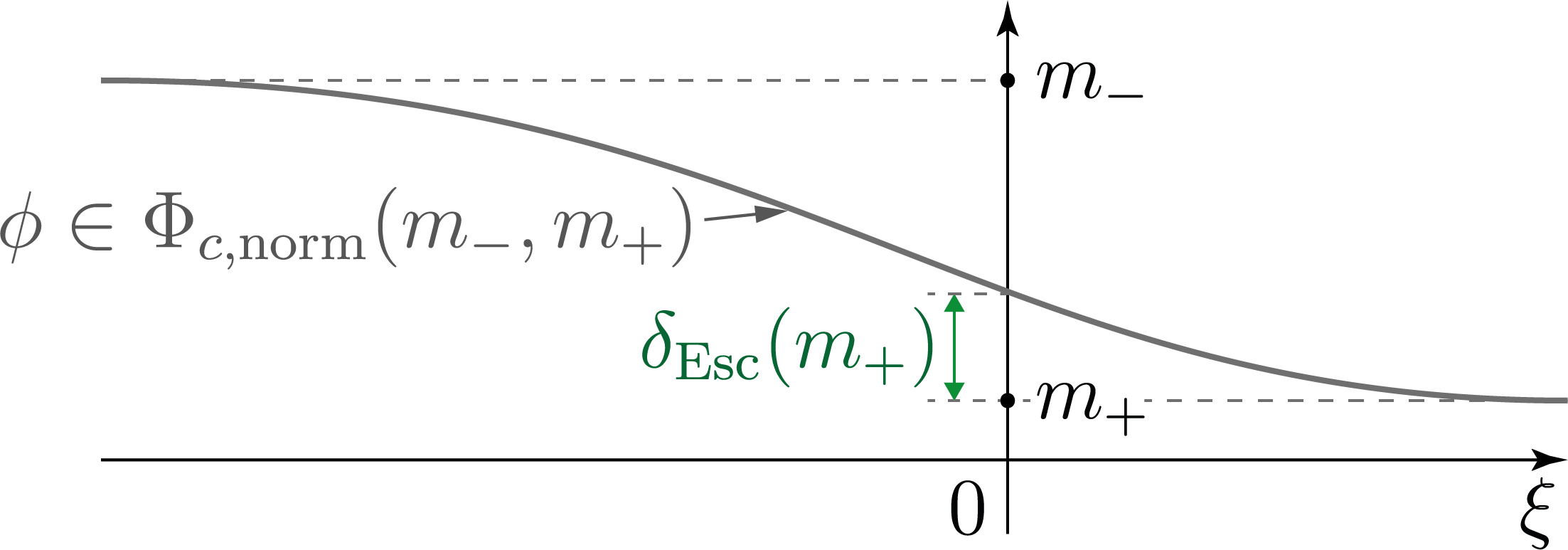}
\caption{Normalized (standing or travelling) bistable front.}
\label{fig:norm_stat}
\end{figure}
\begin{equation}
\label{def_norm}
\begin{aligned}
\PhicNorm{c}(m_-,m_+) = \bigl\{ & \phi\in\Phi_{c}(m_-,m_+): \abs{\phi(0)-m_+} =\dEsc(m_+) \\
& 
\text{and}\quad
\abs{\phi(\xi)-m_+} <\dEsc(m_+)\quad\text{for all}\quad \xi > 0\bigr\} 
\,,
\end{aligned}
\end{equation}
 see \cref{fig:norm_stat}. And if $c$ is positive, let us introduce the set of \emph{normalized profiles of bounded waves travelling at the parabolic speed $c$ and ``invading'' $m_+$}, defined as
\[
\begin{aligned}
\PhicNorm{c}(m_+) = \bigl\{ & \phi\in\Phi_{c}(m_+): \abs{\phi(0)-m_+} =\dEsc(m_+)\quad \text{and}\\
& 
\abs{\phi(\xi)-m_+} <\dEsc(m_+)\quad\text{for all}\quad \xi \text{ in } (0,+\infty)\bigr\} 
\,. 
\end{aligned}
\]
\subsubsection{Statement of the generic hypotheses}
\label{subsubsec:gen_hyp_V}
The main result of this paper (\cref{thm:1} below) requires additional generic hypotheses on the potential $V$, that will now be stated. A formal proof of the genericity (with respect to the potential $V$) of these hypotheses is provided in \cite{JolyRisler_genericTransversalityTravStandFrontsPulses_2023}.
\begin{description}
\item[\hypOnlyBistLabel]\hypertarget{hypOnlyBist} For every $m_+$ in $\mmm$ and every positive quantity $c$, 
\[
\begin{aligned}
\Phi_c(m_+) &= \bigcup_{m_-\in\mmm} \Phi_c(m_-,m_+) \,,\\
\text{or equivalently}\quad
\PhicNorm{c}(m_+) &= \bigcup_{m_-\in\mmm} \PhicNorm{c}(m_-,m_+)
\,.
\end{aligned}
\]
\end{description}
In the next two hypotheses, the subscript ``disc'' refers to the concept of ``discontinuity'' or ``discreteness''. 
\begin{description}
\item[\hypDiscVelLabel]\hypertarget{hypDiscVel} For every $m_+$ in $\mmm$, the set
\[
\bigl\{ c\text{ in }(0,+\infty) : \Phi_c(m_+)\not=\emptyset \bigr\} 
\]
has an empty interior. 
\item[\hypDiscFrontLabel]\hypertarget{hypDiscFront} For every point $m_+$ in $\mmm$ and every real quantity $c$, the set
\[
\bigl\{ \bigl(\phi(0),\phi'(0)\bigr) : \phi\in\PhicNorm{c}(m_+) \bigr\}
\]
is totally discontinuous --- if not empty --- in $\rr^{2d}$. That is, its connected components are singletons. Equivalently, the set $\PhicNorm{c}(m_+)$ is totally disconnected for the topology of compact convergence (uniform convergence on compact subsets of $\rr$).
\end{description}
The next hypothesis will be required to ensure that the number of travelling fronts involved in the asymptotic behaviour of a bistable solution is finite.
\begin{description}
\item[\hypCriticalValuesLabel]\hypertarget{hypCriticalValues} The set of \emph{critical values} of $V$, that is the set
\[
\bigl\{V(u) : u\in\rr^d\text{ and }\nabla V(u)=0 \bigr\} 
\,,
\]
is finite. 
\end{description}
The next hypothesis will be used (as in \cite{Risler_globalRelaxation_2016,Risler_globalBehaviour_2016}) to describe the relaxation of the solution between the propagating terraces of bistable travelling fronts. 
\begin{description}
\item[\hypOnlyMinLabel]\hypertarget{hypOnlyMin} Every critical point of $V$ that belongs to the same level set as a point of $\mmm$ is itself in $\mmm$. 
\end{description}
In other words, for all points $u_1$ and $u_2$ in $\rr^d$,
\[
\Bigl[
\nabla V (u_1) = \nabla V (u_2) = 0 
\ \text{and}\ 
V(u_1) = V(u_2)
\ \text{and}\ 
D^2V(u_1)>0
\Bigr]
\implies
D^2V(u_2)>0
\,.
\]
Finally, let us call \cref{hyp_gen} the union of these five generic hypotheses:
\begin{gather}
\tag{G}
\text{\textup{(\hyperlink{hypOnlyBist}{\hypOnlyBistRef})} and \textup{(\hyperlink{hypDiscVel}{\hypDiscVelRef})} and \textup{(\hyperlink{hypDiscFront}{\hypDiscFrontRef})} and \textup{(\hyperlink{hypCriticalValues}{\hypCriticalValuesRef})} and \textup{(\hyperlink{hypOnlyMin}{\hypOnlyMinRef})}}.
\label{hyp_gen}
\end{gather}
\subsection{Main results}
\begin{theorem}[global asymptotic behaviour]
\label{thm:1}
Let $V$ denote a function in $\ccc^2(\rr^d,\rr)$ satisfying the coercivity hypothesis \cref{hyp_coerc} and the generic hypotheses \cref{hyp_gen}. Then, for every bistable solution $(x,t)\mapsto u(x,t)$ of system \cref{hyp_syst}, there exists a bistable asymptotic pattern $\ppp$ such that
\[
\sup_{x\in\rr}\abs{u(x,t)-\ppp(x,t)}\to 0
\quad\text{as}\quad
t\to + \infty
\,.
\]
\end{theorem}
In this statement the convergence towards the asymptotic pattern is expressed with a uniform norm, but it follows from the proof that the same limit holds for the uniformly local $\Honeul \times \Ltwoul$-norm. Here is an additional conclusion to this theorem. 
\begin{proposition}[residual asymptotic energy]
\label{prop:resid_asympt_energy}
Assume that the assumptions of \cref{thm:1} hold. With the notation of this theorem, if $\tttCentre$ denotes the standing terrace involved in $\ppp$ and if $\valueOfVcentre$ denotes the value taken by $V$ at each of the two points of $\mmm$ connected by $\tttCentre$, then, for every small enough positive quantity $\varepsilon$, 
\[
\int_{-\varepsilon t}^{\varepsilon t} \Bigl( \frac{\alpha}{2}u_t(x,t)^2 + \frac{1}{2}u_x(x,t)^2 + V\bigl(u(x,t)\bigr) - \valueOfVcentre\Bigr) \, dx \to \eee[\tttCentre]
\quad\text{as}\quad
t\to+\infty
\,.
\]
\end{proposition}
These statements are identical to \cite[\GlobalBehaviourThmMain{} and \GlobalBehaviourPropResAsymptEn]{Risler_globalBehaviour_2016} (which are concerned with the parabolic case).
\subsection{Additional questions}
Let us briefly mention some questions that are naturally raised by this result; analogous questions were already discussed in \cite{Risler_globalRelaxation_2016,Risler_globalBehaviour_2016}, where additional comments can be found.
\begin{itemize}
\item Does the correspondence between a solution and its asymptotic pattern display some form of regularity? (some results and comments on this question can be found, in the parabolic case, in \cite{Risler_globalBehaviour_2016}).
\item Does \cref{thm:1} hold without hypothesis \textup{(\hyperlink{hypDiscVel}{\hypDiscVelRef})}? 
\item Is is possible to provide quantitative estimates on the rate of convergence of a solution towards its asymptotic pattern~?
\end{itemize}
\subsection{Organization of the paper}
The organization of this paper closely follows that of the companion paper \cite{Risler_globalBehaviour_2016} where the parabolic case is treated. 
\begin{itemize}
\item The next \cref{sec:preliminaries} is devoted to some preliminaries (existence of solutions, asymptotic compactness, preliminary computations on spatially localized functionals, notation). 
\item The main step in the proof of \cref{thm:1} is \cref{prop:inv_cv} ``invasion implies convergence'' which is proved in \cref{sec:inv_impl_cv} (this \namecref{sec:inv_impl_cv} takes a large part of the paper). This proves the approach towards the terraces of bistable fronts travelling to the left and to the right. 
\item The relaxation behind these terraces of bistable travelling fronts is pursued in \cref{sec:no_inv_implies_relax,sec:convergence}.
\item Finally, combining all these results, the proofs of \cref{thm:1,prop:resid_asympt_energy} are combined together in \cref{sec:proof_thm_1}. 
\item Elementary properties of the profiles of travelling fronts are recalled in \cref{sec:prop_stand_trav}.
\end{itemize}
\section{Preliminaries}
\label{sec:preliminaries}
As everywhere else, let us consider a function $V$ in $\ccc^2(\rr^d,\rr)$ satisfying the coercivity hypothesis \cref{hyp_coerc}. 
\subsection{Global existence of solutions and attracting ball for the flow}
\label{subsec:glob_exist_att_ball}
Let us consider the functional space (uniformly local energy space)
\[
X = \HulofR{1} \times \LtwoulofR
\,,
\]
and, for every $(u,v)$ in $X$, let 
\[
\norm{(u,v)}_X = \sqrt{ \norm{u}_{\HulofR{1}}^2 + \norm{v}_{\LtwoulofR}^2}
\,.
\]
The following proposition is stated and proved in \cite{GallayJoly_globStabDampedWaveBistable_2009} in the case $n=1$ (see Proposition 2.1 of \cite{GallayJoly_globStabDampedWaveBistable_2009}). The proof is identical in the case of systems $n>1$. 
In the statement of this proposition, existence of an attracting ball for the $L^\infty$-norm is redundant; the reason for this redundancy is that the radius $\Rattinfty$ of an attracting ball for the $L^\infty$-norm will be explicitly used in several estimates. 
\begin{proposition}[global existence of solutions and attracting ball]
\label{prop:exist_sol_att_ball}
For every initial condition $(u_0,\tilde{u}_0)$ in $X$, system \cref{hyp_syst} has a unique solution global solution $u$ in the space
\[
\ccc^0\bigl([0,+\infty),\HulofR{1}\bigr) \cap \ccc^1([0,+\infty),\LtwoulofR\bigr)
\]
satisfying $u(0)=u_0$ and $u_t(0)=\tilde{u}_0$. In addition, there exist positive quantities $\RattX$ and $\Rattinfty$ depending only on $V$ and $\alpha$ (radius of attracting balls for the $X$-norm and the $L^\infty$-norm, respectively), such that, for every large enough positive quantity $t$,
\[
\norm{x\mapsto u(x,t)}_{\Linfty}  \le \Rattinfty \
\qquad\text{and}\qquad
\norm{x\mapsto \bigl(u(x,t),u_t(x,t)\bigr)}_{X} \le \RattX
\,. 
\]
\end{proposition}
\subsection{Asymptotic compactness of the solutions}
The following proposition reproduces Proposition 2.3 of \cite{GallayJoly_globStabDampedWaveBistable_2009}. 
\begin{proposition}[asymptotic compactness]
\label{prop:asympt_comp}
For every solution
\[
u\in\ccc^0\bigl([0,+\infty),\HulofR{1}\bigr) \cap \ccc^1([0,+\infty),\LtwoulofR\bigr)
\]
of system \cref{hyp_syst} and for every sequence $\bigl((x_n,t_n)\bigr)_{n\in\nn}$ in $\rr\times[0,+\infty)$ such that $t_n$ goes to $+\infty$ as $n$ goes to $+\infty$, 
there exists a subsequence (still denoted by $\bigl((x_n,t_n)\bigr)_{n\in\nn}$) and there exists an entire solution 
\[
\widebar{u}\in\ccc^0\bigl(\rr,\HulofR{1}\bigr) \cap \ccc^1(\rr,\LtwoulofR\bigr)
\]
of system \cref{hyp_syst} such that, for all positive quantities $L$ and $T$, both quantities
\[
\begin{aligned}
& \sup_{s\in[-T,T]} \norm{y\mapsto u(x_n+y,t_n+s) - \widebar{u}(y,s)}_{H^1([-L,L],\rr^d)} \\
\text{and}\quad
& \sup_{s\in[-T,T]} \norm{u_t(x_n+y,t_n+s) - \widebar{u}_t(y,s)}_{L^2([-L,L],\rr^d)}
\end{aligned}
\]
go to $0$ as $n$ goes to $+\infty$. 
\end{proposition}
\subsection{Time derivative of (localized) energy and \texorpdfstring{$L^2$}{L2}-norm of a solution in a standing or travelling frame}
\label{subsec:1rst_ord}
Let $(x,t)\mapsto u(x,t)$ be a solution of system \cref{hyp_syst}, and let $m$ be a point of $\mmm$. 
\subsubsection{Standing frame}
\label{subsubsec:funct_sf}
As in \cite{GallayJoly_globStabDampedWaveBistable_2009}, taking the scalar product of system \cref{hyp_syst} either with $u_t$ or with $u-m$ and integrating this scalar product with respect to space leads to the following two functionals: the ``energy'' (Lagrangian):
\[
\int_\rr \Bigl(\frac{\alpha}{2}u_t(x,t)^2 + \frac{1}{2}u_x(x,t)^2 + V\bigl(u(x,t)\bigr)-V(m)\Bigr) \, dx
\,,
\]
and the following ``variant of the $L^2$-norm of the distance to $m$'':
\[
\int_\rr \Bigl(\alpha \bigl(u(x,t)-m\bigr) \cdot u_{t}(x,t) + \frac{1}{2}\bigl(u(x,t)-m\bigr)^2 \Bigr) \, dx
\,.
\]
To simplify the presentation, let us assume (only in this \cref{subsec:1rst_ord}) that
\[
m=0_{\rr^d}
\quad\text{and}\quad
V(m) = V(0_{\rr^d}) = 0
\,.
\]
In order to ensure the convergence of such integrals, it is necessary to localize the integrands. Let $x\mapsto\psi(x)$ denote a function in the space $W^{2,1}(\rr,\rr)$ (that is a function belonging to $L^1(\rr)$, together with its first and second derivatives). Then, the time derivatives of these two functionals --- localized by $\psi(x)$ --- read:
\begin{equation}
\label{ddt_loc_en_sf}
\frac{d}{dt} \int_\rr \psi \Bigl( \frac{\alpha}{2}u_t^2 + \frac{1}{2}u_x^2 + V(u)\Bigr) \, dx = 
\int_\rr(-  \psi u_t^2 - \psi' u_x\cdot u_t)\, dx
\,,
\end{equation}
and 
\begin{equation}
\label{ddt_loc_L2_sf}
\begin{aligned}
\frac{d}{dt} \int_\rr \psi \Bigl(\alpha u \cdot u_{t} + \frac{1}{2}u^2 \Bigr) \, dx = 
\int_\rr \Bigl( \psi \bigl(-u \cdot \nabla V(u) - u_x^2 + \alpha u_t^2 \bigr) + \frac{\psi''}{2}u^2\Bigr) \, dx 
\,.
\end{aligned}
\end{equation}
Let us see how these two functionals can be appropriately combined in order to prove, say, the local stability of the homogeneous solution $u\equiv m$ (here $u\equiv 0_{\rr^d}$). The combination must fulfil two properties (provided that the solution is close to $0_{\rr^d}$): coercivity and decrease with time. If the coefficient of the second functional is equal to $1$, then in order to ensure decrease with time, the (positive) coefficient of the first functional must be larger than $\alpha$ (so that the term $+\alpha u_t^2$ in the time derivative of the second functional be properly balanced); assume that this coefficient is equal to $\alpha + \beta$, where $\beta$ is a positive quantity to be chosen appropriately. In short, let us consider the following combination:
\begin{equation}
\label{combination_alpha_plus_beta_energy_plus_Ltwo}
(\alpha+\beta)\times \text{ energy } + L^2 \text{ variant}
\,.
\end{equation}
\begin{itemize}
\item With respect to the local coercivity, using the inequality
\[
\alpha u \cdot u_{t} \ge - \frac{\alpha^2}{2}u_t^2 - \frac{1}{2}u^2
\,,
\]
the combination \cref{combination_alpha_plus_beta_energy_plus_Ltwo} is bounded from below by the integral of an integrand made of $\psi$ times the expression
\[
\frac{\beta\alpha}{2}u_t^2 + \frac{\alpha+\beta}{2} u_x^2 + (\alpha+\beta) V(u)
\,.
\]
\item With respect to the decrease, \emph{neglecting the terms involving the derivatives of $\psi$}, the time derivative of the combination \cref{combination_alpha_plus_beta_energy_plus_Ltwo} reduces to the integral of an integrand made of $\psi$ times the expression
\[
-\beta u_t^2 - u \cdot \nabla V(u) - u_x^2
\,.
\]
\end{itemize}
In view of these two expressions, a reasonable choice is (as is \cite{GallayJoly_globStabDampedWaveBistable_2009}) to choose $\beta=\alpha$, or in other words to introduce the following combined functional:
\begin{equation}
\label{fire_def_sf}
2\alpha\times \text{ energy } + L^2 \text{ variant} = 
\int_\rr \psi \Bigl( \alpha^2 u_t^2 + \alpha u_x^2 + 2\alpha V(u) + \alpha u \cdot u_t + \frac{1}{2}u^2 \Bigr) \, dx
\,.
\end{equation}
\subsubsection{Travelling frame}
\label{subsubsec:energy_L2_trav}
Let $c$ and $\tInit$ and $\xInit$ denote three real quantities (the ``parabolic'' speed, origin of time, and initial origin of space for the travelling frame, see \vref{fig:trav_fr}), with $\tInit$ nonnegative. Usually, besides the \emph{parabolic speed} $c$ in $(0,+\infty)$, it is convenient to define the \emph{physical speed} $\sigma$ in $(0,1/\sqrt{\alpha})$, these two speeds being related by
\[
\sigma = \frac{c}{\sqrt{1+\alpha c^2}}
\iff
c = \frac{\sigma}{\sqrt{1-\alpha \sigma^2}}
\,.
\]
Let us introduce the function $(\xi,s)\mapsto v(\xi,s)$ defined, for every real quantity $\xi$ and nonnegative quantity $s$, as
\[
v(\xi,s) = u(x,t)
\,,
\]
where $(\xi,s)$ and $(x,t)$ are related by
\[
t = \tInit + s 
\quad\text{and}\quad
x = \xInit + \sigma s + \frac{\xi}{\sqrt{1+\alpha c^2}} 
\iff
\xi = \sqrt{1+\alpha c^2} (x - \xInit) - cs
\,.
\]
The evolution system for the function $(\xi,s)\mapsto v(\xi,s)$ reads
\begin{equation}
\label{hyp_syst_tf}
\alpha v_{ss} + v_s - 2\alpha c v_{\xi s} = -\nabla V(v) + c v_\xi + v_{\xi\xi}
\,.
\end{equation}
Let us introduce a function $(\xi,s)\mapsto \psi(\xi,s)$ such that, for every nonnegative quantity $s$, the function $\xi\mapsto \psi(\xi,s)$ belongs to $W^{2,1}(\rr,\rr)$ and its time derivative $\xi\mapsto \psi_s(\xi,s)$ is defined and belongs to $L^1(\rr,\rr)$. As in \cite{GallayJoly_globStabDampedWaveBistable_2009}, the natural analogues for the travelling frame of the two functionals considered above in a standing frame will now be introduced; again, they are obtained by taking the scalar product of system \cref{hyp_syst_tf} either with $v_s$ or with $v$ and integrating this scalar product with respect to space. The time derivatives of the resulting functionals read:
\begin{equation}
\label{ddt_loc_en_tf_first}
\begin{aligned}
&\frac{d}{ds}  \int_\rr \psi\Bigl(\frac{\alpha}{2}v_s^2 + \frac{1}{2}v_\xi^2 + V(v)\Bigr) \, d\xi = \\
&\qquad 
\int_\rr \biggl[ \psi_s \Bigl(\frac{\alpha}{2}v_s^2 + \frac{1}{2}v_\xi^2 + V(v)\Bigr) - (\psi + \alpha c \psi_\xi) v_s^2 + (c\psi - \psi_\xi) v_\xi \cdot v_s \biggr]\, d\xi
\,,
\end{aligned}
\end{equation}
and
\begin{equation}
\label{ddt_loc_L2_tf}
\begin{aligned}
\frac{d}{ds} \int_\rr \psi\Bigl( &\alpha v \cdot v_s  + \frac{1}{2}v^2 - 2\alpha c v \cdot v_\xi \Bigr) \, d\xi = \\
& \int_\rr \Biggl[ \psi_s \Bigl( \alpha v \cdot v_s + \frac{1}{2}v^2 - 2\alpha c v \cdot v_\xi \Bigr) \\
&+ \psi \Bigl( -v \cdot \nabla V(v) - v_\xi^2  + \alpha v_s^2 - 2\alpha c  v_\xi\cdot v_s \Bigr) + \frac{\psi_{\xi\xi}-c\psi_\xi}{2} v^2
\Biggr] \, d\xi
\,.
\end{aligned}
\end{equation}
\begin{remark}
Subtracting and adding the same quantity $\alpha c^2 \psi v_s^2$ to the integrand on the right-hand side of equality \cref{ddt_loc_en_tf_first}, this equality becomes
\begin{equation}
\label{ddt_loc_en_tf_second}
\begin{aligned}
&\frac{d}{ds}  \int_\rr \psi\Bigl(\frac{\alpha}{2}v_s^2 + \frac{1}{2}v_\xi^2 + V(v)\Bigr) \, d\xi = \\
&\quad 
\int_\rr \biggl[ -(1+\alpha c^2) \psi v_s^2 + \psi_s \Bigl(\frac{\alpha}{2}v_s^2 + \frac{1}{2}v_\xi^2 + V(v)\Bigr) + (c\psi - \psi_\xi) (\alpha c v_s^2+v_\xi\cdot v_s) \biggr]\, d\xi
\,,
\end{aligned}
\end{equation}
so that if $\psi(\xi,s)$ is replaced with $e^{c\xi}$, the previous equality reduces (formally) to
\begin{equation}
\label{ddt_formal_en_tf}
\frac{d}{ds} \int_\rr e^{c\xi}\Bigl(\frac{\alpha}{2}v_s^2 + \frac{1}{2}v_\xi^2 + V(v)\Bigr) \, d\xi = -(1+\alpha c^2) \int_\rr e^{c\xi} v_s^2 \, d\xi
\,.
\end{equation}
\end{remark}
\begin{remark}
The second (``$L^2$ variant'') integral (left-hand side of \cref{ddt_loc_L2_tf}) can be rewritten (after an integration by parts, assuming that the function $\psi$ does not vanish) as
\begin{equation}
\label{loc_L2_tf_alternative_expression}
\int_\rr \psi\Bigl( \alpha v \cdot v_s  + \frac{1}{2}v^2 - 2\alpha c v \cdot v_\xi \Bigr) \, d\xi = \int_\rr \psi\Bigl( \alpha v \cdot v_s  + \frac{1}{2}v^2 +\alpha c \frac{\psi_\xi}{\psi} v^2 \Bigr) \, d\xi
\,.
\end{equation}
\end{remark}
Let us assume that
\begin{itemize}
\item $\psi$ varies slowly with time, 
\item and that $\psi$ does not vanish,
\item and that the ratio $\psi_\xi/\psi$ is either small or close to $c$, 
\item and that the function $\psi_{\xi\xi}-c\psi_\xi$ is small,
\end{itemize}
and let us again wonder what would be an appropriate combination of these two functionals (those of \cref{ddt_loc_en_tf_first,ddt_loc_L2_tf}), to recover altogether decrease with time and coercivity where $v$ is small. Once again, if the coefficient of the second functional is equal to $1$, then the coefficient of the first functional must be larger than $\alpha$ (to ensure decrease due to dissipation). Once again, let us write $\alpha + \beta$ for the coefficient of the first functional, or in other words let us consider, again, the combination \cref{combination_alpha_plus_beta_energy_plus_Ltwo}. 
\begin{itemize}
\item With respect to the coercivity, again using the inequality
\[
\alpha v \cdot v_{s} \ge - \frac{\alpha^2}{2}v_s^2 - \frac{1}{2}v^2
\,,
\]
the combination \cref{combination_alpha_plus_beta_energy_plus_Ltwo} (using the expression of the right-hand side of \cref{loc_L2_tf_alternative_expression} for the second functional) is bounded from below by the integral of an integrand made of $\psi$ times the expression
\[
\frac{\alpha\beta}{2} v_s^2 + \frac{\alpha + \beta}{2} v_\xi^2 + (\alpha + \beta) V(v) + \alpha c \frac{\psi_\xi}{\psi} v^2
\,.
\]
\item With respect to the decrease with time, \emph{neglecting terms that are small according to the assumptions on $\psi$}, the time derivative of the combination \cref{combination_alpha_plus_beta_energy_plus_Ltwo} is bounded from above by the integral of an integrand made of $\psi$ times the following expression (using rather expression \cref{ddt_loc_en_tf_first} for the time derivative of the localized energy):
\begin{equation}
\label{ddt_loc_comb_integrand}
\Bigl(-\beta -(\alpha + \beta) \alpha c \frac{\psi_\xi}{\psi} \Bigr) v_s^2 + \Bigl( c(\beta-\alpha) - (\alpha+\beta) \frac{\psi_\xi}{\psi}\Bigr) v_\xi\cdot v_s - v\cdot\nabla V(v) - v_\xi^2
\,.
\end{equation}
\end{itemize}
As in the case of a standing frame, it thus turns out that a reasonable choice is $\beta = \alpha$ (as in \cite{GallayJoly_globStabDampedWaveBistable_2009}), and even that this choice is especially relevant here since it fires one of the terms in the derivative (the term with the factor $\beta-\alpha$). The corresponding combined functional thus reads
\begin{equation}
\label{fire_def_tf}
2\alpha\times \text{ energy } + L^2 \text{ variant} = 
\int_\rr \psi\biggl[\alpha^2 v_s^2 + \alpha v_\xi^2 + 2\alpha V(v) + \alpha v \cdot v_s + \Bigl(\frac{1}{2} +\alpha c \frac{\psi_\xi}{\psi} \Bigr)v^2  \biggr] \, d\xi
\,,
\end{equation}
and expression \cref{ddt_loc_comb_integrand} simplifies into
\[
-\alpha\Bigl( 1+2 \alpha c \frac{\psi_\xi}{\psi} \Bigr) v_s^2 -2\alpha \frac{\psi_\xi}{\psi} v_\xi \cdot v_s - v\cdot\nabla V(v) - v_\xi^2
\,.
\]
If $\psi_\xi/\psi$ is close to zero, this last quantity is roughly equal to
\[
-\alpha v_s^2 - v\cdot\nabla V(v) - v_\xi^2
\,,
\]
and if $\psi_\xi/\psi$ is close to $c$, it is roughly equal to
\begin{equation}
\label{rough_estimate_derivative_firewall}
-\alpha(1+2 \alpha c^2) v_s^2 - 2\alpha c v_\xi \cdot v_s - v\cdot\nabla V(v) - v_\xi^2 
\,,
\end{equation}
and using the inequality
\[
- 2\alpha c v_\xi \cdot v_s \le 2\alpha^2 c^2 v_s^2 + \frac{1}{2}v_\xi^2
\,,
\]
it follows that this last expression \cref{rough_estimate_derivative_firewall} is less than or equal to
\[
-\alpha v_s^2 - v\cdot\nabla V(v) - \frac{1}{2}v_\xi^2
\,;
\]
in both cases this provides the desired decrease with time (provided that $v$ is close to $0_{\rr^d}$).
\subsection{Miscellanea}
\label{subsec:misc}
\subsubsection{Second order estimates for the potential around a minimum point}
\begin{lemma}[second order estimates for the potential around a minimum point]
\label{lem:estim_from_def_escape}
For every $m$ in $\mmm$ and every vector $u$ in $\rr^d$ satisfying $\abs{u-m} \le\dEsc(m)$, the following estimates hold:
\begin{align}
\label{posit_pot_around_loc_min}
V(u)-V(m) &\ge \frac{\eigVmin(m)}{4} (u-m)^2 \,, \\
\text{and}\qquad
\label{v_nablaV_controls_square_around_loc_min}
(u-m)\cdot \nabla V(u) &\ge \frac{\eigVmin(m)}{2} (u-m)^2 \,, \\
\text{and}\qquad
\label{v_nablaV_controls_pot_around_loc_min}
(u-m)\cdot \nabla V(u) &\ge V(u)-V(m)
\,.
\end{align}
\end{lemma}
\begin{proof}
The three inequalities follow from inequality \vref{property_dEsc} ensured by the definition of $\dEsc(m)$ and from three variants of Taylor's theorem with Lagrange remainder applied to the function $f$ defined on $[0,1]$ by $f(\theta)=V\bigl(m+\theta (u-m)\bigr)$ (see \cite[\GlobalRelaxationLemEstimFromEscDist]{Risler_globalRelaxation_2016}).
\end{proof}
\subsubsection{Maximum split between the minimum values of the potential}
\label{subsubsec:max_distance_values_V}
\begin{notation}
Let us introduce the quantity
\begin{equation}
\label{def_Delta_V}
\begin{aligned}
\Delta_V &= \max\{ V(m_1) - V(m_2) : (m_1,m_2)\in\mmm^2 \} \\
&= \max \{V(m):m\in\mmm\} -\min(V)
\,,
\end{aligned}
\end{equation}
where $\min(V)$ is the minimum value of $V(v)$ over all $v$ in $\rr^d$. 
\end{notation}
\section{Invasion implies convergence}
\label{sec:inv_impl_cv}
\subsection{Definitions and hypotheses}
\label{subsec:inv_cv_def_hyp}
As everywhere else, let us consider a function $V$ in $\ccc^2(\rr^d,\rr)$ satisfying the coercivity hypothesis \cref{hyp_coerc}. Let us consider a point $m$ in $\mmm$, an ordered pair (initial condition) $(u_0,\tilde{u}_0)$ in $X$, and the solution $(x,t)\mapsto u(x,t)$ of system \cref{hyp_syst} corresponding to this initial condition. 
Let us make the following hypothesis, illustrated by \cref{fig:inv_cv}. 
\begin{figure}[!htbp]
\centering
\includegraphics[width=.8\textwidth]{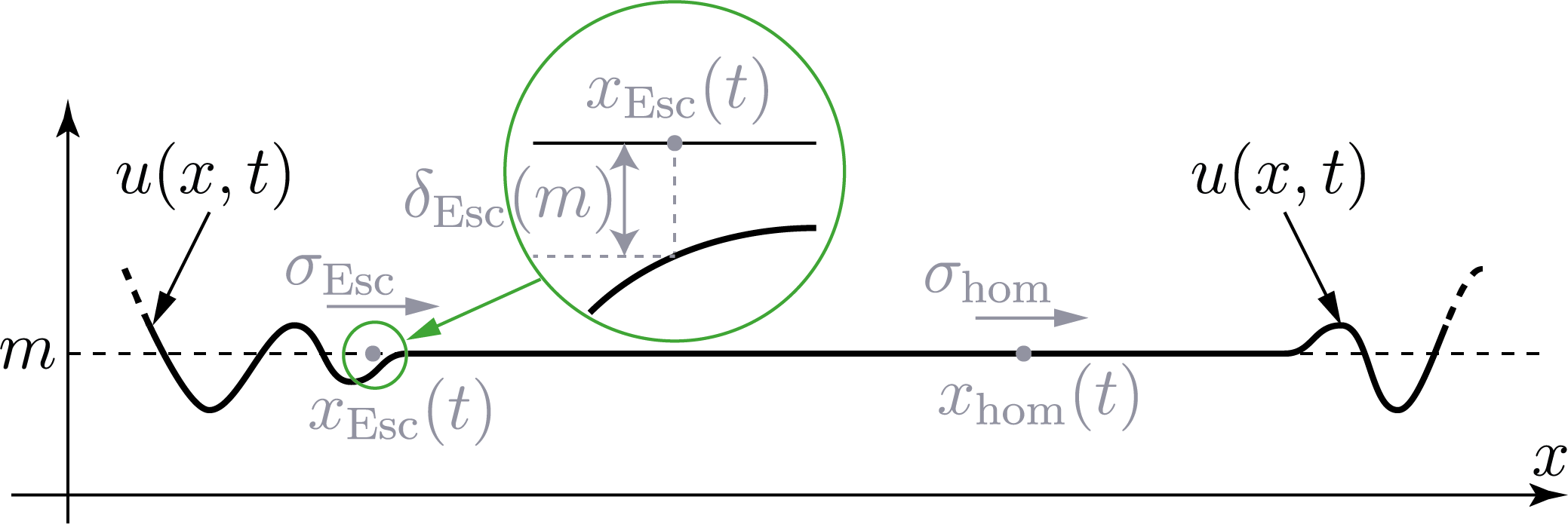}
\caption{Illustration of hypotheses \textup{(\hypHomRightRef)} and \textup{(\hypInvRef)}.}
\label{fig:inv_cv}
\end{figure}
\begin{description}
\item[\hypHomRightLabel]\hypertarget{hypHomRight} There exists a positive quantity $\sHom$ and a $\ccc^1$-function 
\[
\xHom : [0,+\infty)\to \rr\,,
\quad\text{satisfying}\quad
\xHom'(t) \to \sHom
\quad\text{as}\quad
t\to + \infty
\,,
\]
such that, for every positive quantity $L$, the quantity
\[
\norm{y\mapsto \Bigl(u\bigl( \xHom(t) + y, t\bigr) - m, u_t\bigl( \xHom(t) + y, t\bigr)\Bigr) }_{H^1([-L,L],\rr^d)\times L^2([-L,L],\rr^d)} 
\]
goes to $0$ as time goes to $+\infty$. 
\end{description}
For every $t$ in $[0+\infty)$, let us denote by $\xEsc(t)$ the supremum of the set
\[
\Bigl\{ x\in \bigl(-\infty,\xHom(t)\bigr] : \abs{u(x,t) - m} = \dEsc(m) \Bigr\}
\,,
\]
with the convention that $\xEsc(t)$ equals $-\infty$ if this set is empty. In other words, $\xEsc(t)$ is the first point at the left of $\xHom(t)$ where the solution ``Escapes'' at the distance $\dEsc(m)$ from the stable homogeneous equilibrium $m$. This point will be called the ``Escape point'' (with an upper-case ``E'', by contrast with another ``escape point'' that will be introduced later, with a lower-case ``e'' and a slightly different definition). Observe that, if $\xEsc(t)>-\infty$, then
\begin{equation}
\label{dist_to_zero_at_xEsc_of_t}
\abs{u\bigl(\xEsc(t),t\bigr)} = \dEsc(m)
\quad\text{and}\quad
\abs{u(x,t)} < \dEsc(m) \text{ for all } x \text{ in } \bigl( \xEsc(t),\xHom(t) \bigr) 
\,.
\end{equation}
Let us consider the upper limit of the mean speeds between $0$ and $t$ of this Escape point:
\[
\sEsc = \limsup_{t\to +\infty} \frac{\xEsc(t)}{t}
\,,
\]
and let us make the following hypothesis, stating that the area around $\xHom(t)$ where the solution is close to $m$ is ``invaded'' from the left at a nonzero (mean) speed.
\begin{description}
\item[\hypInvLabel]\hypertarget{hypInv} The quantity $\sEsc$ is positive. 
\end{description}
\subsection{Statement}
\label{subsec:inv_cv_stat}
The aim of \cref{sec:inv_impl_cv} is to prove the following proposition (illustrated by \cref{fig:inv_cv_bis}), which is the main step in the proof of \cref{thm:1}. 
\begin{figure}[!htbp]
	\centering
    \includegraphics[width=\textwidth]{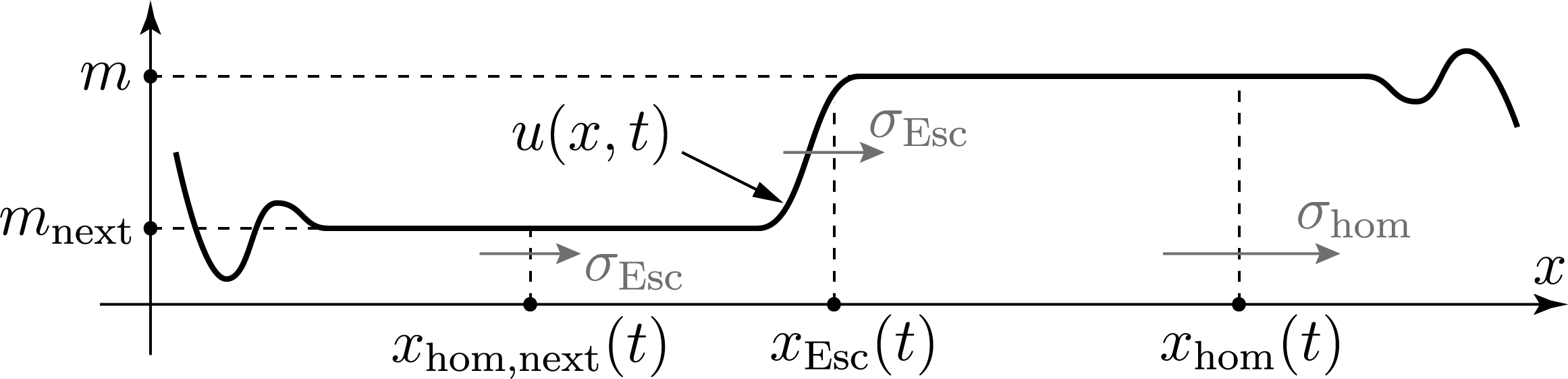}
    \caption{Illustration of \cref{prop:inv_cv}.}
    \label{fig:inv_cv_bis}
\end{figure}
The first assertion of this proposition is that the mean ``physical'' speed $\sEsc$ is smaller than $1/\sqrt{\alpha}$; thus it is legitimate to use the following notation for the ``parabolic'' counterpart of that speed:
\[
\cEsc = \frac{\sEsc}{\sqrt{1-\alpha\sEsc^2}}
\,.
\]
\begin{proposition}[invasion implies convergence]
\label{prop:inv_cv}
Assume that $V$ satisfies the coercivity hypothesis \cref{hyp_coerc} and the generic hypotheses \textup{(\hyperlink{hypOnlyBist}{\hypOnlyBistRef})} and \textup{(\hyperlink{hypDiscVel}{\hypDiscVelRef})} and \textup{(\hyperlink{hypDiscFront}{\hypDiscFrontRef})}, and, keeping the definitions and notation above, let us assume that for the solution under consideration hypotheses \textup{(\hyperlink{hypHomRight}{\hypHomRightRef})} and \textup{(\hyperlink{hypInv}{\hypInvRef})} hold. Then the following conclusions hold. 
\begin{enumerate}
\item The mean speed $\sEsc$ is smaller than $1/\sqrt{\alpha}$.
\label{item:mean_speed_sEsc_smaller_than_sound_speed}
\item There exist:
\begin{itemize}
\item a point $\mNext$ in $\mmm$ satisfying $V(\mNext)<V(m)$,
\item a profile of travelling front $\phi$ in $\PhicNorm{\cEsc}(\mNext,m)$,
\item $\ccc^1$-functions $t\mapsto \xHomNext(t)$ and $t\mapsto \txEsc(t)$ defined on $[0,+\infty)$ and with values in $\rr$,
\end{itemize}
such that, as time goes to $+\infty$, the following limits hold:
\[
\txEsc(t)-\xEsc(t) \to 0
\quad\text{and}\quad
\txEsc'(t) \to \sEsc
\,,
\]
and
\[
\xEsc(t)-\xHomNext(t)\to +\infty
\quad\text{and}\quad
\xHomNext'(t)\to \sEsc
\,,
\]
and
\[
\sup_{x\in[\xHomNext(t) \, ,\ \xHom(t)]} \abs{u(x,t) - \phi\Bigl( \sqrt{1+\alpha\cEsc^2} \bigl(x-\xEsc(t)\bigr) \Bigr)} \to 0 
\,,
\]
and, for every positive quantity $L$, the norm in $H^1([-L,L],\rr^d)\times L^2([-L,L],\rr^d)$ of the function
\[
y\mapsto \Bigl( u\bigl( \xHomNext(t)+y,t\bigr) - \mNext , u_t \bigl( \xHomNext(t)+y,t\bigr) \Bigr) 
\]
goes to $0$.
\label{item:invasion_implies_convergence_main_conclusion}
\end{enumerate}
\end{proposition}
In this statement, the very last conclusion is partly redundant with the previous one. The reason why this last conclusion is stated this way is that it emphasizes the fact that a property similar to \textup{(\hyperlink{hypHomRight}{\hypHomRightRef})} is recovered ``behind'' the travelling front. As can be expected this will be used to prove \cref{thm:1} by re-applying \cref{prop:inv_cv} as many times as required (to the left and to the right), as long as ``invasion of the equilibria behind the last front'' occurs. 
\subsection{Set-up for the proof, 1}
\label{subsec:inv_cv_set_pf}
\subsubsection{Assumptions holding up to changing the origin of time}
Let us keep the notation and assumptions of \cref{subsec:inv_cv_def_hyp}, and let us assume that the hypotheses \cref{hyp_coerc} and \textup{(\hyperlink{hypOnlyBist}{\hypOnlyBistRef})} and \textup{(\hyperlink{hypDiscVel}{\hypDiscVelRef})} and \textup{(\hyperlink{hypDiscFront}{\hypDiscFrontRef})} and \textup{(\hyperlink{hypHomRight}{\hypHomRightRef})} and \textup{(\hyperlink{hypInv}{\hypInvRef})} of \cref{prop:inv_cv} hold. 
\begin{itemize}
\item According to \vref{prop:exist_sol_att_ball}, it may be assumed (without loss of generality, up to changing the origin of time) that, for all $t$ in $[0,+\infty)$, 
\end{itemize}
\begin{align}
\label{hyp_attr_ball_Linfty}
\norm{x\mapsto u(x,t)}_{\Linfty} &\le \Rattinfty \\
\text{and}\qquad
\label{hyp_attr_ball_X}
\norm{x\mapsto \bigl(u(x,t),u_t(x,t)\bigr)}_X &\le \RattX
\,.
\end{align}
\begin{itemize}
\item According to \textup{(\hyperlink{hypHomRight}{\hypHomRightRef})}, it may be assumed (without loss of generality, up to changing the origin of time) that, for all $t$ in $[0,+\infty)$, 
\end{itemize}
\begin{equation}
\label{hyp_xHom_prime_pos}
\xHom'(t)\ge 0
\,.
\end{equation}
\subsubsection{Normalized potential and corresponding solution}
For notational convenience, let us introduce:
\begin{itemize}
\item a new ``normalized'' potential $V^\dag:\rr^d\to\rr$, $v\mapsto V^\dag(v)$,
\item and the corresponding solution $u^\dag:\rr\times[0,+\infty)\to\rr$, $(x,t)\mapsto u^\dag(x,t)$, 
\end{itemize}
defined as
\[
V^\dag(v)=V(m +v)-V(m)
\quad\text{and}\quad
u^\dag(x,t) = u(x,t)-m
\,.
\]
Thus the origin $0_{\rr^d}$ of $\rr^d$ is to $V^\dag$ what $m$ is to $V$, it is a nondegenerate minimum point for $V^\dag$ (with $V^\dag(0_{\rr^d})=0$), and $u^\dag$ is a solution of system \cref{hyp_syst} with potential $V^\dag$ instead of $V$; and, for all $(x,t)$ in $\rr\times[0,+\infty)$, 
\[
V^\dag\bigl(u^\dag(x,t)\bigr) = V\bigl(u(x,t)\bigr)-V(m)
\,.
\]
It follows from inequalities \cref{posit_pot_around_loc_min,v_nablaV_controls_square_around_loc_min,v_nablaV_controls_pot_around_loc_min} that, for all $v$ in $\rr^d$ satisfying $\abs{v}\le\dEsc(m)$, 
\begin{align}
\label{posit_pot_around_loc_min_dag}
V^\dag(v) &\ge \frac{\eigVmin(m)}{4} v^2 \,, \\
\text{and}\qquad
\label{v_nablaV_controls_square_around_loc_min_dag}
v\cdot \nabla V^\dag(v) &\ge \frac{\eigVmin(m)}{2} v^2 \,, \\
\text{and}\qquad
\label{v_nablaV_controls_pot_around_loc_min_dag}
v\cdot \nabla V^\dag(v) &\ge V^\dag(v)
\,,
\end{align}
and it follows from definition \cref{def_Delta_V} of $\Delta_V$ that
\begin{equation}
\label{Delta_V_for_V_dag}
\min_{v\in\rr^d} V^\dag(v) \ge -\Delta_V
\,.
\end{equation}
\subsubsection{Looking for another definition of the escape point}
Unfortunately, the Escape point $\xEsc(t)$ presents a significant drawback: there is no reason why it should display any form of continuity (it may jump back and forth while time increases). This lack of control is problematic with respect to the purpose of writing down a dissipation argument precisely around the position in space where the solution escapes from $m$. 

The answer to this will be to define another ``escape point'' (this one will be denoted by ``$\xesc(t)$'' --- with a lower-case ``e'' --- instead of $\xEsc(t)$). This second definition is a bit more involved than that of  $\xEsc(t)$, but the resulting escape point will have the significant advantage of growing at a finite (and even bounded) rate (\vref{lem:inv}). The material required to define this escape point is introduced in the next \namecref{subsec:firewall_fct_lab_frame}. 
\subsection{Firewall function in the laboratory frame}
\label{subsec:firewall_fct_lab_frame}
\subsubsection{Definition}
\label{subsubsec:firewall_sf_def}
Let 
\begin{equation}
\label{def_kappaZero}
\kappa_0 = \min\Bigl( \frac{1}{4}, \frac{1}{4\alpha}, \frac{\sqrt{\eigVmin(m)}}{4} \Bigr)
\,.
\end{equation}
In this \namecref{subsubsec:firewall_sf_def}, only the following properties of $\kappa_0$ will be used (to derive inequality \cref{dt_fffZero_preliminary} below):
\begin{equation}
\label{conditions_on_kappaZero}
\kappa_0 \le \frac{1}{2} 
\quad\text{and}\quad
\alpha\kappa_0 \le \frac{1}{2} 
\quad\text{and}\quad
\frac{\kappa_0^2}{2} \le \frac{\eigVmin(m)}{8}
\,.
\end{equation}
The slightly more stringent definition \cref{def_kappaZero} of $\kappa_0$ will enable us to reuse this quantity in \cref{sec:no_inv_implies_relax} (see in particular \cref{subsec:relax_sc_stand}). 

Let us introduce the weight function $\psi_0$ defined as
\[
\psi_0(x) = \exp(-\kappa_0\abs{x})
\,.
\]
For $\bar{x}$ in $\rr$, let $T_{\bar{x}}\psi_0$ denote the translate of $\psi_0$ by $\bar{x}$, that is the function defined as
\[
T_{\bar{x}}\psi_0(x) = \psi_0 (x-\bar{x})
\,.
\]
For every real quantity $x$ and nonnegative quantity $t$, following expression \vref{fire_def_sf}, let
\begin{align}
\label{def_EZero}
E_0^\dag(x,t) &= \frac{\alpha}{2}u^\dag_t(x,t)^2 + \frac{1}{2}u^\dag_x(x,t)^2 + V^\dag\bigl(u^\dag(x,t)\bigr)\,,\\
\text{and}\quad
\label{def_fireZero}
F_0^\dag(x,t) &= 2\alpha E_0^\dag(x,t) + \alpha u^\dag(x,t)\cdot u^\dag_t(x,t) + \frac{1}{2}u^\dag(x,t)^2 \\
&= \left(\alpha^2 (u^\dag_t)^2 + \alpha (u^\dag_x)^2 + 2\alpha V^\dag(u^\dag) + \alpha u^\dag\cdot u^\dag_t + \frac{1}{2}(u^\dag)^2\right)(x,t)
\,,
\end{align}
and let us introduce the ``firewall'' function $\fff_0$ defined, for every real quantity $\bar{x}$ and nonnegative quantity $t$, as
\[
\fff_0(\bar{x},t) = \int_\rr T_{\bar{x}}\psi_0(x) F^\dag_0(x,t) \, dx 
\,.
\]
\subsubsection{Upper bound}
\begin{lemma}[firewall upper bound]
\label{lem:upp_bound_fireZero}
For every nonnegative time $t$ and for every real quantity $\bar{x}$,
\begin{equation}
\label{upp_bound_fireZero}
\fff_0(\bar{x},t) \le \int_\rr T_{\bar{x}}\psi_0(x) \Bigl[\frac{3\alpha^2}{2} (u^\dag_t)^2 + \alpha (u^\dag_x)^2 + 2\alpha V^\dag(u^\dag) + (u^\dag)^2 \Bigr]\, dx 
\,.
\end{equation}
\end{lemma}
\begin{proof}
Inequality \cref{upp_bound_fireZero} follows from the definition \cref{def_fireZero} of $F^\dag_0(x,t)$ and from the inequality
\[
\alpha u^\dag\cdot u^\dag_t \le \frac{\alpha^2}{2} (u^\dag_t)^2 + \frac{1}{2}(u^\dag)^2
\,.
\]
\end{proof}
\subsubsection{Linear decrease up to pollution}
For $t$ in $[0,+\infty)$, let us introduce the set 
\[
\SigmaEscZero(t) = \bigl\{x\in\rr: \abs{u^\dag(t)}>\dEsc(m)\bigr\}
\,.
\]
\begin{lemma}[firewall linear decrease up to pollution]
\label{lem:decrease_fire0}
There exist positive quantities $\nuFZero$ and $\KFZero$, both depending only on $\alpha$ and $V$ and $m$, such that for every real quantity $\bar{x}$ and every nonnegative time $t$, 
\begin{equation}
\label{dt_F0_final}
\partial_t \fff_0(\bar{x},t) \le - \nuFZero \, \fff_0(\bar{x},t) + \KFZero \int_{\SigmaEscZero(t)} T_{\bar{x}}\psi_0(x)  \, dx
\,.
\end{equation}
\end{lemma}
\begin{proof}
According to expressions \vref{ddt_loc_en_sf,ddt_loc_L2_sf}, for every real quantity $\bar{x}$ and nonnegative time $t$,
\[
\partial_t \fff_0(\bar{x},t) = \int_\rr \biggl[T_{\bar{x}} \psi_0 \bigl( - \alpha (u^\dag_t)^2 - (u^\dag_x)^2 - u^\dag\cdot\nabla V^\dag(u^\dag)\bigr) - 2\alpha T_{\bar{x}}\psi_0' u^\dag_x\cdot u^\dag_t + \frac{T_{\bar{x}}\psi_0''}{2}(u^\dag)^2\biggr]\, dx 
\,.
\]
Since 
\[
\abs{\psi_0'(\cdot)} = \kappa_0 \psi_0 
\quad\text{and}\quad
\psi_0''(\cdot) \le \kappa_0^2 \psi_0
\]
(indeed $\psi_0''$ equals $\kappa_0^2 \psi_0$ plus a Dirac mass of negative weight), it follows that
\[
\partial_t \fff_0(\bar{x},t)  \le 
\int_\rr T_{\bar{x}}\psi_0 \Bigl[- \alpha (u^\dag_t)^2 - (u^\dag_x)^2 - u^\dag\cdot\nabla V^\dag(u^\dag) + 2\alpha \kappa_0 \abs{u^\dag_x\cdot u^\dag_t} + \frac{\kappa_0^2}{2}(u^\dag)^2\Bigr] \, dx
\,.
\]
Using the inequality
\[
2 \abs{u^\dag_x\cdot u^\dag_t} \le  (u^\dag_x)^2 + (u^\dag_t)^2
\,,
\]
it follows that
\[
\partial_t \fff_0(\bar{x},t)  \le
\int_\rr T_{\bar{x}}\psi_0 \Bigl( \alpha (-1 + \kappa_0) (u^\dag_t)^2 + (-1 + \alpha\kappa_0) (u^\dag_x)^2 - u^\dag\cdot\nabla V^\dag(u^\dag) + \frac{\kappa_0^2}{2} (u^\dag)^2 \Bigr) \, dx 
\,,
\]
and, according to the conditions \cref{conditions_on_kappaZero} on $\kappa_0$, it follows that
\begin{equation}
\label{dt_fffZero_preliminary}
\partial_t \fff_0(\bar{x},t)  \le
\int_\rr T_{\bar{x}}\psi_0 \Bigl( -\frac{\alpha}{2}  (u^\dag_t)^2 -\frac{1}{2} (u^\dag_x)^2 - u^\dag\cdot\nabla V^\dag(u^\dag) + \frac{\eigVmin(m)}{8} (u^\dag)^2\Bigr) \, dx 
\,.
\end{equation}
Let $\nuFZero$ be a positive quantity to be chosen below. It follows from the previous inequality and from the upper bound \cref{upp_bound_fireZero} of \cref{lem:upp_bound_fireZero} that
\begin{equation}
\label{dt_fffZero_preliminary_2}
\begin{aligned}
\partial_t \fff_0(\bar{x},t) + \nuFZero \fff_0(\bar{x},t) &\le \int_\rr T_{\bar{x}}\psi_0 \biggl[\frac{\alpha}{2}(-1+3\alpha\nuFZero)(u^\dag_t)^2 +\Bigl(-\frac{1}{2} + \alpha\nuFZero\Bigr)(u^\dag_x)^2  \\
& - u^\dag\cdot\nabla V^\dag(u^\dag)+ \Bigl(\frac{\eigVmin(m)}{8} + \nuFZero\Bigr)(u^\dag)^2 + 2 \alpha\nuFZero V^\dag(u^\dag) \biggr]\, dx 
\,.
\end{aligned}
\end{equation}
In view of this expression and of inequalities \vref{v_nablaV_controls_square_around_loc_min_dag,v_nablaV_controls_pot_around_loc_min_dag}, let us assume that $\nuFZero$ is small enough so that
\begin{equation}
\label{conditions_on_nuFireZero}
3\alpha\nuFZero\le 1
\quad\text{and}\quad
\alpha\nuFZero\le\frac{1}{2}
\quad\text{and}\quad
\nuFZero\le \frac{\eigVmin(m)}{8}
\quad\text{and}\quad
2\alpha \nuFZero \le \frac{1}{2}
\,;
\end{equation}
the quantity $\nuFZero$ may be chosen as
\begin{equation}
\label{def_nuFireZero}
\nuFZero = \min\Bigl(\frac{1}{4\alpha},\frac{\eigVmin(m)}{8}\Bigr)
\,.
\end{equation}
Then, it follows from \cref{dt_fffZero_preliminary_2,conditions_on_nuFireZero} that
\begin{equation}
\label{dt_fffZero_preliminary_3}
\partial_t \fff_0(\bar{x},t) + \nuFZero \fff_0(\bar{x},t) \le \int_\rr T_{\bar{x}}\psi_0 \Bigl[ - u^\dag\cdot\nabla V^\dag(u^\dag)+ \frac{\eigVmin(m)}{4}(u^\dag)^2 + \frac{1}{2} \abs{V^\dag(u^\dag)} \Bigr]\, dx 
\,.
\end{equation}
According to \cref{v_nablaV_controls_square_around_loc_min_dag,v_nablaV_controls_pot_around_loc_min_dag}, the integrand of the integral at the right-hand side of this inequality is nonpositive as long as $x$ is \emph{not} in $\SigmaEscZero(t)$.
Therefore this inequality still holds if the domain of integration of this integral is changed from $\rr$ to $\SigmaEscZero(t)$. Besides, observe that, in terms of the ``initial'' potential $V$ and solution $u(x,t)$, the factor of $T_{\bar{x}}\psi_0$ under the integral of the right-hand side of this last inequality reads
\[
- (u-m)\cdot \nabla V(u)  + \frac{\eigVmin(m)}{4}(u-m)^2+ \frac{1}{2}\abs{V(u)-V(m)}
\,,
\]
Thus, if $\KFZero$ denotes the maximum of the previous expression over all possible values for $u$, that is, according to the $L^\infty$-bound \vref{hyp_attr_ball_Linfty} on the solution, the (positive) quantity
\begin{equation}
\label{def_KFireZero}
\KFZero = \max_{v\in\rr^d,\ \abs{v}\le\Rattinfty}\Bigl[ - (v-m)\cdot \nabla V(v) + \frac{\eigVmin(m)}{4}(v-m)^2+ \frac{1}{2}\abs{V(v)-V(m)}  \Bigr]
\,,
\end{equation}
then inequality \cref{dt_F0_final} follows from \cref{dt_fffZero_preliminary_3} (with the domain of integration of the integral on the right-hand side restricted to $\SigmaEscZero(t)$). Observe that $\KFZero$ depends only on $\alpha$ and $V$. This finishes the proof of \cref{lem:decrease_fire0}.
\end{proof}
\subsubsection{Coercivity up to pollution}
For every nonnegative time $t$ and for every real quantity $\bar{x}$, let 
\begin{equation}
\label{def_QZero}
\qqq_0(\bar{x},t) = \int_{\rr} T_{\bar{x}}\psi_0(x) \bigl( \alpha u^\dag_t(x,t)^2 + u^\dag_x(x,t)^2 + u^\dag(x,t)^2 \bigr) \, dx
\,.
\end{equation}
The reason for the factor $\alpha$ in front of the term $u^\dag_t(x,t)^2$ in this definition of $\qqq_0(\bar{x},t)$ is that it slightly simplifies the expression of the time derivative of $\qqq_0$ in \vref{lem:bd_dt_Q0}). However dropping this factor $\alpha$ would only induce minor changes. Let
\begin{equation}
\label{def_Sigma_Esc_0}
\SigmaEscZero(t)=\bigl\{x\in\rr: \abs{u^\dag(x,t)} >\dEsc(m)\bigr\} 
\,.
\end{equation}
\begin{lemma}[firewall coercivity up to pollution]
\label{lem:coerc_fire0}
There exist a positive quantity $\epsFZeroCoerc$ and a nonnegative quantity $\KFZeroCoerc$, both depending only on $\alpha$ and $V$, such that for every real quantity $\bar{x}$ and every nonnegative quantity $t$, 
\begin{equation}
\label{coerc_F0}
\fff_0(\bar{x},t) \ge \epsFZeroCoerc \, \qqq_0(\bar{x},t)
- \KFZeroCoerc \int_{\SigmaEscZero(t)} T_{\bar{x}}\psi_0(x)  \, dx
\,. 
\end{equation}
\end{lemma}
\begin{proof}
By polarization, 
\begin{equation}
\label{polarization_inequality}
\alpha u^\dag\cdot u^\dag_t = \Bigl(\alpha\sqrt{\frac{3}{2}}u^\dag_t\Bigr) \cdot \Bigl( \sqrt{\frac{2}{3}} u^\dag \Bigr) \ge - \frac{3}{4} \alpha^2 (u^\dag_t)^2 - \frac{1}{3} (u^\dag)^2 
\,,
\end{equation}
thus for every real quantity $\bar{x}$ and nonnegative quantity $t$, 
\[
\fff_0(\bar{x},t) \ge \int_\rr T_{\bar{x}}\psi_0 \biggl( \frac{1}{4} \alpha^2 (u^\dag_t)^2 + \alpha (u^\dag_x)^2 + 2\alpha V^\dag(u^\dag) + \frac{1}{6} (u^\dag)^2 \biggr) \, dx
\,.
\]
According to inequality \vref{posit_pot_around_loc_min_dag}, the term $2\alpha V^\dag(u^\dag)$ is nonnegative when $x$ is not in the set $\SigmaEscZero(t)$. As a consequence, the previous inequality still holds if the integration domain of this term is reduced to this set. In other words, 
\begin{align}
\fff_0(\bar{x},t) &\ge \int_\rr T_{\bar{x}}\psi_0 \biggl( \frac{1}{4} \alpha^2 (u^\dag_t)^2 + \alpha (u^\dag_x)^2 + \frac{1}{6} (u^\dag)^2 \biggr) \, dx + 2\alpha \int_{\SigmaEscZero(t)} T_{\bar{x}}\psi_0 \ V^\dag(u^\dag) \, dx  \nonumber \\
&\ge \min\Bigl( \frac{\alpha}{4} , \frac{1}{6}\Bigr)\qqq_0(\bar{x},t) + 2\alpha \Bigl(\min_{v\in\rr^d} V^\dag(v)\Bigr) \int_{\SigmaEscZero(t)} T_{\bar{x}}\psi_0 (x)\, dx
\,.
\label{coerc_F0_proof}
\end{align}

Thus, according to inequality \cref{Delta_V_for_V_dag}, introducing the quantities $\epsFZeroCoerc$ and $\KFZeroCoerc$ as
\[
\epsFZeroCoerc = \min\Bigl( \frac{\alpha}{4} , \frac{1}{6}\Bigr)
\quad\text{and}\quad
\KFZeroCoerc = 2\alpha \Delta_V
\,,
\]
inequality \cref{coerc_F0} follows from inequality \cref{coerc_F0_proof}. \Cref{lem:coerc_fire0} is proved. 
\end{proof}
\subsubsection{Elementary inequalities involving \texorpdfstring{$u(\cdot,\cdot)$}{u(.,.)} and \texorpdfstring{$\qqq_0(\cdot,\cdot)$}{Q0(.,.)} and \texorpdfstring{$\fff_0(\cdot,\cdot)$}{F0(.,.)} and \texorpdfstring{$\partial_t\fff_0(\cdot,\cdot)$}{DtF0(.,.)} and \texorpdfstring{$\partial_t\qqq_0(\cdot,\cdot)$}{DtQ0(.,.)}}
The aim of the following definitions and statements is to prove \cref{lem:inv} below, providing a bound on the speed at which a spatial domain where the solution $u$ (respectively $u^\dag$) is close to $m$ (respectively to $0_{\rr^d}$) can be ``invaded''. This lemma involves the two ``hull functions'' $\hullNoescQZero$ and $\hullNoescFZero$ controlling $\fff_0(\cdot,\cdot)$ and $\qqq_0(\cdot,\cdot)$ respectively. The definition of these two hull functions is based on the three quantities $\descQZero(m)$ and $\descFZero(m)$ and $L$ that will be defined now with \cref{lem:inv} as a purpose.
Let
\begin{equation}
\label{def_descQZero}
\descQZero(m) = \sqrt{\frac{2}{1+\kappa_0}}\dEsc(m) 
\,.
\end{equation}
\begin{lemma}[$\qqq_0$ controls $\abs{u^\dag}$]
\label{lem:esc_Esc_qqq0}
For every real quantity $\bar{x}$ and every nonnegative quantity $t$, the following assertion holds
\[
\qqq_0(\bar{x},t) \le \descQZero(m)^2  
\implies
\abs{u^\dag(\bar{x},t)} \le \dEsc(m)
\,.
\]
\end{lemma}
\begin{proof}
Let $v$ denote a function in $\HulofR{1}$. Then,
\[
\begin{aligned}
v(0)^2 &= \psi_0(0) v(0)^2 \\
&\le \frac{1}{2} \int_{\rr} \abs{\frac{d}{dx} \bigl( \psi_0(x) v(x)^2 \bigr)} \, dx \\
&\le \frac{1}{2} \int_{\rr} \bigl( \abs{\psi_0'(x)} v(x)^2 + 2 \psi_0 (x) v(x) \cdot v'(x) \bigr) \, dx \\
&\le \frac{1}{2} \int_{\rr} \psi_0 (x) \bigl( (1+\kappa_0)v(x)^2 + v'(x)^2 \bigr) \, dx \\
&\le \frac{1+\kappa_0}{2} \int_{\rr} \psi_0 (x) \bigl( v(x)^2 + v'(x)^2 \bigr) \, dx 
\,,
\end{aligned}
\]
and the conclusion follows from the definitions \cref{def_descQZero} of $\descQZero(m)$ and \cref{def_QZero} of $\qqq_0(\cdot,\cdot)$.
\end{proof}
Let 
\[
\descFZero(m) = \sqrt{\frac{\epsFZeroCoerc}{8}}\descQZero(m)
\,,
\]
and let $L$ be a positive quantity satisfying the following properties (that will be used below)
\begin{align}
\frac{\KFZeroCoerc}{\epsFZeroCoerc} \frac{2}{\kappa_0} \exp(-\kappa_0 L) & \le \frac{1}{8} \descQZero(m)^2 
\label{property_L_1} \\
\quad\text{and}\quad
\KFZero \frac{2}{\kappa_0} \exp(-\kappa_0 L) &\le \frac{\nuFZero \ \descFZero(m)^2}{4}
\,.
\label{property_L_2}
\end{align}
namely
\[
L = \frac{1}{\kappa_0}\log\biggl[ \max\Bigl(
\frac{16}{\kappa_0} \frac{\KFZeroCoerc}{\epsFZeroCoerc} \frac{1}{\descQZero(m)^2} \ , \ 
\frac{8}{\kappa_0}\frac{\KFZero}{\nuFZero \ \descFZero(m)^2}
\Bigr) \biggr]
\,.
\]
Those requirements on $L$ are related to the fact that
\[
\int_{\rr\setminus[-L,L]} \psi_0(x) \, dx = \frac{2}{\kappa_0} \exp(-\kappa_0 L)
\,.
\]
\begin{lemma}[$\fff_0$ controls $\qqq_0$]
\label{lem:esc_fff0_qqq0}
For every real quantity $\bar{x}$ and every nonnegative quantity $t$,  
\[
\left.
\begin{aligned}
&\fff_0(\bar{x},t)\le \descFZero(m)^2 \\
\text{and, for all } x \text{ in } [\bar{x}-L,\bar{x}+L] \,, \ 
&\abs{u^\dag(x,t)} \le \dEsc(m)
\end{aligned}
\right\}
\implies
\qqq_0(\bar{x},t) \le \frac{1}{4}\descQZero(m)^2 
\,.
\]
\end{lemma}
\begin{proof}
This assertion is an immediate consequence of the coercivity property \cref{coerc_F0} for $\fff_0(\cdot,\cdot)$, the definition of the quantity $\descFZero(m)$ above, and the first property \cref{property_L_1} satisfied by the quantity $L$. 
\end{proof}
\begin{lemma}[$\fff_0$ remains small far from $\SigmaEscZero(t)$]
\label{lem:fff0_decrease}
For every real quantity $\bar{x}$ and every nonnegative quantity $t$,  
\[
\left.
\begin{aligned}
\fff_0(\bar{x},t) &\ge \frac{1}{2}\descFZero(m)^2 \\
\text{and, for every } x \text{ in } [\bar{x}-L,\bar{x}+L] \,, \quad
\abs{u^\dag(x,t)} &\le \dEsc(m)
\end{aligned}
\right\}
\implies
\partial_t \fff_0(\bar{x},t) < 0
\,.
\]
\end{lemma}
\begin{proof}
This assertion is an immediate consequence of the decrease property \cref{dt_F0_final} and the second property \cref{property_L_2} satisfied by the quantity $L$. 
\end{proof}
\begin{lemma}[bound on growth of $\qqq_0$]
\label{lem:bd_dt_Q0}
There exists a positive quantity \\
$\KQZeroGrowth$, depending only on $\alpha$ and $V$, such that, for every real quantity $\bar{x}$ and every nonnegative quantity $t$, 
\[
\partial_t \qqq_0(\bar{x},t) \le \KQZeroGrowth
\,.
\]
\end{lemma}
\begin{proof}
For every real quantity $\bar{x}$ and every nonnegative quantity $t$, 
\[
\begin{aligned}
\partial_t \qqq_0(\bar{x},t) &= 2 \int_\rr \biggl[ T_{\bar{x}} \psi_0  \Bigl( u^\dag_t \cdot \bigl( - u^\dag_t - \nabla V^\dag(u^\dag) \bigr) + u^\dag \cdot u^\dag_t \Bigr) - T_{\bar{x}} \psi_0' \, u^\dag_x \cdot u^\dag_t \biggr] \, dx \\
&= 2 \int_\rr \biggl[ T_{\bar{x}} \psi_0  \Bigl( u_t \cdot \bigl( - u_t - \nabla V(u) \bigr) + (u-m) \cdot u_t \Bigr) - T_{\bar{x}} \psi_0' \, u_x \cdot u_t \biggr] \, dx
\,,
\end{aligned}
\]
thus the conclusion follows from the bounds \vref{hyp_attr_ball_Linfty,hyp_attr_ball_X} for the solution.
\end{proof}
\subsection{Upper bound on the invasion speed}
Let us introduce the following two ``no-escape hull'' functions
\[
\hullNoescQZero:\rr\to\rr\cup \{+\infty\}
\quad\text{and}\quad
\hullNoescFZero:\rr\to\rr\cup \{+\infty\}
\]
defined as
\begin{figure}[!htbp]
	\centering
    \includegraphics[width=\textwidth]{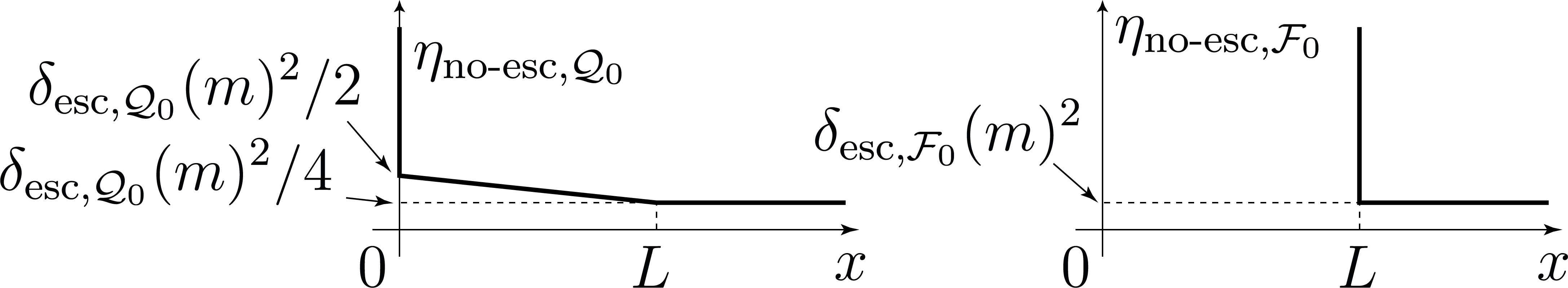}
    \caption{Graphs of the hull functions $\hullNoescQZero$ and $\hullNoescFZero$.}
    \label{fig:graph_hull}
\end{figure}
\[
\hullNoescQZero(x) = 
\left\{
\begin{aligned}
& +\infty & \quad\text{for}\quad & x<0 \,, \\
& \frac{\descQZero(m)^2}{2}\Bigl(1-\frac{x}{2\,L}\Bigr) & \quad\text{for}\quad & 0\le x\le L \,, \\
& \frac{\descQZero(m)^2}{4} & \quad\text{for}\quad & x\ge L \,,
\end{aligned}
\right.
\]
and
\[
\hullNoescFZero(x) = 
\left\{
\begin{aligned}
& +\infty & \quad\text{for}\quad & x<L \,, \\
& \descFZero(m)^2 & \quad\text{for}\quad & x\ge L \,,
\end{aligned}
\right.
\]
see \cref{fig:graph_hull}, and let us introduce the positive quantity $\snoesc$ (``no-escape speed'') defined as
\[
\snoesc = \frac{4 \, L \, \KQZeroGrowth}{\descQZero(m)^2}
\,.
\]
This quantity depends on $\alpha$ and $V$ and $m$ (only). The following lemma is a variant of \cite[\GlobalRelaxationLemBoundInvasionSpeed]{Risler_globalRelaxation_2016}. 
\begin{lemma}[bound on invasion speed]
\label{lem:inv}
For all real quantities $\xLeft$ and $\xRight$ and every nonnegative time $t_0$, if for all $x$ in $\rr$ the following properties holds:
\[
\begin{aligned}
\qqq_0(x,t_0) &\le \max\bigl( \hullNoescQZero(x-\xLeft) , \hullNoescQZero(\xRight - x) \bigr) \\
\text{and}\quad
\fff_0(x,t_0) &\le \max\bigl( \hullNoescFZero(x-\xLeft) , \hullNoescFZero(\xRight - x) \bigr)
\,,
\end{aligned}
\]
then, for every time $t$ greater than or equal to $t_0$ and for all $x$ in $\rr$, the following two inequalities hold
\[
\begin{aligned}
\qqq_0(x,t) &\le \max\Bigl( \hullNoescQZero \bigl(\xLeft-\snoesc\, (t-t_0)\bigr), \hullNoescQZero \bigl(\xRight + \snoesc\, (t-t_0)- x\bigr) \Bigr)\,, \\
\fff_0(x,t) &\le \max\Bigl( \hullNoescFZero \bigl(\xLeft-\snoesc\, (t-t_0)\bigr), \hullNoescFZero \bigl(\xRight + \snoesc\, (t-t_0)- x\bigr) \Bigr)\,.
\end{aligned}
\]
\end{lemma}
\begin{proof}
The proof follows from \cref{lem:esc_Esc_qqq0,lem:esc_fff0_qqq0,lem:fff0_decrease,lem:bd_dt_Q0}. It is almost identical to the proof of \cite[\GlobalRelaxationLemBoundInvasionSpeed]{Risler_globalRelaxation_2016} (see also \cite[\GlobalBehaviourLemBoundInvasionSpeed{} and \GlobalBehaviourFigBoundInvasionSpeed]{Risler_globalBehaviour_2016}). The details are skipped. 
\end{proof}
\subsection{Set-up for the proof, 2: escape point and associated speeds}
\label{subsec:inv_cv_set_pf_cont}
With the notation and results of the previous subsections in hand, let us pursue the set-up for the proof of \cref{prop:inv_cv} ``invasion implies convergence''. According to hypothesis \textup{(\hyperlink{hypHomRight}{\hypHomRightRef})}, it may be assumed, up to changing the origin of time, that, for all $t$ in $[0,+\infty)$ and for all $x$ in $\rr$,
\begin{equation}
\label{hyp_for_def_xesc}
\begin{aligned}
\qqq_0(x,t) & \le \max\biggl( \hullNoescQZero \Bigl( x-\bigl( \xHom(t)-1\bigr) \Bigr) , \hullNoescQZero \bigl( \xHom(t) - x \bigr) \biggr) \\
\text{and}\quad
\fff_0(x,t) & \le \max\biggl( \hullNoescFZero \Bigl( x-\bigl( \xHom(t)-1\bigr) \Bigr) , \hullNoescFZero \bigl( \xHom(t) - x \bigr) \biggr)
\,.
\end{aligned}
\end{equation}
As a consequence, for all $t$ in $[0,+\infty)$, the set
\[
\begin{aligned}
\IHom(t) = \Bigl\{ & x_{\ell} \le \xHom(t) : \text{ for all } x \text{ in } \rr\,, 
\\
& \qqq_0(x,t) \le \max\Bigl( \hullNoescQZero( x-x_{\ell}) , \hullNoescQZero \bigl( \xHom(t) - x \bigr) \Bigr)
\quad\text{and} \\
& \fff_0(x,t) \le \max\Bigl( \hullNoescFZero( x-x_{\ell}) , \hullNoescFZero \bigl( \xHom(t) - x \bigr) \Bigr)
\Bigr\}
\end{aligned}
\]
is a nonempty interval (containing $[\xHom(t)-1,\xHom(t)]$) that must be bounded from below. Indeed, if at a certain time it was not bounded from below --- in other words if it was equal to $(-\infty,\xHom(t)]$ --- then according to \cref{lem:inv} this would remain unchanged in the future, thus according to \cref{lem:esc_Esc_qqq0} the point $\xEsc(t)$ would remain equal to $-\infty$ in the future, a contradiction with hypothesis \textup{(\hyperlink{hypInv}{\hypInvRef})}.

For all $t$ in $[0,+\infty)$, let 
\begin{equation}
\label{def_xesc}
\xesc(t) = \inf \bigl( \IHom(t) \bigr) 
\quad
\text{(thus } \xesc(t)>-\infty \text{).}
\end{equation}
Somehow like $\xEsc(t)$, this point represents the first point at the left of $\xHom(t)$ where the solution $u$ (respectively $u^\dag$) ``escapes'' (in a sense defined by the functions $\qqq_0$ and $\fff_0$ and the no-escape hulls $\hullNoescQZero$ and $\hullNoescFZero$) at a certain distance from $m$ (respectively from $0_{\rr^d}$). In the following, this point $\xesc(t)$ will be called the ``escape point'' (by contrast with the ``Escape point'' $\xEsc(t)$ defined before). According to the first of the ``hull inequalities'' \cref{hyp_for_def_xesc} and \cref{lem:esc_Esc_qqq0} (``$\qqq_0$ controls $u^\dag$''), for all $t$ in $[0,+\infty)$,
\begin{equation}
\label{xEsc_xesc_xHom}
\xEsc(t) \le \xesc(t) \le \xHom(t)-1 
\quad\text{and}\quad
\SigmaEscZero(t) \cap [\xEsc(t),\xHom(t)] = \emptyset
\,,
\end{equation}
and, according to hypothesis \textup{(\hyperlink{hypHomRight}{\hypHomRightRef})}, 
\begin{equation}
\label{xHom_minus_xesc}
\xHom(t) - \xesc(t) \to +\infty
\quad\text{as}\quad
t\to +\infty
\,.
\end{equation}
The big advantage of $\xesc(\cdot)$ with respect to $\xEsc(\cdot)$ is that, according to \cref{lem:inv}, the growth of $\xesc(\cdot)$ is more under control. More precisely, according to this lemma, for all nonnegative quantities $t$ and $s$, 
\begin{equation}
\label{control_escape}
\xesc(t+s)\le \xesc(t) + \snoesc \, s
\,.
\end{equation}
For every $s$ in $[0,+\infty)$, let us consider the ``upper and lower bounds of the variations of $\xesc(\cdot)$ over all time intervals of length $s$'':
\begin{figure}[!htbp]
	\centering
    \includegraphics[width=\textwidth]{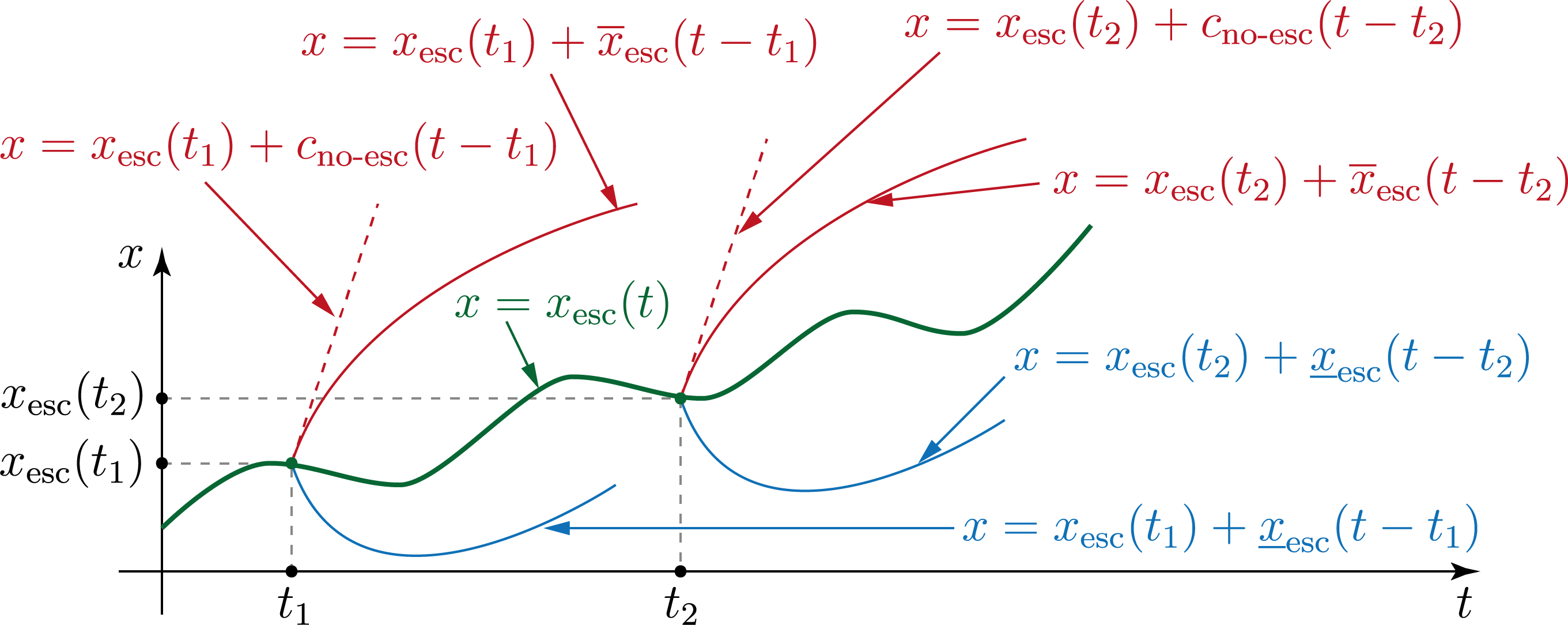}
    \caption{Illustration of the bounds \cref{various_bounds_x_esc}.}
    \label{fig:def_bar_underbar_xesc}
\end{figure}
\[
\barxesc(s) = \sup_{t\in[0,+\infty)} \xesc(t+s) - \xesc(t)
\quad\text{and}\quad
\underbarxesc(s) = \inf_{t\in[0,+\infty)} \xesc(t+s) - \xesc(t)
\,,
\]
see \cref{fig:def_bar_underbar_xesc}.
According to these definitions and to inequality \cref{control_escape} above, for all $t$ and $s$ in $[0,+\infty)$,
\begin{equation}
\label{various_bounds_x_esc}
-\infty\le\underbarxesc(s)\le \xesc(t+s)-\xesc(t)\le \barxesc(s) \le \snoesc\, s
\,.
\end{equation}
Let us consider the four limit mean speeds:
\[
\sescinf = \liminf_{t\to+\infty}\frac{\xesc(t)}{t}
\quad\text{and}\quad
\sescsup = \limsup_{t\to+\infty}\frac{\xesc(t)}{t}
\]
and
\[
\underbarsescinf = \liminf_{s\to+\infty}\frac{\underbarxesc(s)}{s}
\quad\text{and}\quad
\barsescsup = \limsup_{s\to+\infty}\frac{\barxesc(s)}{s}
\,.
\]
The following inequalities follow from these definitions and from hypothesis \textup{(\hyperlink{hypInv}{\hypInvRef})}:
\[
-\infty \le \underbarsescinf \le \sescinf \le \sescsup \le \barsescsup \le \snoesc
\quad\text{and}\quad
0 < \sEsc \le \sescsup
\,.
\]
The four limit mean speeds defined just above will turn out to be equal. The proof of this equality is based on the ``relaxation scheme'' that will be set up in \cref{subsec:relax_sch_tr_fr} below. To this end, an additional estimate on these speeds (namely, the fact that they are smaller than the maximum speed of propagation $1/\sqrt{\alpha}$) is required. This is the purpose of the next \namecref{subsec:further_bd_finite_speed}.
\subsection{Further (subsonic) bound on invasion speed, preparation}
\label{subsec:further_bd_finite_speed}
The next \namecref{subsec:further_bd_finite_speed} will be devoted to the relaxation scheme in a travelling frame that is the core of the proof of \cref{thm:1}. This relaxation scheme will require an upper bound on the parabolic speed of the travelling frame, in other words it will require that the physical speed of the travelling frame be (strictly) subsonic (without this requirement all estimates would literally blow up). The aim of this \namecref{subsec:further_bd_finite_speed} is to define the value of this upper bound (namely the quantity $\cUpp$ defined below). Using the relaxation scheme set up in the next \namecref{subsec:further_bd_finite_speed}, it will be proved later (\cref{lem:further_bd_finite_speed} in \cref{subsubsec:proof_bd_cmax}) that the (upper) limit mean speed $\barsescsup$ is not larger than this (subsonic) bound $\cUpp$.

These observations and statements are very similar to (and much inspired by) those made by Gallay and Joly in \cite{GallayJoly_globStabDampedWaveBistable_2009}. To define the subsonic bound on invasion speed, these authors used a Poincaré inequality in the weighted Sobolev spaces $H^1_c(\rr,\rr^d)$ (see \cite[subsection~4.2]{GallayJoly_globStabDampedWaveBistable_2009}). Although based on the same idea, the definition of $\cUpp$ below is slightly different and suits better the purpose pursued here (that is, the convergence towards a stacked family of travelling fronts). 

Let us recall the quantity $\Delta_V$ defined in \vref{subsubsec:max_distance_values_V} and let us introduce the (positive) quantities
\begin{equation}
\label{def_cmax}
\cUpp = \frac{4 \Delta_V}{\dEsc(m)^2\min\Bigl(\frac{1}{2}, \frac{\eigVmin(m)}{4}\Bigr) } + 1
\quad\text{and}\quad
\Eesc = \frac{1}{4}\dEsc(m)^2\min\Bigl(\frac{1}{2}, \frac{\eigVmin(m)}{4}\Bigr)
\,.
\end{equation}
These two quantities depend on $\alpha$ and $V$ and $m$ (only).
The following lemma provides a justification for this value of $\cUpp$ and will be used in \cref{subsubsec:proof_bd_cmax} to prove \cref{lem:further_bd_finite_speed} stating that the (upper) limit mean speed $\barsescsup$ is not larger than $\cUpp$. Note that the ``$+1$'' in the definition of $\cUpp$ is only to ensure that $\cUpp$ is nonzero (and actually not smaller than $1$), since the quantity $\Delta_V$ may be equal to $0$ (if the set $\mmm$ is reduced to a single point). 
\begin{lemma}[positive energy at Escape point when travelling frame speed is large positive]
\label{lem:posit_en_Esc}
For every function $w$ in $\HulofR{1}$ and every quantities $\xi_0$ and $c$ satisfying the conditions
\[
\abs{w(\xi_0)} = \dEsc(m)
\quad\text{and}\quad
\abs{w(\xi)} \le \dEsc(m) \text{ for all } \xi \text{ in } [\xi_0,\xi_0+1]
\quad\text{and}\quad
c \ge \cUpp
\,,
\]
the following estimate holds:
\begin{equation}
\label{posit_en_Esc}
\int_{-\infty}^{\xi_0+1} e^{c\xi} \Bigl( \frac{1}{2}w'(\xi)^2 + V^\dag\bigl(w(\xi)\bigr) \Bigr) \, d\xi \ge  \Eesc e^{c \xi_0}
\,.
\end{equation}
\end{lemma}
\begin{proof}
Let us introduce a function $w$ in $\HulofR{1}$ and quantities $\xi_0$ and $c$ satisfying the hypotheses above. Then, according to inequality \vref{posit_pot_around_loc_min_dag}, 
\[
\begin{aligned}
\int_{-\infty}^{\xi_0+1} & e^{c\xi} \Bigl( \frac{1}{2}w'(\xi)^2 + V^\dag\bigl(w(\xi)\bigr) \Bigr) \, d\xi \\
& \ge \int_{-\infty}^{\xi_0} e^{c\xi} (-\Delta_V) \, d\xi + \int_{\xi_0}^{\xi_0+ 1} e^{c\xi}\Bigl( \frac{1}{2}w'(\xi)^2 + \frac{\eigVmin(m)}{4} w(\xi)^2 \Bigr) \, d\xi \\
& \ge e^{c \xi_0} \biggl( -\frac{\Delta_V }{c} + \min\Bigl(\frac{1}{2},\frac{\eigVmin(m)}{4}\Bigr) \int_{\xi_0}^{\xi_0+1} \bigl( w'(\xi)^2 + w(\xi)^2 \, d\xi \biggr)
\,.
\end{aligned}
\] 
Let us denote by $\theta$ the affine function taking the value $1$ at $\xi_0$ and $0$ at $\xi_0+1$, namely defined as $\theta (\xi) = \xi_0+1 - \xi$. Then, 
\[
\begin{aligned}
\dEsc(m)^2 & = w(\xi_0)^2 = \theta(\xi_0) w(\xi_0)^2 \\
& = - \int_{\xi_0}^{\xi_0+1} \frac{d}{d\xi}\bigl( \theta(\xi) w(\xi)^2 \bigr) \, d\xi \\
& = - \int_{\xi_0}^{\xi_0+1} \bigl( \theta'(\xi) w(\xi)^2 + 2 \theta(\xi) w(\xi) w'(\xi) \bigr) \, d\xi \\
& \le 2 \int_{\xi_0}^{\xi_0+1} \bigl( w(\xi)^2 + w'(\xi)^2\bigr) \, d\xi
\,.
\end{aligned}
\]
It follows from these two inequalities that
\[
\int_{-\infty}^{\xi_0+1} e^{c\xi} \Bigl( \frac{1}{2}w'(\xi)^2 + V^\dag\bigl(w(\xi)\bigr) \Bigr) \, d\xi \ge e^{c \xi_0} \biggl( -\frac{\Delta_V }{c} + \frac{1}{2} \min\Bigl(\frac{1}{2},\frac{\eigVmin(m)}{4}\Bigr) \dEsc(m)^2 \biggr)
\,,
\]
and in view of the definitions \cref{def_cmax} of $\cUpp$ and $\Eesc$, inequality \cref{posit_en_Esc} follows. \Cref{lem:posit_en_Esc} is proved. 
\end{proof}
\subsection{Relaxation scheme in a travelling frame}
\label{subsec:relax_sch_tr_fr}
The aim of this \namecref{subsec:relax_sch_tr_fr} is to set up an appropriate relaxation scheme in a travelling frame. This means defining an appropriate localized energy and controlling the ``flux'' terms occurring in the time derivative of this localized energy. The considerations made in \vref{subsec:1rst_ord} will be put in practice. 
\subsubsection{Notation for the travelling frame}
\label{subsubsec:def_trav_f}
Let us keep the notation and hypotheses introduced above (since the beginning of \cref{subsec:inv_cv_set_pf}), and let us introduce the following real quantities that will play the role of ``parameters'' for the relaxation scheme below:
\begin{itemize}
\item the ``initial time'' $\tInit$ of the time interval of the relaxation;
\item the initial position $\xInit$ of the origin of the travelling frame;
\item the ``parabolic'' speed $c$ of the travelling frame and its ``physical'' speed $\sigma$, related by
\[
\sigma=\frac{c}{\sqrt{1+\alpha c^2}}
\iff 
c = \frac{\sigma}{\sqrt{1-\alpha \sigma^2}}
\,;
\]
\item a quantity $\initCut$ that will be the the position of the maximum point of the weight function $y\mapsto\chi(y,\tInit)$ localizing energy at initial time $t=\tInit$ (this weight function is defined below). 
\end{itemize}
Let us recall the (positive) quantity $\cUpp$ defined in the previous \namecref{subsubsec:def_trav_f} and let us make on these parameters the following hypotheses:
\begin{equation}
\label{hyp_param_relax_sch}
0\le \tInit
\quad\text{and}\quad
0 < c \le \cUpp
\quad\text{and}\quad
0 \le \initCut
\,.
\end{equation}
The relaxation scheme will be applied several time in the next pages, for various choices of this set of parameters. 

For every real quantity $\xi$ and every nonnegative quantity $s$, let
\[
v(\xi,s) = u^\dag(x,t)
\]
where $(\xi,s)$ and $(x,t)$ are related by
\[
t = \tInit + s 
\quad\text{and}\quad
x = \xInit + \sigma s + \frac{\xi}{\sqrt{1+\alpha c^2}} 
\iff
\xi = \sqrt{1+\alpha c^2} (x - \xInit) - cs
\,,
\]
see \cref{fig:trav_fr}. 
\begin{figure}[!htbp]
	\centering
    \includegraphics[width=\textwidth]{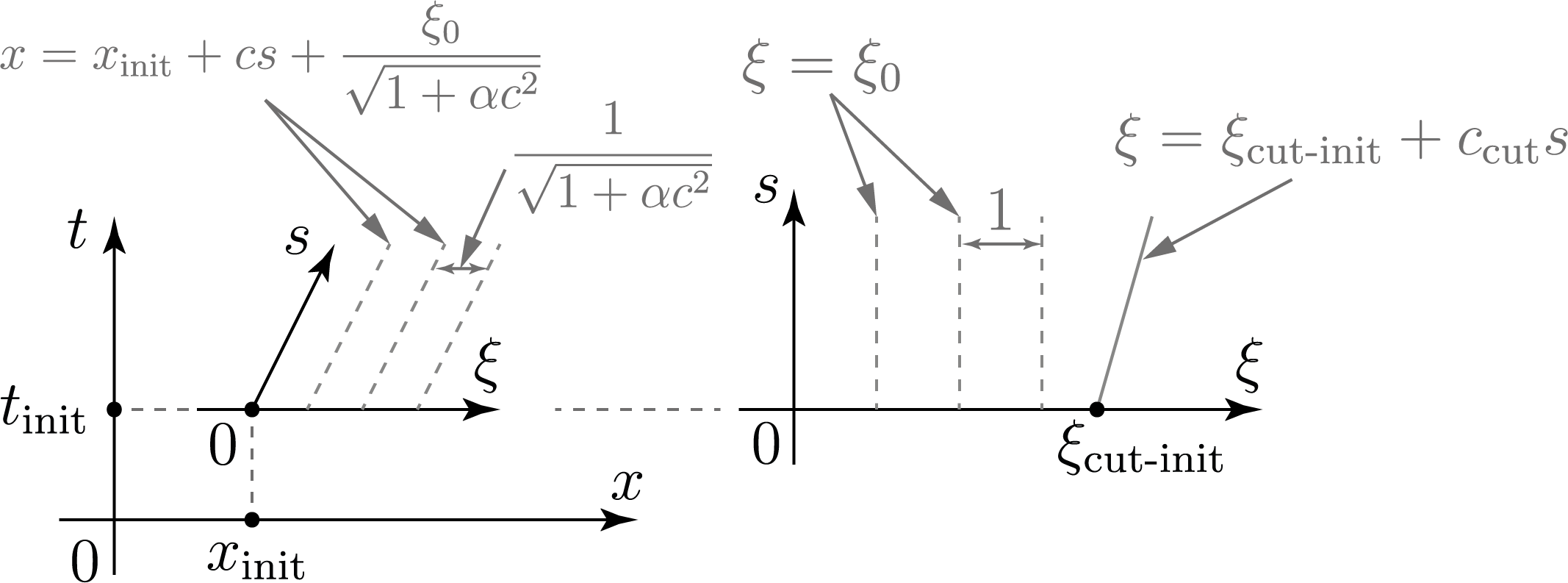}
    \caption{Space coordinate $\xi$ and time coordinate $s$ in the travelling frame, and parameters $\tInit$ and $\xInit$ and $c$ and $\initCut$.}
    \label{fig:trav_fr}
\end{figure}
The system satisfied by $v(\cdot,\cdot)$ reads
\[
\alpha v_{ss} + v_s - 2 \alpha c v_{\xi s} = - \nabla V^\dag(v) + c v_\xi + v_{\xi\xi}
\,.
\]
Let $\kappa$ (rate of decrease of the weight functions) and $\cCut$ (speed of the cutoff point in the travelling frame) be two positive quantities, small enough so that the following conditions be satisfied:
\begin{equation}
\label{condition_kappa_lower_bound_F}
\alpha\kappa c\le\frac{1}{6}
\end{equation}
(this condition will be used in \vref{lem:lower_bound_F}, lower bound on the firewall function) and 
\begin{equation}
\label{condition_kappa_cCut_upper_bound_dsF}
\begin{aligned}
(c+\kappa)\bigl( 2 \alpha\kappa + \cCut ( \alpha + 1/2)  \bigr) \le \frac{1}{2}
\quad&\text{and}\quad
\alpha \cCut(c+\kappa)(c+1)\le \frac{1}{4} \\
\text{and}\quad
\frac{c+\kappa}{2} \Bigl( \kappa + \cCut \bigl( 1 + \alpha(2c+1) \bigr) \Bigr) \le \frac{\eigVmin(m)}{8}
\quad&\text{and}\quad
2\alpha \cCut(c+\kappa)\le \frac{1}{4}
\end{aligned}
\end{equation}
(these conditions will be used to derive the upper bound \vref{condition_kappa_cCut_upper_bound_dsF} on the time derivative of the firewall). These two quantities may be chosen as
\[
\begin{aligned}
\kappa &= \min\Bigl( \frac{1}{16\alpha\cUpp},\frac{1}{\sqrt{\alpha}},\frac{\eigVmin(m)}{16\cUpp},\frac{\sqrt{\eigVmin(m)}}{4}\Bigr) \\
\text{and}\qquad
\cCut &= \frac{1}{\cUpp+\kappa}\min\Bigl(\frac{1}{2(2\alpha+1)},\frac{1}{4\alpha(\cUpp+1)},\frac{1}{8\alpha},\frac{\eigVmin(m)}{8\bigl(1+\alpha(2\cUpp+1)\bigr)}  \Bigr)
\,.
\end{aligned}
\]
\subsubsection{Localized energy}
\label{subsubsec:def_loc_en}
For every real quantity $s$, let us introduce the two intervals
\[
\iMain(s) = ( - \infty , \initCut + \cCut s] 
\quad\text{and}\quad
\iRight(s) = [ \initCut + \cCut s , +\infty) \,,
\]
and let us introduce the function $\chi(\xi,s)$ (weight function for the localized energy) defined as
\[
\chi(\xi,s) =
\left\{
\begin{aligned}
&\exp(c\xi) 
& &\text{if}\quad
\xi \in \iMain(s) \,, \\
&\exp\bigl( (c+ \kappa) (\initCut + \cCut s) -\kappa \xi \bigr) 
& &\text{if}\quad
\xi \in \iRight(s) 
\,,
\end{aligned}
\right.
\]
see \cref{fig:chi_psi}, 
\begin{figure}[!htbp]
	\centering
    \includegraphics[width=0.8\textwidth]{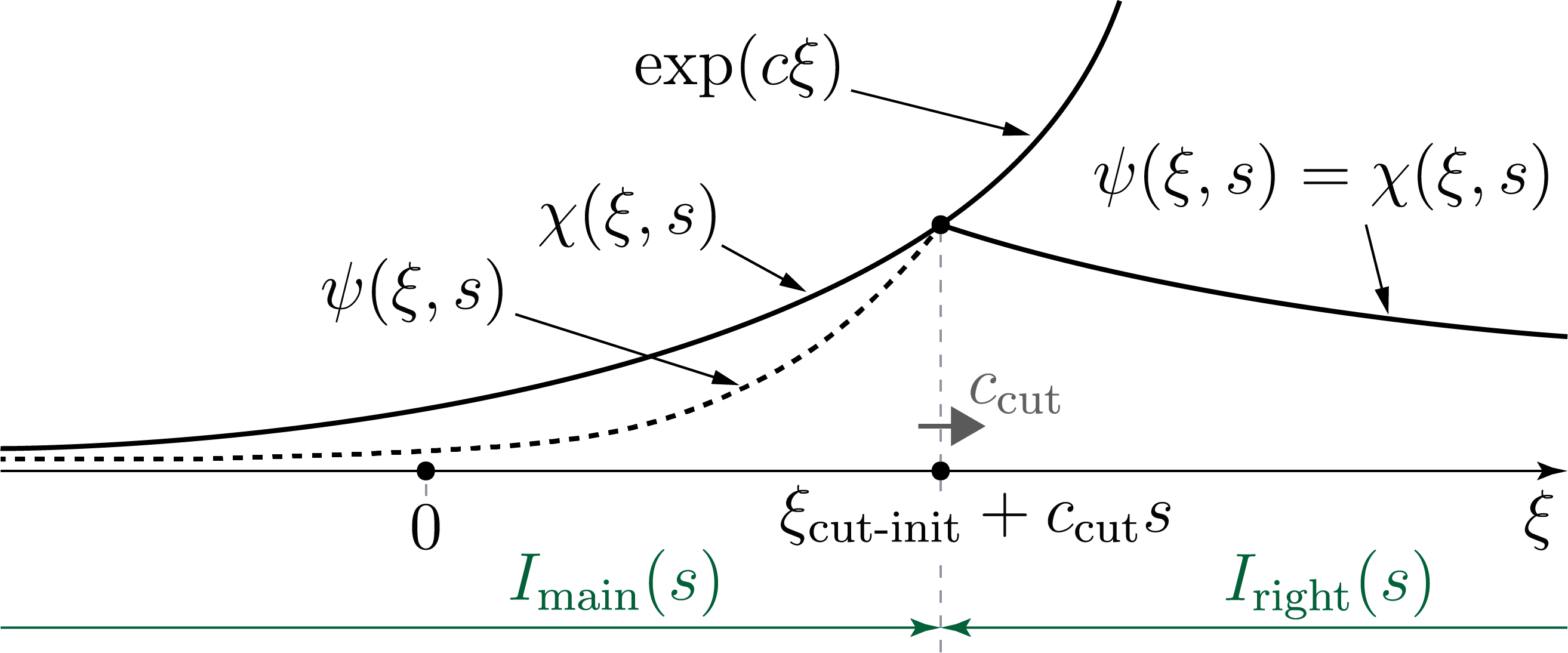}
    \caption{Graphs of the weight functions $\chi(\xi,s)$ and $\psi(\xi,s)$.}
    \label{fig:chi_psi}
\end{figure}
and, for all $s$ in $[0,+\infty)$, let us define the ``energy'' $\eee(s)$ by
\[
\eee(s) = \int_{\rr} \chi(\xi,s)E^\dag(\xi,s) \, d\xi\,,
\quad\text{where}\quad
E^\dag(\xi,s) = \frac{\alpha}{2}v_s(\xi,s)^2 + \frac{1}{2}v_\xi(\xi,s)^2 + V^\dag\bigl(v(\xi,s)\bigr)
\,.
\]
\subsubsection{Time derivative of the localized energy}
\label{subsubsec:der_loc_en}
For every nonnegative quantity $s$, let us define the ``dissipation'' $\ddd(s)$ by
\begin{equation}
\label{def_dissip_tr_fr}
\ddd(s) = \int_{\rr} \chi(\xi,s)\, v_s(\xi,s)^2 \, d\xi
\,.
\end{equation}
\begin{lemma}[time derivative of the localized energy]
\label{lem:upp_bd_dE_tf}
For every nonnegative quantity $s$, 
\begin{equation}
\label{upp_bd_dE_tf}
\begin{aligned}
\eee'(s) \le & -(1+\alpha c^2) \ddd(s) \\
& + (c+\kappa)\int_{\iRight(s)}\chi \biggl[ \frac{\alpha(2c+\cCut)+1}{2} v_s^2 + \frac{\cCut+1}{2} v_\xi^2 + \cCut V^\dag(v) \biggr] \, d\xi
\,.
\end{aligned}
\end{equation}
\end{lemma}
\begin{proof}
According to expression \vref{ddt_loc_en_tf_second} for the derivative of a localized energy and from the definition \cref{def_dissip_tr_fr} of $\ddd(s)$,
\begin{equation}
\label{upp_bd_dE_tf_proof}
\eee'(s) = -(1+\alpha c^2) \ddd(s) + \int_{\rr}\biggl[ \chi_s \Bigl( \frac{\alpha}{2}v_s^2 +  \frac{1}{2}v_\xi^2 + V^\dag(v) \Bigr) + (c\chi-\chi_\xi)(\alpha c v_s^2 + v_\xi\cdot v_s) \biggr] \, d\xi 
\,.
\end{equation}
It follows from the definition of $\chi$ that
\[
\chi_s(\xi,s) = 
\left\{
\begin{aligned}
&0  
& &\text{if}\quad \xi \in \iMain(s) \,, \\ 
&\cCut(c+\kappa)\, \chi(\xi,s) 
& &\text{if}\quad \xi \in \iRight(s) \,,
\end{aligned}
\right.
\]
and
\[
(c\chi-\chi_\xi)(\xi,s) = 
\left\{
\begin{aligned}
&0 
& &\text{if}\quad \xi \in \iMain(s) \,, \\ 
&(c+\kappa)\, \chi(\xi,s) 
& &\text{if}\quad \xi \in\iRight(s) \,.
\end{aligned}
\right.
\]
Thus it follows from \cref{upp_bd_dE_tf_proof}  that
\[
\begin{aligned}
\eee'(s) & = -(1+\alpha c^2) \ddd(s) \\
& + (c+\kappa)\int_{\iRight(s)}\chi \biggl[ \cCut \Bigl( \frac{\alpha}{2}v_s^2 +  \frac{1}{2}v_\xi^2 + V^\dag(v) \Bigr) +  (\alpha c v_s^2 + v_\xi\cdot v_s) \biggr] \, d\xi 
\,,
\end{aligned}
\]
and using the inequality
\[
v_\xi \cdot v_s \le \frac{1}{2}v_\xi^2 + \frac{1}{2}v_s^2
\,,
\]
it follows that
\[
\begin{aligned}
\eee'(s) \le & -(1+\alpha c^2) \ddd(s) \\
& + (c+\kappa)\int_{\iRight(s)}\chi \biggl[ \Bigl(\frac{\alpha\cCut}{2}+\alpha c + \frac{1}{2}\Bigr) v_s^2 + \Bigl( \frac{\cCut}{2} + \frac{1}{2}\Bigr) v_\xi^2 + \cCut V^\dag(v) \biggr] \, d\xi
\,.
\end{aligned}
\]
and inequality \cref{upp_bd_dE_tf} follows. \Cref{lem:upp_bd_dE_tf} is proved.
\end{proof}
\subsubsection{Firewall function}
\label{subsubsec:def_firewall}
A second function (the ``firewall'') will now be defined, to get some control over the second term of the right-hand side of inequality \cref{upp_bd_dE_tf}.
Let us introduce the function $\psi(\xi,s)$ (weight function for the firewall function) defined as
\[
\psi(\xi,s) = 
\left\{
\begin{aligned}
&\exp\Bigl( \kappa \bigl(\xi - ( \initCut + \cCut s)\bigr) \Bigr) \chi(\xi,s) 
& &\text{if}\quad \xi \in \iMain(s) \,, \\ 
&\chi(\xi,s) 
& &\text{if}\quad \xi \in \iRight(s)
\,,
\end{aligned}
\right.
\]
see \cref{fig:chi_psi}. For every real quantity $\xi$ and every nonnegative quantity $s$, following expression \vref{fire_def_tf}, let
\begin{equation}
\label{def_fire}
\begin{aligned}
F^\dag(\xi,s) &= 2\alpha E^\dag(\xi,s) + \alpha v(\xi,s)\cdot v_s(\xi,s) + \Bigl( \frac{1}{2} + \alpha c \frac{\psi_\xi(\xi,s)}{\psi(\xi,s)} \Bigr) v(\xi,s)^2 \\
&= \left(\alpha^2 v_s^2 + \alpha v_\xi^2 + 2\alpha V^\dag(v) + \alpha v\cdot v_s + \Bigl( \frac{1}{2} + \alpha c \frac{\psi_\xi}{\psi} \Bigr) v^2\right)(\xi,s)
\,, 
\end{aligned}
\end{equation}
and let
\[
\fff(s) = \int_{\rr} \psi(\xi,s) F^\dag(\xi,s) \, d\xi 
\,.
\]
\subsubsection{Lower bound on the firewall function}
\begin{lemma}[lower bound on the firewall function]
\label{lem:lower_bound_F}
For every nonnegative quantity $s$,
\begin{equation}
\label{lower_bound_F}
\fff(s) \ge \int_{\rr}\psi(\xi,s)\Bigl[ \frac{\alpha^2}{4} v_s(\xi,s)^2 + \alpha v_\xi(\xi,s)^2 + 2 \alpha V^\dag\bigl(v(\xi,s)\bigr) \Bigr] \, d\xi 
\,.
\end{equation}
\end{lemma}
\begin{proof}
According to the polarization inequality \vref{polarization_inequality} and since the ratio $\psi_\xi/\psi$ is greater than or equal to $-\kappa$, the following inequality holds for every real quantity $\xi$ and every nonnegative quantity $s$:
\[
F^\dag(\xi,s) \ge \frac{\alpha^2}{4} v_s^2 + \alpha v_\xi^2 + 2 \alpha V^\dag(v) + \Bigl( \frac{1}{6} - \alpha c \kappa \Bigr) v^2 
\,.
\]	
Thus inequality \cref{lower_bound_F} follows from condition \vref{condition_kappa_lower_bound_F} satisfied by $\kappa$. 
\end{proof}
\subsubsection{Energy decrease up to firewall and pollution}
For every nonnegative quantity $s$, let
\[
\SigmaEsc(s) =\bigl\{\xi\in\rr : \abs{v(\xi,s)} > \dEsc(m) \bigr\}
\,.
\]
\begin{lemma}[energy decrease up to firewall and pollution]
\label{lem:approx_decr_energy}
There exist nonnegative quantities $\KEF$ and $\KEEsc$, depending on $\alpha$ and $V$ and $m$ (only), such that for every nonnegative quantity $s$, 
\begin{equation}
\label{dt_E_final}
\eee'(s) \le -(1+\alpha c^2) \ddd(s) + \KEF \fff(s) + \KEEsc \int_{\SigmaEsc(s)} \psi(\xi,s) \, d\xi
\,.
\end{equation}
\end{lemma}
\begin{proof}
For every nonnegative quantity $s$, since $\chi(\xi,s)=\psi(\xi,s)$ for all $\xi$ in $\iRight(s)$, it follows from inequality \cref{upp_bd_dE_tf} of \cref{lem:upp_bd_dE_tf} that (substituting $\chi$ with $\psi$ and replacing $V^\dag(v)$ by its absolute value),
\[
\begin{aligned}
\eee'(s)& +(1+\alpha c^2) \ddd(s)\le   \\
& (c+\kappa)\int_{\iRight(s)}\psi \biggl[ \frac{\alpha(2c+\cCut)+1}{2} v_s^2 + \frac{\cCut+1}{2} v_\xi^2 + \cCut \abs{V^\dag(v)} \biggr] \, d\xi
\,,
\end{aligned}
\]
and since the integrand of the integral on the right-hand side of this inequality is nonnegative, this inequality still holds if the domain of integration is changed from $\iRight(s)$ to $\rr$. 

Let $\KEF$ be a positive quantity to be chosen below. According to \cref{lower_bound_F}, it follows that, for every nonnegative quantity $s$,
\[
\begin{aligned}
\eee'(s) +& (1+\alpha c^2)\ddd(s) - \KEF \fff(s) \le  \int_{\rr}\psi \biggl[ \Bigl(\frac{(c+\kappa)\bigl(\alpha(2c+\cCut)+1\bigr)}{2} - \frac{\alpha^2\KEF}{4}\Bigr) v_s^2  \\
&  + \Bigl(\frac{(c+\kappa)(\cCut+1)}{2} - \alpha\KEF\Bigr) v_\xi^2 + (c+\kappa) \cCut \abs{V^\dag(v)} - 2\alpha \KEF V^\dag(v)  \biggr] \, d\xi
\,.
\end{aligned}
\]
Thus, introducing the quantity $\KEF$ as
\[
\KEF = \max\biggl[\frac{2(\cUpp+\kappa)\bigl(\alpha(2\cUpp+\cCut)+1\bigr)}{\alpha^2},\frac{(\cUpp+\kappa)(\cCut+1)}{2\alpha},\frac{(\cUpp+\kappa)\cCut}{2\alpha}\biggr]
\]
(this quantity depends only on $\alpha$ and $V$), it follows that
\[
\eee'(s) +(1+\alpha c^2)\ddd(s)- \KEF \fff(s) \le \int_{\rr}\psi \Bigl[(c+\kappa) \cCut \abs{V^\dag(v)} - 2\alpha \KEF V^\dag(v)\Bigr]\, d\xi
\,.
\]
As long as $\xi$ is \emph{not} in $\SigmaEsc(s)$, it follows from \cref{posit_pot_around_loc_min_dag} that $V^\dag(v)$ is nonnegative and it follows from the last condition defining $\KEF$ that the integrand of the integral at the right-hand side of this last inequality is nonpositive. As a consequence, this inequality still holds if the integration domain of this integral is changed from $\rr$ to $\SigmaEsc(s)$. Namely, 
\begin{equation}
\label{dt_E_final_proof}
\begin{aligned}
\eee'(s) +(1+\alpha c^2)\ddd(s)- \KEF \fff(s) &\le \int_{\SigmaEsc(s)}\psi \Bigl[(c+\kappa) \cCut \abs{V^\dag(v)} - 2\alpha \KEF V^\dag(v)\Bigr]\, d\xi \\
&\le \bigl[(c+\kappa) \cCut + 2\alpha\KEF \bigr]\int_{\SigmaEsc(s)}\psi\abs{V^\dag(v)} \, d\xi 
\,.
\end{aligned}
\end{equation}
Thus, introducing the quantity $\KEEsc$ as
\[
\KEEsc = \Bigl((\cUpp+\kappa) \cCut + 2\alpha\KEF \Bigr)\max_{u\in\rr^d, \ \abs{u}\le\Rattinfty}\abs{V(u)-V(m)}
\,,
\]
 inequality \cref{dt_E_final} follows from \cref{dt_E_final_proof}. \Cref{lem:approx_decr_energy} is proved.
\end{proof}
\subsubsection{Relaxation scheme inequality, 1}
For every nonnegative quantity $s$, let 
\[
\mathcal{G} (s) = \int_{\SigmaEsc(s)} \psi(\xi,s) \, d\xi 
\,.
\]
Let $\sFin$ be a nonnegative quantity (denoting the length of the time interval on which the relaxation scheme will be applied). 
It follows from \cref{lem:approx_decr_energy} that
\begin{equation}
\label{relax_sch_1}
(1+\alpha c^2) \int_0^{\sFin} \ddd(s) \, ds \le \eee(0) - \eee(\sFin) + \KEF \int_0^{\sFin} \fff(s) \, ds + \KEEsc \int_0^{\sFin} \mathcal{G}(s) \, ds
\,.
\end{equation}
This is the first version of the relaxation scheme inequality that is the key argument to prove \cref{prop:inv_cv} (invasion implies convergence).
The aim of the two next \namecref{subsubsec:der_fire} is to gain some control over the quantities $\fff(s)$ and $\mathcal{G}(s)$. 
\subsubsection{Firewall upper bound}
%
The following lemma is the ``travelling frame'' analogue of \cref{lem:upp_bound_fireZero}.
\begin{lemma}[firewall upper bound]
\label{lem:fire_upp_bd}
For every nonnegative quantity $s$, 
\begin{equation}
\label{fire_upp_bd}
\fff(s) \le \int_{\rr} \psi \Bigl[\frac{3\alpha^2}{2} v_s^2 + \alpha v_\xi^2 + 2\alpha V^\dag(v) + \bigl(1+\alpha c(c+\kappa)\bigr)v^2 \Bigr]\, d\xi
\,.
\end{equation}
\end{lemma}
\begin{proof}
Inequality \cref{fire_upp_bd} follows from the definition \vref{def_fire} of $F\dag(\xi,s)$, from the fact that $\psi_\xi/\psi$ is bounded from above by $c+\kappa$, and from the inequality
\[
\alpha v\cdot v_s \le \frac{\alpha^2}{2} v_s^2 + \frac{1}{2} v^2
\,.
\]
\end{proof}
\subsubsection{Firewall linear decrease up to pollution}
\label{subsubsec:der_fire}
The following lemma is the ``travelling frame'' analogue of \cref{lem:decrease_fire0}.
\begin{lemma}[firewall linear decrease up to pollution]
\label{lem:fire_decr}
There exist positive quantities $\nuF$ and $\KFire$, depending on $\alpha$ and $V$ and $m$ (only), such that for every nonnegative quantity $s$, 
\begin{equation}
\label{dt_F_final}
\fff'(s) \le - \nuF \fff (s)+ \KFire \mathcal{G}(s)
\,.
\end{equation}
\end{lemma}
\begin{proof}
According to expressions \vref{ddt_loc_en_tf_first,ddt_loc_L2_tf} for the time derivatives of the functionals in a travelling frame, for every nonnegative quantity $s$, 
\[
\begin{aligned}
\fff'(s) = & \int_{\rr} \Biggl[ 
\alpha\psi_s \Bigl( \alpha v_s^2 + v_\xi^2 + 2 V^\dag(v) \Bigr) 
- 2\alpha \bigl( \psi + \alpha c \psi_\xi \bigr) v_s^2
+ 2\alpha (c \psi - \psi_\xi) v_\xi \cdot v_s \\
& + \psi_s \Bigl( \alpha v \cdot v_s + \frac{1}{2}v^2 - 2 \alpha c v \cdot v_\xi \Bigr)
+ \psi \Bigl( - v \cdot \nabla V^\dag (v) - v_\xi^2 + \alpha v_s^2 - 2 \alpha c v_\xi \cdot v_s \Bigr) \\
& + \frac{\psi_{\xi\xi} - c \psi_\xi}{2} v^2 
\Biggr] \, d\xi
\,.
\end{aligned}
\]
Simplifying the terms involving $\psi \ v_s^2$ and those involving $\psi \ v_\xi \cdot v_s$, and rearranging terms, it follows that
\[
\begin{aligned}
\fff'(s) = & \int_{\rr} \Biggl[ 
\alpha \bigl( - \psi - 2\alpha c \psi_\xi + \alpha \psi_s \bigr)v_s^2 + \bigl( - \psi + \alpha \psi_s \bigr)v_\xi^2  - \psi v \cdot \nabla V^\dag (v) \\
& - 2 \alpha \psi_\xi v_\xi \cdot v_s + \frac{\psi_s + \psi_{\xi\xi} - c \psi_\xi}{2} v^2  + \alpha\psi_s \bigl( 2 V^\dag(v) + v \cdot v_s - 2  c v \cdot v_\xi \bigr)
\Biggr] \, d\xi
\,.
\end{aligned}
\]
According to the definition of $\psi$, 
\[
\psi_s(\xi,s) = 
\left\{
\begin{aligned}
&- \kappa\cCut\psi(\xi,s)
& &\text{if}\quad \xi \in \iMain(s) \,, \\ 
&(c+\kappa)\cCut\psi(\xi,s) 
& &\text{if}\quad \xi \in \iRight(s) \,,
\end{aligned}
\right.
\]
and 
\[
c\psi(\xi,s)-\psi_\xi(\xi,s) = 
\left\{
\begin{aligned}
&-\kappa\psi(\xi,s)
& &\text{if}\quad \xi \in \iMain(s) \,, \\ 
&(c+\kappa)\psi(\xi,s) 
& &\text{if}\quad \xi \in \iRight(s) \,,
\end{aligned}
\right.
\]
and, for all $\xi$ in $\rr$, if $\delta_{\initCut + \cCut s}(\cdot)$ denotes the Dirac mass at $\xi=\initCut + \cCut s$, then
\[
\psi_{\xi\xi}(\xi,s)- c\psi_\xi(\xi,s)= \kappa(c+\kappa) \psi(\xi,s)-(c+2\kappa)\exp\bigl[c(\initCut + \cCut s)\bigr]\delta_{\initCut + \cCut s}(\xi)
\,.
\]
As a consequence, the following inequalities hold for all values of the arguments:
\begin{equation}
\label{bounds_psi_psi_s_psi_xi_etc}
\abs{\psi_s}\le \cCut(c+\kappa) \, \psi
\quad\text{and}\quad
\psi_{\xi\xi}-c\psi_\xi \le \kappa(c+\kappa) \, \psi
\end{equation}
Thus, for every nonnegative quantity $s$, it follows from the previous expression of $\fff'(s)$ that
\[
\begin{aligned}
&\fff'(s)  \le \int_{\rr} \psi \Biggl[ 
 \alpha\Bigl( -1 - 2\alpha c \frac{\psi_\xi}{\psi} + \alpha \cCut(c+\kappa) \Bigr) v_s^2+ \bigl(-1 + \alpha \cCut(c+\kappa)\bigr)v_\xi^2 - v \cdot \nabla V^\dag (v) \\
& - 2 \alpha \frac{\psi_\xi}{\psi} v_\xi \cdot v_s + \frac{(\cCut+\kappa) (c+\kappa)}{2}v^2 + \alpha\cCut(c+\kappa) \Bigl( 2 \abs{V^\dag(v)} +  \abs{v \cdot v_s} + 2  c \abs{v \cdot v_\xi} \Bigr)
\Biggr] \, d\xi
\,.
\end{aligned}
\]
Using the inequalities
\[
- 2 \alpha \frac{\psi_\xi}{\psi} v_\xi \cdot v_s \le \frac{1}{2}v_\xi^2 + 2 \alpha^2 \frac{\psi_\xi^2}{\psi^2} v_s^2
\quad\text{and}\quad
\abs{v \cdot v_s} \le \frac{1}{2}v^2+ \frac{1}{2}v_s^2
\quad\text{and}\quad
2\abs{v \cdot v_\xi}\le v^2+ v_\xi^2
\,,
\]
it follows that
\[
\begin{aligned}
\fff'(s) &  \le \int_{\rr} \psi \Biggl[ 
\alpha\Bigl( -1 - 2\alpha c \frac{\psi_\xi}{\psi} + \alpha \cCut(c+\kappa) + 2\alpha \frac{\psi_\xi^2}{\psi^2} + \frac{ \cCut(c+\kappa)}{2}\Bigr) v_s^2  \\
& +  \Bigl(-1 + \frac{1}{2} + \alpha \cCut(c+\kappa) (c+1)\Bigr)v_\xi^2 - v \cdot \nabla V^\dag (v) \\
& +  (c+\kappa)\Bigl( \frac{(\cCut+\kappa) }{2} + \frac{\alpha \cCut}{2} + \alpha c \cCut  \Bigr) v^2 + 2\alpha \cCut(c+\kappa) \abs{V^\dag(v)} 
\Biggr] \, d\xi
\,.
\end{aligned}
\]
Observe that the following equality holds, be the argument $\xi$ in $\iMain(s)$ or in $\iRight(s)$:
\[
- 2\alpha c \frac{\psi_\xi}{\psi}+ 2\alpha \frac{\psi_\xi^2}{\psi^2} 
= - 2\alpha \frac{\psi_\xi}{\psi}\cdot\frac{c\psi-\psi_\xi}{\psi}
= 2\alpha \kappa (c+\kappa)
\,.
\]
Thus, the previous inequality becomes
\[
\begin{aligned}
\fff'(s) &  \le \int_{\rr} \psi \Biggl[ 
\alpha  \Bigl( - 1 + (c+\kappa)\bigl( 2 \alpha\kappa + \cCut ( \alpha + 1/2)  \bigr) \Bigr) v_s^2  +  \Bigl(-\frac{1}{2} + \alpha \cCut(c+\kappa)(c+1) \Bigr) v_\xi^2 \\
& - v \cdot \nabla V^\dag (v) + \frac{c+\kappa}{2} \Bigl( \kappa + \cCut \bigl( 1 + \alpha(2c+1) \bigr) \Bigr) v^2 + 2\alpha \cCut(c+\kappa) \abs{V^\dag(v)} 
\Biggr] \, d\xi 
\,.
\end{aligned}
\]
According to the conditions \vref{condition_kappa_cCut_upper_bound_dsF} on $\kappa$ and $\cCut$,  it follows that
\begin{equation}
\label{dt_F_final_proof_use_of_conditions_on_kappa_cCut}
\fff'(s) \le \int_{\rr} \psi \Bigl[ 
-\frac{\alpha}{2} v_s^2 -\frac{1}{4} v_\xi^2 + \frac{\eigVmin(m)}{8} v^2 - v \cdot \nabla V^\dag (v) + \frac{1}{4} \abs{V^\dag(v)}
\Bigr] \, d\xi 
\,.
\end{equation}
Let $\nuF$ be a positive quantity to be chosen below. It follows from the previous inequality and from the upper bound \cref{fire_upp_bd} on $\fff(s)$ that
\begin{equation}
\label{dt_F_final_proof}
\begin{aligned}
\fff'(s) + \nuF \fff(s) \le \int_{\rr} \psi&\Biggl[ \frac{\alpha}{2}(-1 + 3\alpha\nuF)v_s^2 + \Bigl(-\frac{1}{4}+\alpha\nuF\Bigr) v_\xi^2 - v \cdot \nabla V^\dag (v) \\
&+ \Bigl(\frac{\eigVmin(m)}{8} + \nuF\bigl(1+\alpha c(c+\kappa)\bigr)\Bigr)v^2 + \Bigl(\frac{1}{4}+ 2\alpha\nuF\Bigr)\abs{V^\dag(v)} \Biggr] \, d\xi 
\,.
\end{aligned}
\end{equation}
In view of this inequality and of inequalities \vref{v_nablaV_controls_square_around_loc_min_dag,v_nablaV_controls_pot_around_loc_min_dag}, let us assume that $\nuF$ is small enough so that
\begin{equation}
\label{conditions_on_nuFire}
3\alpha\nuF \le 1
\quad\text{and}\quad
\alpha\nuF \le \frac{1}{4}
\quad\text{and}\quad
\nuF\bigl(1+\alpha c(c+\kappa)\bigr) \le \frac{\eigVmin(m)}{8}
\quad\text{and}\quad
2\alpha\nuF \le \frac{1}{4}
\,;
\end{equation}
The quantity $\nuF$ may be chosen as
\[
\nuF = \min\Bigl( \frac{1}{8\alpha}, \frac{\eigVmin(m)}{8\bigl(1+\alpha\cUpp(\cUpp+\kappa)\bigr)} \Bigr)
\,.
\]
Then, it follows from \cref{dt_F_final_proof,conditions_on_nuFire} that
\begin{equation}
\label{dt_F_final_proof_2}
\fff'(s) + \nuF \fff(s) \le \int_{\rr}\Bigl[- v \cdot \nabla V^\dag (v) + \frac{\eigVmin(m)}{4} v^2 + \frac{1}{2} \abs{V^\dag(v)} \Bigr]\, d\xi 
\,.
\end{equation}
According to \cref{v_nablaV_controls_square_around_loc_min_dag,v_nablaV_controls_pot_around_loc_min_dag}, the integrand of the integral at the right-hand side of this inequality is nonpositive as long as $\xi$ is \emph{not} in $\SigmaEsc(s)$.
Therefore this inequality still holds if the domain of integration of this integral is changed from $\rr$ to $\SigmaEsc(s)$. Besides, observe that, in terms of the ``initial'' potential $V$ and solution $u(x,t)$, the factor of $\psi$ under the integral of the right-hand side of this last inequality reads
\[
- (u-m)\cdot \nabla V(u)  + \frac{\eigVmin(m)}{4}(u-m)^2 + \frac{1}{2}\abs{V(u)-V(m)}
\,.
\]
Thus, if $\KFire$ denotes the quantity $\KFZero$ defined in \vref{def_KFireZero}, then, according to the $L^\infty$-bound \vref{hyp_attr_ball_Linfty} on the solution, inequality \cref{dt_F_final} follows from \cref{dt_F_final_proof_2} (with the domain of integration of the integral on the right-hand side restricted to $\SigmaEsc(s)$). This finishes the proof of \cref{lem:fire_decr}.
\end{proof}
\subsubsection{Firewall nonnegativity up to pollution}
For every nonnegative quantity $s$, let
\[
\SigmaEsc(s)=\{\xi\in\rr: \abs{v(\xi,s)} >\dEsc(m)\} 
\,.
\]
\begin{lemma}[firewall nonnegativity up to pollution]
\label{lem:coerc_fire}
For every nonnegative quantity $s$,
\begin{equation}
\label{coerc_F}
\fff(s) \ge -2 \alpha \Delta_V \int_{\SigmaEsc(s)} \psi (\xi,s) \, d\xi
\,.
\end{equation}
\end{lemma}
\begin{proof}
According to inequality \vref{posit_pot_around_loc_min_dag} the quantity $V^\dag(v)$ is nonnegative for $\xi$ in $\rr\setminus\SigmaEsc(s)$. Thus, inequality \cref{coerc_F} follows from the lower bound \cref{lower_bound_F} of \vref{lem:lower_bound_F} and from inequality \vref{Delta_V_for_V_dag}. \Cref{lem:coerc_fire} is proved.
\end{proof}
\subsubsection{Relaxation scheme inequality, 2}
For every nonnegative quantity $\sFin$, inequality \cref{dt_F_final} yields
\[
\int_0^{\sFin} \fff(s) \, ds \le \frac{1}{\nuF} \Bigl( \fff(0) - \fff(\sFin) + \KFire \int_0^{\sFin} \mathcal{G}(s) \, ds \Bigr) 
\,,
\]
and in view of inequality \cref{coerc_F} of \cref{lem:coerc_fire} (firewall coercivity up to pollution term), 
\[
-\fff(\sFin) \le 2 \alpha \Delta_V \mathcal{G}(\sFin)
\,.
\]
Thus the ``relaxation scheme'' inequality \cref{relax_sch_1} becomes
\begin{equation}
\label{relax_sch_2}
\begin{aligned}
(1+\alpha c^2) \int_0^{\sFin} \ddd(s) \, ds \le & \eee(0) - \eee(\sFin) + \frac{\KEF}{\nuF} \fff(0) + \frac{2 \alpha \Delta_V \KEF }{\nuF} \mathcal{G}(\sFin) \\
& + \Bigl( \frac{\KEF \, \KFire}{\nuF} + \KEEsc \Bigr) \int_0^{\sFin} \mathcal{G}(s) \, ds
\,.
\end{aligned}
\end{equation}
This is the second version of the relaxation scheme inequality. The aim of the next \namecref{subsubsec:control_add_flux} is to gain some control over the quantity $\mathcal{G}(s)$.
\subsubsection{Control over the pollution in the time derivative of the firewall function}
\label{subsubsec:control_add_flux}
For every nonnegative quantity $s$, let
\begin{equation}
\label{def_xiHom_xiesc_xiEsc}
\begin{aligned}
\xiHom(s) &= \sqrt{1+\alpha c^2} \bigl( \xHom(\tInit + s) - \xInit - \sigma s \bigr) \,,\\
\text{and}\quad 
\xiesc(s) &= \sqrt{1+\alpha c^2} \bigl( \xesc(\tInit + s) - \xInit - \sigma s \bigr) \,,\\
\text{and}\quad 
\xiEsc(s) &= \sqrt{1+\alpha c^2} \bigl( \xEsc(\tInit + s) - \xInit - \sigma s \bigr) \,, 
\end{aligned}
\end{equation}
see \vref{fig:inv_cv,fig:inv_cv_bis}.
According to properties \vref{xEsc_xesc_xHom} for the set $\SigmaEscZero(t)$, for all $s$ in $[0,+\infty)$, 
\[
\SigmaEsc(s) \subset (-\infty, \xiesc(s)] \cup [\xiHom(s),+\infty)
\,,
\]
thus, introducing the quantities 
\[
\Gback(s) = \int_{-\infty}^{\xiesc(s)} \psi (\xi,s)\, d\xi 
\quad\text{and}\quad
\Gfront(s) = \int_{\xiHom(s)}^{+\infty} \psi (\xi,s)\, d\xi 
\,,
\]
it follows that, for all $s$ in $[0,+\infty)$,
\[
\mathcal{G}(s) \le \Gback(s)+\Gfront(s) 
\,.
\]
The aim of this \namecref{subsubsec:control_add_flux} is to prove the bounds on $\Gback(s)$ and $\Gfront(s)$ provided by the next lemma. 
\begin{lemma}[upper bounds on $\Gback(s)$ and $\Gfront(s)$]
\label{lem:first_bound_Gback_Gfront}
For every nonnegative quantity $s$, the following estimates hold:
\begin{align}
\label{first_bd_Gback}
\Gback(s) &\le \frac{1}{\kappa} \exp\bigl( (c+\kappa) \, \xiesc(s) - \kappa \, \initCut - \kappa\, \cCut s\bigr)
\,, \\
\label{first_bd_Gfront}
\Gfront(s) &\le \frac{1}{\kappa}\exp\bigl[ (c+\kappa)\,  \initCut + (c + \kappa) (\cCut + \kappa)  s  -\kappa \, \xiHom(0)\bigr]
\,.
\end{align}
\end{lemma}
\begin{proof}
The integrand $\psi(\xi,s)$ in the expression of $\Gback(s)$ and $\Gfront(s)$ is less than or equal to
\[
\begin{aligned}
\exp\bigl[ (c+\kappa) \, \xi - \kappa(\initCut + \cCut\,  s) \bigr]
& \quad\text{for}\quad
\Gback(s) \,, \\
\text{and}\quad
\exp\bigl[ (c+\kappa) (\initCut + \cCut \, s) - \kappa \, \xi\bigr]
& \quad\text{for}\quad
\Gfront(s)
\,.
\end{aligned}
\]
Thus, by explicit calculation, 
\[
\Gback(s) \le \frac{1}{c+\kappa}\exp\bigl[ (c+\kappa) \xiesc(s) - \kappa \initCut - \kappa\, \cCut s\bigr] 
\,,
\]
and inequality \cref{first_bd_Gback} follows. 

Concerning $\Gfront(s)$, since $\xHom'(\cdot)$ is nonnegative (inequality \vref{hyp_xHom_prime_pos}), for all $s$ in $[0,+\infty)$, 
\[
\xiHom'(s) \ge -c 
\quad\text{thus}\quad
\xiHom(s) \ge \xiHom(0) - cs
\,.
\]
By explicit calculation, it follows that
\[
\Gfront(s) \le \frac{1}{\kappa} \exp\Bigl[ (c+\kappa)\,  \initCut + \bigl((c+\kappa) \, \cCut + \kappa \,  c\bigr) s -\kappa \, \xiHom(0)\Bigr]
\]
and inequality \cref{first_bd_Gfront} follows. \Cref{lem:first_bound_Gback_Gfront} is proved. 
\end{proof}
\subsubsection{Further (subsonic) bound on invasion speed}
\label{subsubsec:proof_bd_cmax}
\paragraph{Statement.}
Up to now, the quantity $\cUpp$ has only been used to state hypothesis \cref{hyp_param_relax_sch}, which assumes that the parabolic speed of the travelling frame under consideration does not exceed this quantity. Now, the relaxation scheme set up above will be applied in order to prove that this quantity $\cUpp$ is indeed an upper bound for the speed of invasion. The aim of this \namecref{subsubsec:fin_relax} is to prove the following lemma.
\begin{lemma}[invasion speed is subsonic]
\label{lem:further_bd_finite_speed}
The following inequality holds
\[
\barsescsup \le \frac{\cUpp}{\sqrt{1+\alpha\cUpp^2}}
\,.
\]
\end{lemma}
It follows from this lemma that the mean speed $\barsescsup$ is smaller than $1/\sqrt{\alpha}$ (which proves conclusion \cref{item:mean_speed_sEsc_smaller_than_sound_speed} of \cref{prop:inv_cv}). If $\sUpp$ denotes the ``physical'' counterpart of $\cUpp$ and $\barcescsup$ denotes the ``parabolic'' counterpart of $\barsescsup$, that is
\[
\sUpp = \frac{\cUpp}{\sqrt{1+\alpha\cUpp^2}}
\quad\text{and}\quad
\barcescsup = \frac{\barsescsup}{\sqrt{1-\alpha\barsescsup^2}}
\,,
\]
then the conclusion of \cref{lem:further_bd_finite_speed} may be stated under the form of the following two equivalent inequalities:
\[
\barsescsup \le \sUpp
\iff
\barcescsup \le \cUpp
\,.
\]
\paragraph{Idea of the proof.}
The idea of the proof of \cref{lem:further_bd_finite_speed} is due to Gallay and Joly, see \cite[Lemma~5.2]{GallayJoly_globStabDampedWaveBistable_2009}). The principle is that, if the previous relaxation scheme is applied in a travelling frame with a parabolic speed $c$ greater than or equal to $\cUpp$, then, according to \vref{lem:posit_en_Esc}, the following lower bound holds (for the quantity $\Eesc$ defined in \vref{def_cmax}):
\[
\int_{-\infty}^{\xiEsc(s)+1} e^{c\xi} \Bigl( \frac{\alpha}{2}v_s(\xi,s)^2 + \frac{1}{2}v_\xi(\xi,s)^2 + V^\dag\bigl(v(\xi,s)\bigr) \Bigr) \, d\xi \ge \Eesc \exp\bigl( \xiEsc(s) \bigr)
\,,
\]
and as a consequence the same kind of lower bound holds for the localized energy $\eee(s)$ defined in \cref{subsubsec:def_loc_en}. On the other hand, the relaxation scheme inequality \cref{relax_sch_2} provides an upper bound for this localized energy, and under appropriate conditions this will enable us to prove that this localized energy remains bounded from above. Finally, it will follow from these bounds that the Escape point $\xiEsc(s)$ must itself be bounded from above. It will turn out that this is contradictory with arbitrarily large positive values of the escape point $\xiesc(s)$, and in turn contradictory with a mean speed $\barcescsup$ exceeding $\cUpp$. 
\paragraph{Set-up.}
Let us proceed by contradiction and assume that the converse assertion holds: 
\[
\sUpp < \barsescsup\,,
\quad\text{or equivalently,}\quad
\cUpp < \barcescsup
\,.
\]
Let $\varepsilon$ denote a positive quantity, small enough so that
\[
\sUpp < \barsescsup - \varepsilon
\,,
\]
and let us make in addition the following technical hypothesis (see the comment below after the statement of \cref{lem:choice_t_np}):
\begin{equation}
\label{hyp_vicinty_max_vel}
\varepsilon < \frac{1}{\sqrt{1+\alpha\cUpp^2}} \frac{\kappa\cCut}{2(\cUpp+\kappa)}
\,.
\end{equation}
\paragraph{Origin of time intervals.}
The following lemma provides appropriate time intervals where the relaxation scheme will be applied. Here are the features of these time intervals:
\begin{itemize}
\item the mean speed of the escape point is almost maximal on them;
\item their length is arbitrarily large;
\item for a given length they occur at arbitrarily large positive times. 
\end{itemize}
\begin{lemma}[time intervals with controlled length and large positive left endpoints where mean speed of escape point is almost maximal]
\label{lem:choice_t_np}
For every positive integer $n$, there exists a sequence $(t_{n,p})_{p\in\nn}$ of positive quantities going to $+\infty$ as $p$ goes to $+\infty$, and such that, for every nonnegative integer $p$,
\begin{equation}
\label{choice_t_np}
\xesc(t_{n,p} + n) - \xesc(t_{n,p}) \ge (\barsescsup - \varepsilon) n
\,.
\end{equation}
\end{lemma}
The technical hypothesis \cref{hyp_vicinty_max_vel} above will be used in the proof of \vref{lem:growth_xiesc_of_n}, stating that the escape point ends ``far to the right'' at the end of the relaxation scheme that is going to be considered. 
\begin{proof}[Proof of \cref{lem:choice_t_np}]
If the converse was true, then there would exist a positive integer $n$ and a positive time $t_0$ such that, for every time $t$ greater than or equal to $t_0$,
\[
\frac{\xesc(t + n) - \xesc(t)}{n} \le \barsescsup - \varepsilon
\]
and this would imply that
\[
\limsup_{s\to+\infty}\sup_{t\in[0,+\infty)}\frac{\xesc(t+s)-\xesc(t_1)}{s}\le \barsescsup - \varepsilon
\,,
\]
a contradiction with the definition of $\barsescsup$.
\end{proof}
For every positive integer $n$, let us introduce a sequence $(t_{n,p})_{p\in\nn}$ satisfying the conclusions of \cref{lem:choice_t_np} above, and let $p(n)$ and $\xInit^{(n)}$ denote a nonnegative integer and a real quantity to be chosen below. Finally, let us take the following notation:
\[
\tInit^{(n)} = t_{n,p(n)}
\,.
\]
The relaxation scheme set up in the previous \namecref{subsubsec:der_fire} will be applied with the following set of parameters:
\[
\tInit = \tInit^{(n)}
\quad\text{and}\quad
\xInit=\xInit^{(n)}
\quad\text{and}\quad
c = \cUpp
\quad\text{and}\quad
\initCut = 0
\,.
\]
Let us denote by
\[
\begin{aligned}
& \xiesc^{(n)}(\cdot) 
\quad\text{and}\quad
\xiEsc^{(n)}(\cdot)
\quad\text{and}\quad
\chi^{(n)}(\cdot,\cdot) 
\quad\text{and}\quad
\eee^{(n)}(\cdot) 
\quad\text{and}\quad
\fff^{(n)}(\cdot)  \\
\text{and}\quad
& \xiHom^{(n)}(\cdot)
\quad\text{and}\quad 
\Gback^{(n)}(\cdot)
\quad\text{and}\quad
\Gfront^{(n)}(\cdot)
\end{aligned}
\]
the objects defined in the previous \namecref{subsubsec:proof_bd_cmax}s (with the same notation except the ``$(n)$'' superscripts to emphasize the fact that these objects depend on $n$). The relaxation scheme will be considered on a time interval of length $\sFin = n$, that is between the times $\tInit^{(n)}$ and $\tInit^{(n)} + n$. Observe that, according to the conclusion \cref{choice_t_np} of \cref{lem:choice_t_np}, whatever the choice of $p(n)$ and $\xInit^{(n)}$, 
\begin{equation}
\label{growth_esc_point_trav_frame}
\frac{\xiesc^{(n)}(n)-\xiesc^{(n)}(0)}{n}\ge \sqrt{1+\cUpp^2}(\barsescsup - \varepsilon - \sUpp) >0
\,,
\end{equation}
\begin{figure}[!htbp]
	\centering
    \includegraphics[width=.8\textwidth]{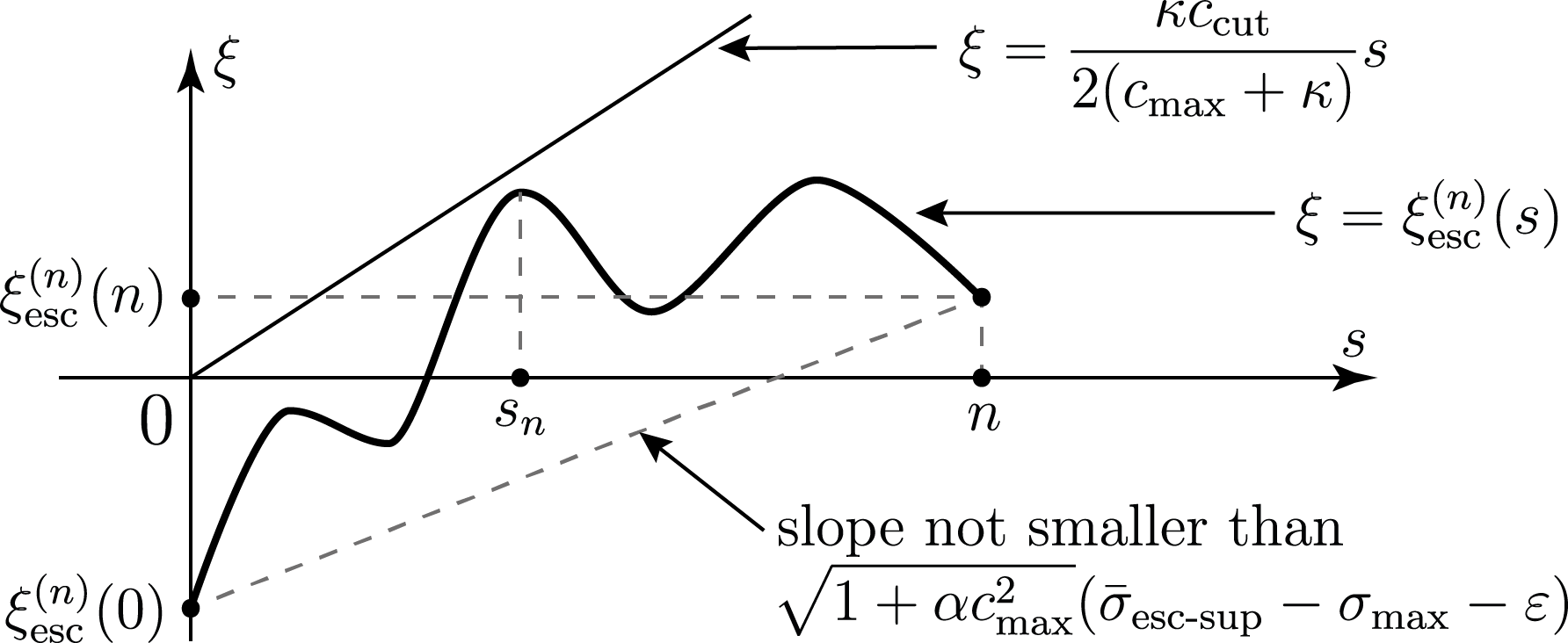}
    \caption{Definition of the quantity $\xInit(n)$. An increase of $\xInit^{(n)}$ translates the graph of $x\mapsto\xiesc^{(n)}(s)$ downwards. The value chosen for $\xInit^{(n)}$ is the least one so that this graph remains below the slope starting from the origin on the interval $[0,n]$. The figure aims at displaying the assertion of \cref{lem:growth_xiesc_of_n}, that is the fact that $\xiesc^{(n)}(n)$ goes to $+\infty$ as $n$ goes to $+\infty$.}
    \label{fig:proof_upper_bound_speed}
\end{figure}
see \cref{fig:proof_upper_bound_speed}.

To set up this relaxation scheme there still remains to define the two quantities $p(n)$ and $\xInit^{(n)}$. The purpose is to make this choice in such a way that the following two conditions be fulfilled: 
\begin{itemize}
\item the quantity $\eee^{(n)}(n)$ (the localized energy in travelling frame at the end of the relaxation time interval) remains bounded as $n$ goes to $+\infty$;
\item the quantity $\xi^{(n)}(n)$ (the escape point in travelling frame at the end of the relaxation time interval) goes to $+\infty$ as $n$ goes to $+\infty$.
\end{itemize}
\paragraph{Origin of space.}
Guided by expression inequality \cref{first_bd_Gback} on $\Gback(\cdot)$, let us choose the quantity $\xInit^{(n)}$ as the least real quantity such that, for every $s$ in the interval $[0,n]$, the following condition be fulfilled:
\begin{equation}
\label{hyp_no_excursion_right_cutoff}
(\cUpp+\kappa) \xiesc^{(n)}(s) \le \frac{\kappa\cCut}{2} s
\,,
\end{equation}
see \cref{fig:proof_upper_bound_speed}. 

According to definition \cref{def_xiHom_xiesc_xiEsc}
\[
\xiesc^{(n)}(s) = \sqrt{1 + \alpha\cUpp^2}\bigl( \xesc^{(n)}(\tInit^{(n)} + s ) - \xInit^{(n)} - \sUpp s \bigr)
\,,
\]
thus in other words, let us choose the quantity $\xInit^{(n)}$ as
\begin{equation}
\label{def_xinit_n}
\xInit^{(n)} = \sup_{s\in[0,n]} \xesc(\tInit^{(n)} + s) - \Bigl( \sUpp + \frac{\kappa\cCut}{2\sqrt{1 + \alpha\cUpp^2}(\cUpp + \kappa)} \Bigr) s
\end{equation}
(according to inequality \vref{control_escape} controlling the increase of $\xesc(\cdot)$, this supremum is finite). Condition \cref{hyp_no_excursion_right_cutoff} will ensure that the terms involving $\Gback^{(n)}(\cdot)$ in the relaxation scheme inequality \cref{relax_sch_2} remain bounded. 

The relevance of this definition for the quantity $\xInit^{(n)}$ is justified by the following two lemmas. 
\paragraph{Origin of time intervals: upper bound on the final energy.}
\begin{lemma}[upper bound on the energy at the end of the time intervals]
\label{lem:upper_bound_final_energy}
For every positive integer $n$, if the integer $p(n)$ is chosen large enough, then the ``final'' energy $\eee^{(n)}(n)$ is bounded from above by a quantity that does not depend on $n$. 
\end{lemma}
\begin{proof}
The proof is based of the relaxation scheme inequality \cref{relax_sch_2}. Thus, let us consider the various terms involved in this inequality.

First, let us observe that since the quantity $\initCut$ is equal to $0$, the quantities $\eee^{(n)}(0)$ and $\fff^{(n)}(0)$ are bounded from above by quantities depending only on $\alpha$ and $V$ (this follows from the bound \vref{hyp_attr_ball_X} for the solution). 

Now, according to inequalities \cref{first_bd_Gback,hyp_no_excursion_right_cutoff}, for every $s$ in $[0,n])$, 
\[
\Gback^{(n)}(s) \le \frac{1}{\kappa} \exp(-\kappa\cCut s/2)
\,,
\]
and this ensures that the terms involving $\Gback^{(n)}(\cdot)$ in inequality \cref{relax_sch_2} are bounded from above by quantities that do not depend on $n$. 

Finally, let us deal with the function $\Gfront^{(n)}(\cdot)$. According to inequality \cref{first_bd_Gfront}, for every nonnegative quantity $s$, 
\[
\Gfront^{(n)}(s) \le \frac{1}{\kappa} \exp \bigl( (\cUpp+\kappa)(\cCut + \kappa) s - \kappa \xiHom^{(n)}(0)\bigr) 
\]
and according to definition \cref{def_xiHom_xiesc_xiEsc},
\[
\xiHom^{(n)}(0) = \sqrt{1+\alpha \cUpp^2} \Bigl( \xHom(\tInit^{(n)})- \xInit^{(n)} \Bigr) 
\,.
\]
On the other hand, according to the definition of $\xInit^{(n)}$ and to inequality \vref{control_escape} controlling the increase of $\xesc(\cdot)$, 
\begin{equation}
\label{upp_bd_xInit_n}
\xInit^{(n)} \le \xesc(\tInit^{(n)}) + \snoesc n
\,,
\end{equation}
thus
\[
\xiHom^{(n)}(0) \ge \sqrt{1+\alpha \cUpp^2} \Bigl( \xHom(\tInit^{(n)})- \xesc(\tInit^{(n)}) - \snoesc n \Bigr) 
\]
and this shows that the quantity $\xiHom^{(n)}(0)$ is arbitrarily large positive provided that the integer $p(n)$ is chosen large enough (depending on $n$). As a consequence, if the integer $p(n)$ is chosen large enough (depending on $n$), then the terms involving $\Gfront^{(n)}(\cdot)$ in inequality \cref{relax_sch_2} are bounded from above by quantities that do not depend on $n$. \Cref{lem:upper_bound_final_energy} is proved. 
\end{proof}
\paragraph{Length of time intervals: final position of escape point.}
\begin{lemma}[escape point ends up far to the right in travelling frame]
\label{lem:growth_xiesc_of_n}
The following convergence holds:
\[
\xiesc^{(n)}(n)\to +\infty
\quad\text{as}\quad
n\to +\infty
\,.
\]
\end{lemma}
\begin{proof}
According to inequality \cref{growth_esc_point_trav_frame} and to definition \vref{def_xiHom_xiesc_xiEsc},
\[
\begin{aligned}
\xiesc^{(n)}(n) & \ge \sqrt{1+\alpha\cUpp^2}(\barsescsup - \varepsilon - \sUpp) n + \xiesc^{(n)}(0) \\
& \ge \sqrt{1+\alpha\cUpp^2} \bigl( (\barsescsup - \varepsilon - \sUpp) n + \xesc ( \tInit^{(n)}) - \xInit^{(n)} \bigr) 
\,.
\end{aligned}
\]
Now, according to the definition \cref{def_xinit_n} of $\xInit^{(n)}$, there exists a quantity $s_n$ in $[0,n]$ such that 
\[
\xInit^{(n)} \le 1 + \xesc(\tInit^{(n)} + s_n) - \Bigl( \sUpp + \frac{\kappa\cCut}{2\sqrt{1 + \alpha\cUpp^2}(\cUpp + \kappa)} \Bigr) s_n
\,.
\]
It follows from the two previous inequalities that
\[
\begin{aligned}
\xesc & (\tInit^{(n)} + s_n) - \xesc( \tInit^{(n)}) \ge \\ 
& (\barsescsup - \varepsilon - \sUpp) n + \Bigl( \sUpp + \frac{\kappa\cCut}{2\sqrt{1 + \alpha\cUpp^2}(\cUpp + \kappa)} \Bigr) s_n - 1 - \frac{\xiesc^{(n)}(n)}{\sqrt{1+\alpha\cUpp^2}}
\,,
\end{aligned}
\]
thus, provided that $s_n$ is nonzero,
\[
\begin{aligned}
& \frac{\xesc (\tInit^{(n)} + s_n) - \xesc( \tInit^{(n)})}{s_n}  \ge \\
& \barsescsup - \varepsilon + \frac{\kappa\cCut}{2\sqrt{1 + \alpha\cUpp^2}(\cUpp + \kappa)} - \frac{1}{s_n} - \frac{\xiesc^{(n)}(n)}{s_n\sqrt{1+\alpha\cUpp^2}}
\,.
\end{aligned}
\]
Let us proceed by contradiction and assume that there exists a quantity $C$ such that, for arbitrarily large positive values of $n$, the quantity $\xiesc^{(n)}(n)$ is not larger than $C$. Then, according to inequality \cref{growth_esc_point_trav_frame}, for such values of $n$ the quantity $\xiesc^{(n)}(0)$ is large negative, and according to inequality \cref{control_escape} controlling the growth of $\xesc(\cdot)$, the quantity $s_n$ must be large positive. According to the technical hypothesis \cref{hyp_vicinty_max_vel}, it follows that, for such large enough positive values of $n$, 
\[
\frac{\xesc (\tInit^{(n)} + s_n) - \xesc( \tInit^{(n)})}{s_n} > \barsescsup
\,,
\]
a contradiction with the definition of $\barsescsup$. \Cref{lem:growth_xiesc_of_n} is proved. 
\end{proof}
\paragraph{Origin of time intervals: upper bound on the final energy, variant.}
The following lemma is a slight variant of \cref{lem:upper_bound_final_energy} above. 
\begin{lemma}[boundedness of energy at the end of the time intervals, variant]
\label{lem:upper_bound_final_energy_2}
For every positive integer $n$, if the integer $p(n)$ is chosen large enough, then the quantity
\[
\int_{-\infty}^{\cCut n} e^{c\xi} \Bigl( \frac{\alpha}{2}v^{(n)}_s(\xi,n)^2 + \frac{1}{2}v^{(n)}_\xi(\xi,n)^2 + V^\dag\bigl(v^{(n)}(\xi,n)\bigr)\Bigr) \, d\xi 
\]
is bounded from above by a quantity that does not depend on $n$. 
\end{lemma}
\begin{proof}
According to the definition (\cref{def_xiHom_xiesc_xiEsc}) of $\xiHom(\cdot)$, 
\[
\xiHom^{(n)}(n)= \sqrt{1+ \alpha\cUpp^2} \bigl( \xHom(\tInit^{(n)} + n) - \xInit^{(n)} - \sUpp n\bigr) 
\,,
\]
thus, according to inequality (\cref{upp_bd_xInit_n}), 
\[
\xiHom^{(n)}(n) \ge \sqrt{1+ \alpha\cUpp^2} \bigl( \xHom(\tInit^{(n)} + n) - \xesc(\tInit^{(n)}) - ( \snoesc + \sUpp) n\bigr)
\,.
\]
Thus, for every positive quantity $n$, if the integer $p(n)$ is chosen large enough, then the quantity $\xiHom^{(n)}(n)$ is arbitrarily large positive, and in particular greater than the point $\cCut n$.

In this case, according to the definition of the localized energy $\eee(\cdot)$ and of the weight function $\chi(\cdot,\cdot)$, since $\chi^{(n)}(\xi,n)$ equals $e^{c\xi}$ for every $\xi$ in the interval $(-\infty,\cCut n]$, the following inequality holds:
\[
\begin{aligned}
\eee^{(n)}(n) \ge & \int_{-\infty}^{\cCut n} e^{c\xi} \Bigl( \frac{\alpha}{2}v^{(n)}_s(\xi,n)^2 + \frac{1}{2}v^{(n)}_\xi(\xi,n)^2 + V^\dag\bigl(v^{(n)}(\xi,n)\bigr)\Bigr) \, d\xi \\
& + \int_{\xiHom(n)}^{+\infty} \chi^{(n)}(\xi,n) V^\dag\bigl(v^{(n)}(\xi,n)\bigr) \, d\xi
\,.
\end{aligned}
\]
According to the definition of the weight function $\chi(\cdot,\cdot)$, the second integral of the right-hand side of this inequality is arbitrarily close to $0$ if the quantity $\xiHom^{(n)}(n)$ is large enough positive, or in other words if the integer $p(n)$ is chosen large enough. In view of \cref{lem:upper_bound_final_energy}, this finishes the proof of \cref{lem:upper_bound_final_energy_2}.
\end{proof}
Let us assume from now on that for every positive integer $n$, the integer $p(n)$ is chosen large enough so that the conclusions of \cref{lem:upper_bound_final_energy,lem:growth_xiesc_of_n,lem:upper_bound_final_energy_2} be satisfied, and so that (as assumed in the proof of \cref{lem:upper_bound_final_energy_2}), 
\begin{equation}
\label{cCut_n_not_larger_than_xiHom}
\cCut n \le \xiHom^{(n)}(n)
\,.
\end{equation}
\paragraph{Upper bound for Escape point in travelling frame.}
Last not least, the definition of the quantity $\cUpp$ in \vref{subsec:further_bd_finite_speed} (and the fact that the speed of the travelling frame under consideration is as large as $\cUpp$) will now finally be used to prove the following lemma. 
\begin{lemma}[upper bound for Escape point in travelling frame]
\label{lem:boundedness_xiEsc_of_n}
The quantity $\xiEsc^{(n)}(n)$ remains bounded from above as $n$ goes to $+\infty$. 
\end{lemma}
\begin{proof}
According to inequalities \vref{xEsc_xesc_xHom,hyp_no_excursion_right_cutoff}, for every positive integer $n$, 
\begin{equation}
\label{upp_bound_xiEsc_cCut}
\xiEsc^{(n)}(n) + 1 \le \xiesc^{(n)}(n) + 1 \le \frac{\cCut}{2} n + 1
\,,
\end{equation}
thus as soon as $n$ is large enough,
\[
\xiEsc^{(n)}(n) + 1 \le \cCut n
\,,
\]
and it follows from \cref{lem:upper_bound_final_energy_2} and from inequality \cref{cCut_n_not_larger_than_xiHom} that the quantity 
\[
\int_{-\infty}^{\xiEsc^{(n)}(n) + 1} e^{c\xi} \Bigl(\frac{1}{2}v^{(n)}_\xi(\xi,n)^2 + V^\dag\bigl(v^{(n)}(\xi,n)\bigr) \Bigr) \, d\xi 
\]
is bounded from above by a quantity that does not depend on $n$. On the other hand, according to \vref{lem:posit_en_Esc} (involving the positive quantity $\Eesc$), 
\[
\int_{-\infty}^{\xiEsc^{(n)}(n) + 1} e^{c\xi} \Bigl(\frac{1}{2}v^{(n)}_\xi(\xi,n)^2 + V^\dag\bigl(v^{(n)}(\xi,n)\bigr) \Bigr) \, d\xi \ge \exp \bigl( c \xiEsc^{(n)}(n) \bigr) \Eesc
\,,
\]
and the conclusion follows. 
\end{proof}
\paragraph{Convergence towards zero around escape point.}
The final step is provided by the following lemma that will turn out to be contradictory to the definition of the escape point $\xesc(\cdot)$. 
\begin{lemma}[convergence towards zero around escape point]
\label{lem:approach_zero_around_esc_point}
For every positive quantity $L$, the integral
\[
\int_{\xiesc^{(n)}(n)-L}^{\xiesc^{(n)}(n)+L} \bigl( v^{(n)}_s(\xi,n)^2 + v^{(n)}_\xi(\xi,n)^2 + v^{(n)}(\xi,n)^2 \bigr) \, d\xi 
\]
goes to $0$ as $n$ goes to $+\infty$. 
\end{lemma}
\begin{proof}
Let $L$ denote a positive quantity. According to \cref{lem:boundedness_xiEsc_of_n,lem:growth_xiesc_of_n} and to inequalities \cref{upp_bound_xiEsc_cCut,cCut_n_not_larger_than_xiHom}, for every large enough positive integer $n$, the following inequalities hold:
\[
\xiEsc^{(n)}(n) \le \xiesc^{(n)}(n) - L \le \xiesc^{(n)}(n) \le \xiesc^{(n)}(n) + L \le \cCut n \le \xiHom^{(n)}(n)
\,.
\]
Then, it follows from these inequalities that
\[
\begin{aligned}
\int_{-\infty}^{\cCut n} &e^{c\xi} \Bigl(  \frac{\alpha}{2}v^{(n)}_s(\xi,n)^2 + \frac{1}{2}v^{(n)}_\xi(\xi,n)^2 + V^\dag\bigl(v^{(n)}(\xi,n)\bigr)\Bigr) \, d\xi \\
 \ge &
\int_{-\infty}^{\xiEsc^{(n)}(n)} e^{c\xi} V^\dag\bigl(v^{(n)}(\xi,n)\bigr) \, d\xi \ + \\
& \int_{\xiEsc^{(n)}(n)}^{\cCut n} e^{c\xi} \Bigl(  \frac{\alpha}{2}v^{(n)}_s(\xi,n)^2 + \frac{1}{2}v^{(n)}_\xi(\xi,n)^2 + \frac{\eigVmin(m)}{4} v^{(n)}(\xi,n)^2 \Bigr) \, d\xi \\
\ge & -\frac{\Delta_v}{c} e^{\xiEsc^{(n)}(n)}   +  \min\Bigl( \frac{\alpha}{2}, \frac{1}{2}, \frac{\eigVmin(m)}{4}\Bigr) e^{\xiEsc^{(n)}(n)-L}
\int_{\xiesc^{(n)}(n)-L}^{\xiesc^{(n)}(n)+L} \Bigl( v^{(n)}_s(\xi,n)^2\\
&  + v^{(n)}_\xi(\xi,n)^2 + v^{(n)}(\xi,n)^2 \Bigr) \, d\xi 
\,.
\end{aligned}
\]
In view of \cref{lem:upper_bound_final_energy_2,lem:boundedness_xiEsc_of_n,lem:growth_xiesc_of_n}, the conclusion follows. \Cref{lem:approach_zero_around_esc_point} is proved. 
\end{proof}
\paragraph{End of the proof.}
\begin{proof}[End of the proof of \cref{lem:further_bd_finite_speed}]
For every positive integer $n$, let us denote by $t'_n$ the time $\tInit^{(n)}+n$.
It follows from \cref{lem:approach_zero_around_esc_point} that, for every positive quantity $L$, the quantity
\[
\int_{\xesc(t'_n)-L}^{\xesc(t'_n)+L}\bigl( u_t(x,t'_n)^2 + u_x(x,t'_n)^2 + u(x,t'_n)^2 \bigr) \, dx 
\]
goes to $0$ as $n$ goes to $+\infty$. In view of the definitions of the functions $\fff_0(\cdot,\cdot)$ and $\qqq_0(\cdot,\cdot)$ in \vref{subsubsec:firewall_sf_def}, and according to the bound \vref{hyp_attr_ball_X} for the solution, it follows that, for every positive quantity $L$, both quantities
\[
\begin{aligned}
&\sup \left\{ \abs{\fff_0\bigl(\bar{x}, t'_n\bigr)}  : \bar{x}\in [\xesc(t'_n)-L, \xesc(t'_n)+L] \right\} \\
\text{and}\qquad
&\sup \left\{ \qqq_0\bigl(\bar{x}, t'_n\bigr)  : \bar{x}\in [\xesc(t'_n)-L, \xesc(t'_n)+L] \right\}
\end{aligned}
\]
go to $0$ as $n$ goes to $+\infty$, a contradiction with the definition of the ``escape'' point $\xesc(\cdot)$ in \vref{subsec:inv_cv_set_pf_cont}. \Vref{lem:further_bd_finite_speed} is proved.
\end{proof}
\subsubsection{Relaxation scheme inequality, final}
\label{subsubsec:fin_relax}
From now on the relaxation scheme will always be applied with the following choice for $\xInit$:
\[
\xInit = \xesc(\tInit)
\,.
\]
The aim of this \namecref{subsubsec:fin_relax} is to take advantage of this additional hypothesis and of the estimates of \cref{subsubsec:control_add_flux} and of \vref{lem:further_bd_finite_speed} to provide a more explicit version of the relaxation scheme inequality \vref{relax_sch_2}. 

The following additional technical hypothesis will be required to prove the next lemma providing another expression for the upper bound on $\Gback(s)$
\begin{equation}
\label{hyp_param_relax_sch_2}
\barsescsup - \frac{\kappa\cCut}{4(\cUpp + \kappa)\sqrt{1+\alpha\cUpp^2}} \le \sigma
\,.
\end{equation}
This hypothesis is satisfied as soon as the physical speed $\sigma$ is close enough to $\barsescsup$ (or equivalently as soon as the parabolic speed $c$ is close enough to $\barcescsup$). It ensures that the escape point $\xiesc(s)$ remains ``more and more far away to the left'' with respect to the position $\initCut + \cCut\, s$ of the cut-off, as $s$ increases. 
\begin{lemma}[new upper bound on $\Gback(s)$]
\label{lem:new_bound_Gback}
There exists a positive quantity $K[(u_0,\tilde{u}_0)]$, depending on $\alpha$ and $V$ and $m$ and the initial condition $(u_0,\tilde{u}_0)$, such that for every nonnegative quantity $s$ the following estimates hold: 
\begin{equation}
\label{up_bd_G_back}
\Gback(s) \le K[(u_0,\tilde{u}_0)] \exp(-\kappa \, \initCut) \exp\Big( -\frac{\kappa\, \cCut}{2}s\Bigl)
\,.
\end{equation}
\end{lemma}
\begin{proof}
According to inequality \vref{first_bd_Gback},
\begin{equation}
\label{upp_bd_Gback_prel}
\Gback(s) \le \frac{1}{\kappa} \exp(- \kappa \, \initCut) 
\exp\Bigl( (c+\kappa) \, \xiesc(s)  - \frac{\kappa\, \cCut}{2} s \Bigr) 
\exp \Bigl( - \frac{\kappa\, \cCut}{2} s \Bigr) 
\,.
\end{equation}
Let us us denote by $\beta(s)$ the argument of the second exponential of the right-hand side of this last inequality:
\[
\begin{aligned}
\beta(s) = & (c+\kappa) \, \xiesc(s)  - \frac{\kappa\, \cCut}{2} s \\
= & (c+\kappa) \Bigl( \sqrt{1+\alpha c^2} \bigl( \xesc(\tInit + s) - \xesc(\tInit) \bigr) - cs \Bigr) - \frac{\kappa\, \cCut}{2} s \\
\le & (c+\kappa) \Bigl( \sqrt{1+\alpha c^2} \barxesc(s) - cs \Bigr) - \frac{\kappa\, \cCut}{2} s \\
\le & (c+\kappa) \sqrt{1+\alpha c^2} \bigl( \barxesc(s) - \barsescsup s \bigr) \\
& + \Bigl( (c+\kappa) \bigl( \sqrt{1+\alpha c^2} \barsescsup -c \bigr)- \frac{\kappa\, \cCut}{2} \Bigr) s
\,.
\end{aligned}
\]
Besides, according to the condition \cref{hyp_param_relax_sch_2} on the ``physical'' speed $\sigma$, the following inequality holds:
\[
(c+\kappa) \bigl( \sqrt{1+\alpha c^2} \barsescsup -c \bigr) \le \frac{\kappa\, \cCut}{4}
\,,
\]
thus, for every nonnegative quantity $s$, 
\[
\beta(s) \le (c+\kappa) \sqrt{1+\alpha c^2} \bigl( \barxesc(s) - \barsescsup s \bigr) - \frac{\kappa\, \cCut}{4} s
\,,
\]
and according to the definition of $\barsescsup$ this quantity goes to $-\infty$ as $s$ goes to $+\infty$. The following (nonnegative) quantity
\[
\bar\beta[(u_0,\tilde{u}_0)] = \sup_{s\ge0} (\cUpp+\kappa)\sqrt{1+\alpha c^2} \bigl( \barxesc(s) - \barsescsup s \bigr) - \frac{\kappa\, \cCut}{4} s
\]
is an upper bound for all the values of $\beta(s)$, for all $s$ in $[0,+\infty)$. This quantity depends on $V$ and on the function $x\mapsto \barxesc(s)$, in other words on the initial condition $(u_0,\tilde{u}_0)$, but not on the parameters $\tInit$ and $c$ and $\initCut$ of the relaxation scheme. Let
\[
K[(u_0,\tilde{u}_0)] = \frac{1}{\kappa} \exp\bigl( \bar{\beta}[(u_0,\tilde{u}_0)] \bigr)
\,;
\]
with this notation, the upper bound \cref{up_bd_G_back} on $\Gback(s)$ follows from inequality \cref{upp_bd_Gback_prel}.
\end{proof}
Let us introduce the quantities
\[
K_1 = \frac{2 \alpha \Delta_V \KEF }{\nuF}
\quad\text{and}\quad
K_2 = \frac{\KEF \, \KFire}{\nuF} + \KEEsc
\]
and 
\[
\KGback [(u_0,\tilde{u}_0)] = K[(u_0,\tilde{u}_0)] \biggl( K_1 + \frac{2}{\kappa\cCut} K_2 \biggr)
\,,
\]
and, for every nonnegative quantity $s$, the quantity
\[
\KGfront(s) = \Bigl( K_1 + \frac{K_2}{(\cUpp+\kappa) (\cCut+\kappa)} \Bigr) 
\frac{1}{\kappa} \exp\bigl( (\cUpp+\kappa) (\cCut+\kappa) s \bigr)
\,.
\]
Then, for every nonnegative quantity $\sFin$, according to inequalities \cref{first_bd_Gfront} on $\Gfront(s)$ and \cref{up_bd_G_back} on $\Gback(s)$, the relaxation scheme inequality \vref{relax_sch_2} can be rewritten as
\begin{equation}
\label{relax_sch_3}
\begin{aligned}
(1+\alpha c^2) \int_0^{\sFin} \ddd(s) \, ds & \le \ \eee(0) - \eee(\sFin) \\
& + \frac{\KEF}{\nuF} \fff(0) + \KGback [(u_0,\tilde{u}_0)] \exp(-\kappa \, \initCut) \\
& + \KGfront(\sFin) \exp\bigl( (\cUpp+\kappa)\,  \initCut \bigr) \exp \bigl( -\kappa \, \xiHom(0)\bigr)
\,.
\end{aligned}
\end{equation}
This is the last version of the relaxation scheme inequality. The nice feature is that it has \emph{exactly} the same form as in the parabolic case treated in \cite{Risler_globalBehaviour_2016} (actually, the sole difference is the value of the factor in front of the integral of the left-hand side, but this detail plays absolutely no role in the arguments carried out in \cite{Risler_globalBehaviour_2016}).
\subsection{Convergence of the mean invasion speed}
\label{subsec:cv_mean_inv_vel}
The aim of this \namecref{subsec:cv_mean_inv_vel} is to prove the following proposition. 
\begin{proposition}[mean invasion speed]
\label{prop:cv_mean_inv_vel}
The following equalities hold:
\[
\sescinf = \sescsup = \barsescsup
\,.
\]
\end{proposition}
\begin{proof}
\renewcommand{\qedsymbol}{}
Let us proceed by contradiction and assume that
\[
\sescinf < \barsescsup
\,.
\]
Let us take and fix a positive quantity $\sigma$ (``physical speed'') if $c$ denotes the corresponding ``parabolic speed'' defined as
\[
c = \frac{\sigma}{\sqrt{1-\alpha\sigma^2}} 
\iff
\sigma = \frac{c}{\sqrt{1+\alpha c^2}}
\,,
\]
then the following conditions are satisfied:
\[
\sescinf < \sigma < \barsescsup \le \sigma + \frac{\kappa\cCut}{4(\cUpp + \kappa)\sqrt{1+\alpha\cUpp^2}}
\quad\text{and}\quad
\Phi_c(m) = \emptyset
\,.
\]
The first condition is satisfied as soon as $c$ is less than and close enough to $\barcescsup$, thus existence of a quantity $c$ satisfying the two conditions follows from hypothesis \textup{(\hyperlink{hypDiscVel}{\hypDiscVelRef})}.

The contradiction will follow from the relaxation scheme set up in \cref{subsec:relax_sch_tr_fr}. The main ingredient is: since the set $\Phi_c(m)$ is empty, some dissipation must occur permanently around the escape point in a referential travelling at physical speed $\sigma$. This is stated by the following lemma. 
\end{proof} 
\begin{lemma}[nonzero dissipation in the absence of travelling front]
\label{lem:dissip_no_tf_vel}
There exist positive quantities $L$ and $\epsDissip$ such that
\[
\liminf_{t\to+\infty} \int_{-1}^1 \left( \int_{-L}^L \Bigl( u_t\bigl( \xesc (t) + y,t+ s\bigr) + \sigma u_x \bigl( \xesc (t) + y, t+s\bigr) \Bigr)^2 \, dy \right)\, ds \ge \epsDissip
\,.
\]
\end{lemma}
\begin{proof}[Proof of \cref{lem:dissip_no_tf_vel}]
Let us proceed by contradiction and assume that the converse is true. Then, there exists a sequence $(t_n)_{n\in\nn^*}$ in $[1,+\infty)$ going to $+\infty$ as $n$ goes to $+\infty$ such that, for every positive integer $n$, 
\begin{equation}
\label{proof_lem_dissip}
\int_{-1}^{1} \left( \int_{-n}^n \Bigl( u_t\bigl( \xesc (t_n) + y, t_n+s\bigr) + \sigma u_x \bigl( \xesc (t_n) + y, t_n+s\bigr) \Bigr)^2 \, dy \right)\, ds \le \frac{1}{n}
\,.
\end{equation}
By compactness (\vref{prop:asympt_comp}), up to replacing the sequence $(t_n)_{n\in\nn}$ by a subsequence, it may be assumed that there exists an entire solution
\[
\bar{u}\in\ccc^0\bigl(\rr,\HulofR{1}\bigr) \cap \ccc^1(\rr,\LtwoulofR\bigr)
\]
of system \cref{hyp_syst} such that, for every positive quantity $L$, both quantities
\[
\begin{aligned}
& \sup_{s\in[-1,1]} \norm{ y\mapsto u\bigl(\xesc (t_n)+y, t_n + s\bigr) - \bar{u}(y, s)}_{H^1([-L,L],\rr^d)} \\
\text{and}\quad
& \sup_{s\in[-1,1]} \norm{ y\mapsto u_t\bigl(\xesc (t_n)+y, t_n + s\bigr) - \bar{u}_t(y, s)}_{L^2([-L,L],\rr^d)}
\end{aligned}
\]
go to $0$ as $n$ goes to $+\infty$. Let us consider the entire solution
\[
\bar{v}\in\ccc^0\bigl(\rr,\HulofR{1}\bigr) \cap \ccc^1(\rr,\LtwoulofR\bigr)
\]
of system \cref{hyp_syst_tf} defined as
\[
\bar{v}(\xi,s) = \bar{u} \left( \frac{\xi}{\sqrt{1+\alpha c^2}} + \sigma s, s\right) 
\,.
\]
It follows from inequality \cref{proof_lem_dissip} that the function $s\mapsto \bar{v}_s(\cdot,s)$ vanishes in 
\[
\ccc^0\bigl([-1,1],L^2(\rr,\rr^d)\bigr)
\]
and as a consequence the function $\phi$ defined as $\phi(\xi) = \bar{v}(\xi,0)$ is a solution of the differential system \cref{syst_trav_front} governing the profiles of waves travelling at the parabolic speed $c$ for system \cref{hyp_syst}. According to the properties of the escape point \vref{xEsc_xesc_xHom,xHom_minus_xesc}, 
\[
\sup_{\xi\in[0,+\infty)} \abs{ \phi(\xi)-m} \le \dEsc(m)
\,,
\]
thus it follows from assertion \cref{item:cv_spatial_asymptotics_tw} of \vref{lem:asympt_behav_tw_2} that $\phi(\xi)$ goes to $m$ as $\xi$ goes to $+\infty$. On the other hand, according to the bound \cref{hyp_attr_ball_Linfty} on the solution, $\abs{\phi(\cdot)}$ is bounded (by $\Rattinfty$), and since $\Phi_c(m)$ is empty, it follows from hypothesis \textup{(\hyperlink{hypOnlyBist}{\hypOnlyBistRef})} that $\phi(\cdot)$ is identically equal to $m$, a contradiction with the definition of $\xesc(\cdot)$. 
\end{proof}
The remaining of the proof of \cref{prop:cv_mean_inv_vel} is almost identical to the parabolic case treated in \cite{Risler_globalBehaviour_2016}, where more explanations and details can be found. 
The next step is the choice of the time interval and the travelling frame (at physical speed $\sigma$) where the relaxation scheme will be applied. Here is a first attempt. 
\begin{lemma}[large excursions to the right and returns for escape point in travelling frame]
\label{lem:first_attempt_time_int}
There exist sequences $(t_n)_{n\in\nn}$ and $(s_n)_{n\in\nn}$ and $(\bar{s}_n)_{n\in\nn}$ of real quantities such that the following properties hold.
\begin{enumerate}
\item For every $n$ in $\nn$, the following inequalities hold: $0\le t_n$ and $0 \le s_n \le \bar{s}_n$\,;
\item $\xesc(t_n+s_n) - \xesc(t_n) - \sigma s_n$ goes to $+\infty$ as $n$ goes to $+\infty$\,;
\item For every $n$ in $\nn$, the following inequality holds: $\xesc(t_n+\bar{s}_n) - \xesc(t_n) - \sigma \bar{s}_n \le 0$\,.
\end{enumerate}
\end{lemma}
\begin{proof}[Proof of \cref{lem:first_attempt_time_int}]
The proof is identical to that of \cite[\GlobalBehaviourLemLargeExcursionsRight]{Risler_globalBehaviour_2016}.
\end{proof}
Let $\tau$ denote a (large) positive quantity, to be chosen below. The following lemma provides appropriate time intervals to apply the relaxation scheme. 
\begin{lemma}[escape point remains to the right and ends up to the left in travelling frame, controlled duration]
\label{lem:sec_attempt_time_int}
There exist sequences $(t'_n)_{n\in\nn}$ and $(s'_n)_{n\in\nn}$ such that, for every $n$ in $\nn$, the following properties hold:
\begin{enumerate}
\item $0\le t'_n$ and $\tau \le s'_n \le 2\tau$\,,
\item for all $s$ in $[0,\tau]$, the following inequality holds: $\xesc(t'_n + s) - \xesc(t'_n) - \sigma s \ge 0$\,,
\item $\xesc(t'_n + s'_n) - \xesc(t'_n) - \sigma s'_n \le 1$\,,
\end{enumerate}
and such that
\[
t'_n\to+\infty
\quad\text{as}\quad
n\to +\infty
\,.
\]
\end{lemma}
\begin{proof}[Proof of \cref{lem:sec_attempt_time_int}]
The proof is identical to that of \cite[\GlobalBehaviourLemExcursionsRightControlledDuration]{Risler_globalBehaviour_2016}.
\end{proof}
\begin{proof}[Continuation of the proof of \cref{prop:cv_mean_inv_vel}]
\renewcommand{\qedsymbol}{}
For every $n$ in $\nn$, the relaxation scheme will be applied with the following parameters: 
\[
\tInit = t'_n
\quad\text{and}\quad
\xInit = \xesc(\tInit)
\quad\text{and}\quad
\sigma \quad\text{as chosen above, and}\quad
\initCut = 0 
\]
(the relaxation scheme thus depends on $n$). Let us denote by 
\[
v^{(n)}(\cdot,\cdot)
\quad\text{and}\quad
\eee^{(n)}(\cdot)
\quad\text{and}\quad
\ddd^{(n)}(\cdot)
\quad\text{and}\quad
\fff^{(n)}(\cdot)
\quad\text{and}\quad
{\xiesc}^{(n)}(\cdot)
\quad\text{and}\quad
{\xiHom}^{(n)}(\cdot)
\]
the objects defined in \cref{subsec:relax_sch_tr_fr} (with the same notation except the ``$(n)$'' superscript that is here to remind that all these objects depend on the integer $n$). By definition the quantity ${\xiesc}^{(n)}(0)$ equals zero, and according to the conclusions of \cref{lem:sec_attempt_time_int}, 
\[
{\xiesc}^{(n)}(s)\ge 0 \text{ for all }s\text{ in }[0,\tau]
\quad\text{and}\quad
{\xiesc}^{(n)}(s'_n) \le \sqrt{1+\alpha c^2}
\,.
\]
The following two lemmas will be shown to be in contradiction with the relaxation scheme final inequality \vref{relax_sch_3}.
\end{proof} 
\begin{lemma}[bounds on energy and firewall at the ends of relaxation scheme]
\label{lem:claim_energy}
The quantities $\eee^{(n)}(0)$ and $\fff^{(n)}(0)$ are bounded from above and the quantity $\eee^{(n)}(s'_n)$ is bounded from below, and these bounds are uniform with respect to $\tau$ and $n$. 
\end{lemma}
\begin{proof}[Proof of \cref{lem:claim_energy}]
The proof is identical to that of \cite[\GlobalBehaviourLemBoundsEnergyFirewall]{Risler_globalBehaviour_2016}.
\end{proof}
\begin{lemma}[large dissipation integral]
\label{lem:claim_dissip}
The quantity 
\[
\int_0^{s'_n} \ddd^{(n)}(s) \, ds
\]
goes to $+\infty$ as $\tau$ goes to $+\infty$, uniformly with respect to $n$. 
\end{lemma}
\begin{proof}[Proof of \cref{lem:claim_dissip}]
The proof is identical to that of \cite[\GlobalBehaviourLemLargeDissipation]{Risler_globalBehaviour_2016}.
\end{proof}
\begin{proof}[End of the proof of \cref{prop:cv_mean_inv_vel}]
\renewcommand{\qedsymbol}{}
According to \cref{lem:claim_energy}, and since ${\xiHom}^{(n)}(0)$ goes to $+\infty$ as $n$ goes to $+\infty$, the right-hand side of inequality \vref{relax_sch_3} is bounded, uniformly with respect to $\tau$, provided that $n$ (depending on $\tau$) is large enough. This is contradictory to \cref{lem:claim_dissip}, and completes the proof of \vref{prop:cv_mean_inv_vel}. 
\end{proof}
According to \cref{prop:cv_mean_inv_vel}, the three quantities $\sescinf$ and $\sescsup$ and $\barsescsup$ are equal; let
\[
\sesc
\]
denote their common value, and let us consider the corresponding ``parabolic speed'' $\cesc$ defined as
\[
\cesc = \frac{\sesc}{\sqrt{1-\alpha\sesc^2}} 
\iff
\sesc = \frac{\cesc}{\sqrt{1+\alpha\cesc^2}}
\,.
\]
\subsection{Further control on the escape point}
\label{subsec:further_control}
\begin{proposition}[mean invasion speed, further control]
\label{prop:further_control}
The following equality holds:
\[
\underbarsescinf = \sesc
\,.
\]
\end{proposition}
\begin{proof}
The proof is identical to that of \cite[\GlobalBehaviourPropInvasionSpeedFurtherControl]{Risler_globalBehaviour_2016}.
\end{proof}
\subsection{Dissipation approaches zero at regularly spaced times}
\label{subsec:dissip_zero_some_times}
For every $t$ in $[1,+\infty)$, the following set
\[
\biggl\{\varepsilon \text{ in } (0,+\infty) : \int_{-1}^1 \biggl(\int_{-1/\varepsilon}^{1/\varepsilon} \Bigl( u_t\bigl(\xesc(t)+y,t+s\bigr)+\sesc u_x\bigl(\xesc(t)+y,t+s\bigr) \Bigr)^2\, dy \biggr) \, ds \le \varepsilon \biggr\}
\]
is (according to the bound \vref{hyp_attr_ball_X} for the solution) a nonempty interval (which by the way is unbounded from above). Let 
\[
\deltaDissip(t)
\] 
denote the infimum of this interval. This quantity measures to what extent the solution is, at time $t$ and around the escape point $\xesc(t)$, close to be stationary in a frame travelling at physical speed $\sesc$. The goal is to to prove that
\[
\deltaDissip(t) \to 0
\quad\text{as}\quad
t\to +\infty
\,.
\]
\Cref{prop:dissp_zero_some_times} below can be viewed as a first step towards this goal. 
\begin{proposition}[regular occurrence of small dissipation]
\label{prop:dissp_zero_some_times}
For every positive quantity $\varepsilon$, there exists a positive quantity $T(\varepsilon)$ such that, for every $t$ in $[0,+\infty)$, 
\[
\inf_{t'\in[t,t+T(\varepsilon)]} \deltaDissip(t') \le \varepsilon
\,.
\]
\end{proposition}
\begin{proof}
The proof is identical to that of \cite[\GlobalBehaviourPropRegularSmallDissipation]{Risler_globalBehaviour_2016}.
\end{proof}
\subsection{Relaxation}
\label{subsec:relax}
\begin{proposition}[relaxation]
\label{prop:dissip_app_zero}   
The following assertion holds:
\[
\deltaDissip(t)\to0\quad\text{as}\quad
 t\to+\infty \,. 
\]
\end{proposition}
\begin{proof}
The proof is identical to that of \cite[\GlobalBehaviourPropRelaxation]{Risler_globalBehaviour_2016}.
\end{proof}
\subsection{Convergence}
\label{subsec:convergence}
The end of the proof of \vref{prop:inv_cv} (``invasion implies convergence'') is a straightforward consequence of \cref{prop:dissip_app_zero}. Let us call upon the notation $\xEsc(t)$ and $\xesc(t)$ and $\xHom(t)$ introduced in \cref{subsec:inv_cv_def_hyp,subsec:inv_cv_set_pf_cont}. Recall that, according to properties \vref{xEsc_xesc_xHom} and to the hypotheses of \cref{prop:inv_cv}, for every nonnegative time $t$, 
\[
-\infty \le \xEsc(t) \le \xesc(t) \le \xHom(t) < +\infty
\,.
\]
However, by contrast with the parabolic case treated in \cite{Risler_globalBehaviour_2016}, the point $\xEsc(t)$ cannot be used to ``track'' the position of the travelling front approached by the solution around this point, since the solution lacks the required regularity in order the function$t\mapsto\xEsc(t)$ to be of class $\ccc^1$. A convenient way to get around this difficulty is to use the decomposition of the solution into two parts, one regular, and one going to zero as time goes to $+\infty$, as stated by the following lemma (reproduced from \cite{GallayJoly_globStabDampedWaveBistable_2009}). 

Recall the notation $X$ of \vref{subsec:glob_exist_att_ball} and let
\[
Y = \HulofR{2} \times \HulofR{1}
\,,
\]
and, for every nonnegative time $t$, let $U(t) = \bigl(u(\cdot,t),u_t(\cdot,t)\bigr)$ denote the ``position / impulsion'' form of the solution. According to \vref{prop:exist_sol_att_ball},
\[
U\in\ccc^0\bigl([0,+\infty),X\bigr)
\,.
\]
\begin{lemma}[``smooth plus small'' decomposition, \cite{GallayJoly_globStabDampedWaveBistable_2009}]
\label{lem:smooth_plus_small}
There exists
\[
\Usmall \in \ccc^0\bigl([0,+\infty),X\bigr)
\quad\text{and}\quad
\Usmooth \in \ccc^1\bigl([0,+\infty),X\bigr)\cap\ccc^0\bigl([0,+\infty),Y\bigr)
\]
such that: $U$ equals $\Usmooth + \Usmall$ and
\begin{equation}
\label{asympt_Usmall}
\norm{\Usmall(t)}_X \to 0 
\quad\text{at an exponential rate as}\quad
t\to +\infty
\,,
\end{equation}
and
\begin{equation}
\label{a_priori_bound_Usmooth}
\sup_{t\ge 0} \ \norm{\Usmooth}_Y < +\infty
\,.
\end{equation}
\end{lemma}
\begin{proof}
Let 
\[
A = \frac{1}{\alpha} \begin{pmatrix}
0 & \alpha \\
\partial_x^2 - 1 & -1
\end{pmatrix}
\quad\text{and}\quad
F(u,u_t) = \frac{1}{\alpha} \begin{pmatrix}
0 \\ u - \nabla V(u)
\end{pmatrix}
\,,
\]
and let $U_0=U(0)=(u_0,\tilde{u}_0)$ denote the initial condition for the solution under consideration. Then, for every nonnegative time $t$, the following representation holds for the solution at time $t$:
\begin{equation}
\label{def_Usmooth_Usmall}
U(t) = e^{tA} U_0 + \int_0^t e^{(t-s)A}F\bigl(U(s)\bigr) \, ds 
\end{equation}
thus $\Usmall(t)$ and $\Usmooth(t)$ may be chosen as the first and the second term of the right-hand side of this equality, respectively. For more details see \cite[113]{GallayJoly_globStabDampedWaveBistable_2009}. Observe by the way that this decomposition is not unique.
\end{proof}
For every $t$ in $[0,+\infty)$, let us write
\begin{equation}
\label{def_usmooth_usmall}
\Usmooth(t) = \bigl(\usmooth(t), \partial_t \usmooth(t)\bigr)
\quad\text{and}\quad
\Usmall(t) = \bigl(\usmall(t), \partial_t \usmall(t)\bigr)
\,,
\end{equation}
and let us denote by $\xEscSmooth(t)$ the supremum of the set
\[
\bigl\{ x\in(-\infty,\xHom(t)] : \abs{\usmooth(t)} = \dEsc(m) \bigr\}
\,,
\]
with the convention that $\xEscSmooth(t)$ equals $-\infty$ if this set is empty. 
\begin{lemma}[distance between $\xEscSmooth(t)$ and $\xesc(t)$ remains bounded]
\label{lem:finiteness_xEscSmooth}
The following limit holds: 
\[
\limsup_{t\to+\infty} \xesc(t) - \xEscSmooth(t) < +\infty
\,.
\]
\end{lemma}
\begin{proof}
Let us proceed by contradiction and assume that the converse holds. Then there exists a sequence $(t_n)_{n\in\nn}$ of nonnegative times going to $+\infty$ such that
\begin{equation}
\label{finiteness_xEscSmooth_hyp_contr}
\xesc(t_n) - \xEscSmooth(t_n) \to +\infty
\quad\text{as}\quad
n\to +\infty
\,.
\end{equation}
Let us proceed as in the proof of \vref{lem:dissip_no_tf_vel}. By compactness (\vref{prop:asympt_comp}), up to replacing the sequence $(t_n)_{n\in\nn}$ by a subsequence, it may be assumed that there exists an entire solution
\[
\bar{u}\in\ccc^0\bigl(\rr,\HulofR{1}\bigr) \cap \ccc^1(\rr,\LtwoulofR\bigr)
\]
of system \cref{hyp_syst} such that, for every positive quantity $L$, both quantities
\[
\begin{aligned}
& \sup_{t\in[-1,1]} \norm{ y\mapsto u\bigl(\xesc (t_n)+y, t_n + t\bigr) - \bar{u}(y, t)}_{H^1([-L,L],\rr^d)} \,, \\
\text{and}\quad
& \sup_{t\in[-1,1]} \norm{ y\mapsto u_t\bigl(\xesc (t_n)+y, t_n + t\bigr) - \bar{u}_t(y, t)}_{L^2([-L,L],\rr^d)}
\end{aligned}
\]
go to $0$ as $n$ goes to $+\infty$. Let us consider the entire solution
\[
\bar{v}\in\ccc^0\bigl(\rr,\HulofR{1}\bigr) \cap \ccc^1(\rr,\LtwoulofR\bigr)
\]
of system \cref{hyp_syst_tf} defined as
\[
\bar{v}(\xi,s) = \bar{u} \left(\frac{\xi}{\sqrt{1+\alpha \cesc^2}} + \sesc s, s\right) 
\,.
\]
It follows from \vref{prop:dissip_app_zero} that the function $s\mapsto \bar{v}_s(\cdot,s)$ vanishes in $\ccc^0\bigl([-1,1],L^2(\rr,\rr^d)\bigr)$, and as a consequence the function $\phi$ defined as $\phi(\xi) = \bar{v}(\xi,0)$ is a solution of system \cref{syst_trav_front} for the physical speed $\cesc$, or equivalently is the profile of a wave travelling at the speed $\cesc$ for system \cref{hyp_syst}. According to the properties of the escape point \vref{xEsc_xesc_xHom,xHom_minus_xesc}, 
\[
\sup_{\xi\in[0,+\infty)} \abs{ \phi(\xi)-m} \le \dEsc(m)
\,,
\]
thus it follows from assertion \cref{item:cv_spatial_asymptotics_tw} of \vref{lem:asympt_behav_tw_2} that $\phi(\xi)$ goes to $m$ as $\xi$ goes to $+\infty$. In addition, according to the bound \cref{hyp_attr_ball_Linfty} on the solution, $\abs{\phi(\cdot)}$ is bounded (by $\Rattinfty$). In addition again, according to the definition of $\xesc(\cdot)$, the function $\phi$ cannot be identically equal to $m$. In short, the function $\phi$ belongs to the set $\Phi_{\cesc}(m)$. 

On the other hand, it follows from hypothesis \cref{finiteness_xEscSmooth_hyp_contr}, from the definition of $\xEscSmooth(\cdot)$, and from the asymptotics \cref{asympt_Usmall} for $\Usmall(\cdot)$, that 
\[
\sup_{\xi\in\rr} \abs{ \phi(\xi)-m} \le \dEsc(m)
\,,
\]
a contradiction with assertion \cref{item:transv_spatial_asymptotics_tw} of \vref{lem:asympt_behav_tw_2}. \Cref{lem:finiteness_xEscSmooth} is proved.
\end{proof}
\begin{lemma}[vicinity of Escape points and transversality]
\label{lem:vic_transv_xEscSmooth}
The following conclusions hold:
\begin{align}
\label{xEscSmooth_minus_xEsc_goes_to_zero_at_infinity}
&\xEscSmooth(t) - \xEsc(t) \to 0
\quad\text{as}\quad
t\to+\infty\,,\\
\text{and}\quad 
&\limsup_{t\to +\infty} \Bigl(\usmooth \bigl( \xEscSmooth(t) , t \bigr)-m\Bigr) \cdot \partial_x \usmooth \bigl( \xEscSmooth(t) , t \bigr) < 0
\,.
\label{transversality_at_xEscSmooth}
\end{align}
\end{lemma}
\begin{proof}
Let us proceed by contradiction and assume that it is not true that both conclusions \cref{xEscSmooth_minus_xEsc_goes_to_zero_at_infinity,transversality_at_xEscSmooth} hold. Then there exists a sequence $(t_n)_{n\in\nn}$ of nonnegative times going to $+\infty$ such that:
\begin{enumerate}
\item either $\displaystyle\limsup_{n\to+\infty} \ \abs{\xEscSmooth(t_n) - \xEsc(t_n)} > 0$,
\label{item:positive_asympt_dist_between_xEscSmooth_and_xEsc}
\item or for every positive integer $n$  
\[
\usmooth \Bigl(\bigl( \xEscSmooth(t_n) , t_n \bigr)-m\Bigr) \cdot \partial_x \usmooth \bigl( \xEscSmooth(t_n) , t_n \bigr) \ge -\frac{1}{n}
\,.
\]
\label{item:no_transversality_at_xEscSmooth}
\end{enumerate}
Proceeding as in the proof of \cref{lem:finiteness_xEscSmooth} above, and according to this lemma, it may be assumed, up to replacing the sequence $(t_n)_{n\in\nn}$ by a subsequence, that there exists a function $\phi$ in the set $\Phi_{\cesc}(m)$, such that, for every positive quantity $L$, 
\begin{equation}
\label{pf_lem_vic_transv_xEscSmooth}
\norm{x \mapsto u\bigl( \xEscSmooth(t_n) + x, t_n \bigr) - \phi\bigl( \sqrt{1+\alpha \cesc^2} x \bigr) }_{H^1([-L,L],\rr^d)}\to 0
\end{equation}
as $n$ goes to $+\infty$.
It follows from this assertion, from the definition of the quantity $\xEscSmooth(\cdot)$, and from the asymptotics \cref{asympt_Usmall} for $\Usmall(\cdot)$, that 
\[
\abs{\phi(0)-m} = \dEsc(m)
\quad\text{and}\quad
\abs{\phi(\xi)-m} \le \dEsc(m) 
\quad\text{for every positive quantity } \xi 
\,.
\]
Thus, it follows from assertion \cref{item:closer_spatial_asymptotics_tw} of \vref{lem:asympt_behav_tw_2} that
\[
\abs{\phi(\xi)-m} < \dEsc(m) 
\quad\text{for every positive quantity } \xi 
\,.
\]
In other words $\phi$ actually belongs to the set $\PhicNorm{\cesc}(m)$. Thus it follows from assertion \cref{item:transv_spatial_asymptotics_tw} of \vref{lem:asympt_behav_tw_2} that
\[
\bigl(\phi(\xi)-m\bigr)\cdot\phi'(\xi) < 0 
\quad\text{for every } \xi \text{ in } [0,+\infty)
\,,
\]
and this shows that 
\[
\lim_{n\to+\infty} \ \abs{\xEscSmooth(t_n) - \xEsc(t_n)} = 0
\,.
\]
Thus case \cref{item:positive_asympt_dist_between_xEscSmooth_and_xEsc} above cannot hold. 

On the other hand, since both $\phi(\cdot)$ and $\usmooth(\dot,\cdot)$ are of class $\ccc^1$, it follows from the limit \cref{pf_lem_vic_transv_xEscSmooth} and from the asymptotics \cref{asympt_Usmall} for $\Usmall(\cdot)$ that
\[
\Bigl(\usmooth \bigl( \xEscSmooth(t_n) , t_n \bigr)-m\Bigr) \cdot \partial_x \usmooth \bigl( \xEscSmooth(t_n) , t_n \bigr) \to \bigl(\phi(0)-m\bigr)\cdot\phi'(0)
\]
as $n$ goes to $+\infty$,
and since this limit is a negative quantity, this shows that case \cref{item:no_transversality_at_xEscSmooth} above cannot hold either, a contradiction. \Cref{lem:vic_transv_xEscSmooth} is proved.
\end{proof}
\begin{lemma}[smoothness and asymptotic speed of $\xEscSmooth(\cdot)$]
\label{lem:smooth_asympt_vel_xEscSmooth}
The function\\
$t\mapsto \xEscSmooth(t)$ is of class $\ccc^1$ on a neighbourhood of $+\infty$ and
\begin{equation}
\label{xEscSmooth_prime_goes_to_sesc}
\xEscSmooth'(t) \to \sesc
\quad\text{as}\quad
t\to+\infty
\,.
\end{equation}
\end{lemma}
\begin{proof}
Let us introduce the function
\[
f:\rr^d\times[0,+\infty) \to \rr, \quad (x,t)\mapsto \frac{1}{2}\Bigl(\bigl(\usmooth(x,t)-m\bigr)^2-\dEsc(m)^2\Bigr)
\,.
\]
According to the regularity of $\usmooth(\cdot,\cdot)$ (\vref{lem:smooth_plus_small}), this function is of class at least $\ccc^1$, and, for every large enough time $t$, the quantity $f\bigl(\xEscSmooth(t),t\bigr)$ is equal to zero, and it follows from inequality \cref{transversality_at_xEscSmooth} that
\[
\partial_x f\bigl(\xEscSmooth(t), t \bigr) = \Bigl(\usmooth \bigl(\xEscSmooth(t), t\bigr)-m\Bigr) \cdot \partial_x \usmooth \bigl(\xEscSmooth(t), t\bigr) <0
\,.
\]
Thus it follows from the Implicit Function Theorem that the function $x\mapsto \xEscSmooth(t)$ is of class (at least) a neighbourhood of $+\infty$, and that, for every large enough time $t$,
\begin{align}
\nonumber
\xEscSmooth'(t) 
&= -\frac{\partial_t f \bigl(\xEscSmooth(t),t\bigr)}{\partial_x f \bigl(\xEscSmooth(t),t\bigr)} \\
\label{xEsc_prime_of_t}
&= -\frac{\usmooth \Bigl(\bigl(\xEscSmooth(t), t\bigr)-m\Bigr) \cdot \partial_t \usmooth \bigl(\xEscSmooth(t), t\bigr)}{\Bigl(\usmooth \bigl(\xEscSmooth(t), t\bigr)-m\Bigr) \cdot \partial_x \usmooth \bigl(\xEscSmooth(t), t\bigr)}
\,.
\end{align}
According to inequality \cref{transversality_at_xEscSmooth}, the denominator of this expression remains bounded away from zero as time goes to $+\infty$. On the other hand, according to \cref{lem:finiteness_xEscSmooth} and to \vref{prop:dissip_app_zero} and to the asymptotics \cref{asympt_Usmall} for $\Usmall(\cdot)$ and to the the bounds \cref{a_priori_bound_Usmooth} on $\Usmooth(\cdot)$, 
\[
\partial_t \usmooth \bigl(\xEscSmooth(t)+\xi,t\bigr)+\sesc \partial_x  \bigl(\xEscSmooth(t)+\xi,t\bigr) \to 0
\quad\text{as}\quad t\to +\infty
\,.
\]
Thus the limit \cref{xEscSmooth_prime_goes_to_sesc} follows from expression \cref{xEsc_prime_of_t} above. \Cref{lem:smooth_asympt_vel_xEscSmooth} is proved.
\end{proof}
The next lemma is the only place throughout the proof of \cref{prop:inv_cv} where hypothesis \textup{(\hyperlink{hypDiscFront}{\hypDiscFrontRef})} --- which is part of the generic hypotheses \cref{hyp_gen} --- is required. 
\begin{lemma}[convergence around Escape point]
\label{lem:cv_around_Esc_point}
There exists a function $\phi$ in the set $\PhicNorm{\cesc}(m)$ such that, for every positive quantity $L$, both quantities 
\begin{equation}
\label{quantities_going_to_zero_convergence_towards_trav_front}
\begin{aligned}
&\norm{x\mapsto u(x,t)-\phi\left(\sqrt{1+\alpha\cesc^2}\bigl(x-\xEscSmooth(t)\bigr)\right)}_{H^1\bigl([\xEsc(t)-L,\xEsc(t)+L],\rr^d\bigr)} \,,\quad \text{and}\\
&\norm{x\mapsto u_t(x,t)+\cesc\phi'\left(\sqrt{1+\alpha\cesc^2}\bigl(x-\xEscSmooth(t)\bigr)\right)}_{L^2\bigl([\xEsc(t)-L,\xEsc(t)+L],\rr^d\bigr)}
\end{aligned}
\end{equation}
go to $0$ as time goes to $+\infty$. 
In particular, the set $\PhicNorm{\cesc}(m)$ is nonempty.
\end{lemma}
\begin{proof}
Take a sequence $(t_n)_{n\in\nn}$ of positive times going to $+\infty$ as $n$ goes to $+\infty$. Proceeding as in the proof of \cref{lem:finiteness_xEscSmooth} above, and according to this lemma, it may be assumed, up to replacing the sequence $(t_p)_{n\in\nn}$ by a subsequence, that there exists a function $\phi$ in the set $\Phi_{\cesc}(m)$ such that, for every positive quantity $L$, both quantities
\[
\begin{aligned}
&\norm{y\mapsto u\bigl( \xEscSmooth(t_n) + y, t_n \bigr) - \phi\left( \sqrt{1+\alpha \cesc^2} y \right) }_{H^1\bigl([-L,L],\rr^d\bigr)} \text{ and} \\
&\norm{y\mapsto u_t\bigl(\xEscSmooth(t_n) + y,t\bigr)+\cesc\phi'\left(\sqrt{1+\alpha\cesc^2}y\right)}_{L^2\bigl([-L,L],\rr^d\bigr)}
\end{aligned}
\]
go to $0$ as $n$ goes to $+\infty$. 
According to the definition of $\xEscSmooth(\cdot)$ and to the asymptotics \cref{asympt_Usmall} for $\Usmall(\cdot)$, it follows that
\[
\abs{\phi(0)-m} = \dEsc(m)
\quad\text{and}\quad
\abs{\phi(\xi)-m}\le \dEsc(m)
\quad\text{for all}\quad
\xi \text{ in } [0,+\infty)
\,,
\]
thus, according to assertion \cref{item:transv_spatial_asymptotics_tw} of \vref{lem:asympt_behav_tw_2}, it follows that $\phi$ actually belongs to the set $\PhicNorm{\cesc}(m)$. 

Let $\mathcal{L}$ denote the set of all possible limits (in the sense of uniform convergence on compact subsets of $\rr$) of sequences of maps
\[
y\mapsto u \bigl( \xEscSmooth(t'_n) + y, t'_n\bigr)
\]
for all possible sequences $(t'_n)_{n\in\nn}$ such that $t'_n$ goes to $+\infty$ as $n$ goes to $+\infty$. This set $\mathcal{L}$ is included in the set $\PhicNorm{\cesc}(m)$, and, because the semi-flow of system \cref{hyp_syst} is continuous on $X$, this set $\mathcal{L}$ is a continuum (a compact connected subset) of $\HulofR{1}$. 

Since on the other hand --- according to hypothesis \textup{(\hyperlink{hypDiscFront}{\hypDiscFrontRef})} --- the set $\PhicNorm{\cesc}(m)$ is totally disconnected in $\HulofR{1}$, this set $\mathcal{L}$ must actually be reduced to the singleton $\{\phi\}$. \Cref{lem:cv_around_Esc_point} is proved. 
\end{proof}
\begin{lemma}[convergence up to $\xHom(t)$]
\label{lem:cv_right_of_front}
For every positive quantity $L$, 
\[
\sup_{x\in[\xEsc(t)-L,\xHom(t)]} \abs{u(x,t)-\phi\left(\sqrt{1+\alpha\cesc^2}\bigl(x-\xEscSmooth(t)\bigr)\right)} \to 0 
\quad\text{as}\quad 
t\to+\infty 
\,.
\]
\end{lemma}
\begin{proof}
The proof is identical to the proof of \cite[\GlobalBehaviourLemConvergenceUpToxHom]{Risler_globalBehaviour_2016}.
\end{proof}
\subsection{Homogeneous point behind the travelling front}
\label{subsec:hom_behind_front}
According to hypothesis \textup{(\hyperlink{hypOnlyBist}{\hypOnlyBistRef})}, the limit
\[
\lim_{\xi\to-\infty}\phi(\xi)
\]
exists and belongs to $\mmm$; let us denote by $\mNext$ this limit. The following lemma completes the proof of \cref{prop:inv_cv} (``invasion implies convergence''). 
\begin{lemma}[``next'' homogeneous point behind the front]
\label{lem:next_hom_point}
There exists a $\rr$-valued function $\xHomNext$, defined and of class $\ccc^1$ on a neighbourhood of $+\infty$, such that the following limits hold as time goes to $+\infty$:
\[
\begin{aligned}
& \xEsc(t) - \xHomNext(t) \to +\infty
\quad\text{and}\quad
\xHomNext'(t) \to \sesc \\
\text{and}\quad
& \sup_{x\in[\xHomNext(t),\xHom(t)]} \abs{u(x,t)-\phi\left(\sqrt{1+\alpha\cesc^2}\bigl(x-\xEscSmooth(t)\bigr)\right)} \to 0 
\,,
\end{aligned}
\]
and, for every positive quantity $L$, 
\begin{equation}
\label{convergence_towards_mNext_in_Hone_times_Ltwo}
\begin{aligned}
&\norm{y\mapsto u \bigl(\xHomNext(t) + y, t\bigr) - \mNext}_{H^1\left([-L,L],\rr^d\right)} \to 0 \,, \\
\text{and}\quad
&\norm{y\mapsto u_t \bigl( \xHomNext(t) + y, t \bigr)}_{L^2\left([-L,L],\rr^d\right)} \to 0
\,.
\end{aligned}
\end{equation}
\end{lemma}
\begin{proof}
The proof is identical to the proof of \cite[\GlobalBehaviourLemNextHomogeneousPoint]{Risler_globalBehaviour_2016}. The convergence toward $0$ of the quantities \cref{quantities_going_to_zero_convergence_towards_trav_front} yields the limits \cref{convergence_towards_mNext_in_Hone_times_Ltwo}. 
\end{proof}
This completes the proof of conclusion \cref{item:invasion_implies_convergence_main_conclusion} of \cref{prop:inv_cv}. 
\Cref{prop:inv_cv} is proved.
\section{No invasion implies relaxation}
\label{sec:no_inv_implies_relax}
As everywhere else, let us consider a function $V$ in $\ccc^2(\rr^d,\rr)$ satisfying the coercivity hypothesis \cref{hyp_coerc}. The aim of this \namecref{sec:no_inv_implies_relax} is to prove \cref{prop:relax} below. The arguments are similar to those of \cite[\GlobalBehaviourSecNonInvasionImpliesRelaxation]{Risler_globalBehaviour_2016}, where more details and comments can be found. 
\subsection{Definitions and hypotheses}
\label{subsec:def_hyp_dichot}
Let us consider two points $m_-$ and $m_+$ in $\mmm$ and a solution $(x,t)\mapsto u(x,t)$ of system \cref{hyp_syst} defined on $\rr\times[0,+\infty)$. Without assuming that this solution is bistable, let us make the following hypothesis \textup{(\hyperlink{hypHom}{\hypHomRef})}, which is similar to hypothesis \textup{(\hyperlink{hypHomRight}{\hypHomRightRef})} made in \cref{sec:inv_impl_cv} (``invasion implies convergence''), but this time both to the right and to the left in space (see \cref{fig:inv_relax_dichot}).
\begin{figure}[!htbp]
	\centering
    \includegraphics[width=\textwidth]{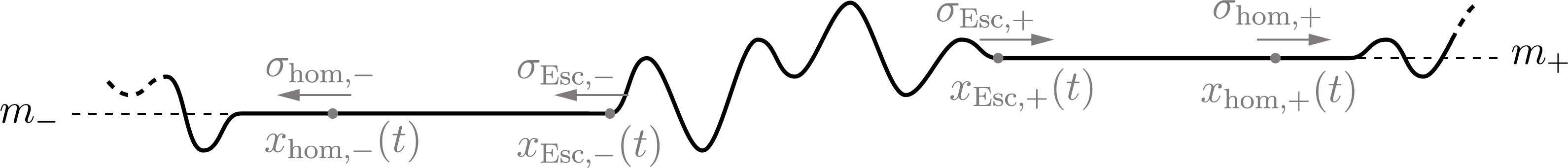}
    \caption{Illustration of hypothesis \textup{(\hypHomRef)} and of \cref{prop:relax}.}
    \label{fig:inv_relax_dichot}
\end{figure}
\begin{description}
\item[\hypHomLabel]\hypertarget{hypHom} There exist a positive quantity $\sHomPlus$ and a negative quantity $\sHomMinus$ and $\ccc^1$-functions
\[
\begin{aligned}
& \xHomPlus:[0,+\infty)\to\rr
\quad\text{satisfying}\quad
\xHomPlus'(t)\to \sHomPlus 
\quad\text{as}\quad
t\to +\infty \\
\text{and} \quad
& \xHomMinus:[0,+\infty)\to\rr
\quad\text{satisfying}\quad
\xHomMinus'(t)\to \sHomMinus 
\quad\text{as}\quad
t\to +\infty 
\end{aligned}
\]
such that, for every positive quantity $L$, both quantities
\[
\begin{aligned}
&\norm{y\mapsto \Bigl( u \bigl( \xHomPlus(t) + y, t \bigr) - m_+ , u_t \bigl( \xHomPlus(t) + y, t \bigr) \Bigr) }_{H^1([-L,L])\times L^2([-L,L])}  \\
\text{and}\quad
&\norm{y\mapsto \Bigl( u \bigl( \xHomMinus(t) + y, t \bigr) - m_- , u_t \bigl( \xHomMinus(t) + y, t \bigr) \Bigr)}_{H^1([-L,L])\times L^2([-L,L])} 
\end{aligned}
\]
go to $0$ as time goes to $+\infty$. 
\end{description}
For every $t$ in $[0,+\infty)$, let us denote by $\xEscPlus(t)$ the supremum of the set
\[
\bigl\{ x\in\rr : \xHomMinus(t) \le x \le \xHomPlus(t) \quad\text{and}\quad \abs{u(x,t)-m_+}=\dEsc(m_+) \bigr\}
\]
(with the convention that $\xEscPlus(t)$ equals $-\infty$ if this set is empty), and let us denote by $\xEscMinus(t)$ the infimum of the set
\[
\bigl\{ x\in\rr : \xHomMinus(t) \le x \le \xHomPlus(t) \quad\text{and}\quad \abs{u(x,t)-m_-}=\dEsc(m_-) \bigr\}
\]
(with the convention that $\xEscMinus(t)$ equals $+\infty$ if this set is empty). Let 
\[
\sEscPlus = \limsup_{t\to+\infty} \frac{\xEscPlus(t)}{t}
\quad\text{and}\quad
\sEscMinus = \liminf_{t\to+\infty} \frac{\xEscPlus(t)}{t}
\,,
\]
see \cref{fig:inv_relax_dichot}. It follows from the definitions of $\xEscMinus$ and $\xEscPlus(t)$ that, for all $t$ in $[0,+\infty)$, 
\[
\xHomMinus(t) \le \xEscMinus(t)
\quad\text{and}\quad
\xEscPlus(t) \le \xHomPlus(t)
\]
thus
\[
\sHomMinus \le \sEscMinus
\quad\text{and}\quad
\sEscPlus \le \sHomPlus
\,.
\]
If the quantity $\sEscPlus$ was positive or if the quantity $\sEscMinus$ was negative, this would mean that the corresponding equilibrium is ``invaded'' at a nonzero mean speed, a situation already studied in \cref{sec:inv_impl_cv}. Let us introduce the following (converse) ``no invasion'' hypothesis.
\begin{description}
\item[\hypNoInvLabel]\hypertarget{hypNoInv} The following inequalities hold:
\[
\sEscMinus \ge 0 
\quad\text{and}\quad
\sEscPlus \le 0 
\,.
\]
\end{description}
\subsection{Statement}
\label{subsec:statement_dichot}
The aim of \cref{sec:no_inv_implies_relax} is to prove the following proposition. 
\begin{proposition}[no invasion implies relaxation]
\label{prop:relax}
Assume that $V$ satisfies hypothesis \\ \cref{hyp_coerc} and that the solution $(x,t)\mapsto u(x,t)$ under consideration satisfies hypotheses \textup{(\hyperlink{hypHom}{\hypHomRef})} and \textup{(\hyperlink{hypNoInv}{\hypNoInvRef})}. Then the following conclusions hold. 
\begin{enumerate}
\item The quantities $V(m_-)$ and $V(m_+)$ are equal.
\item There exists a nonnegative quantity $\eeeResAsympt[u]$ (``residual asymptotic energy'') such that, for all quantities $\sigma_-$ in $(\sHomMinus,0)$ and $\sigma_+$ in $(0,\sHomPlus)$,
\begin{equation}
\label{asympt_energy_prop_relax}
\int_{\sigma_- t}^{\sigma_+ t} \Bigl[ \frac{\alpha}{2}u_t(x,t)^2 +\frac{1}{2}u_x(x,t)^2 + V\bigl( u(x,t)\bigl) - V(m_\pm) \Bigr] \, dx \to \eeeResAsympt[u]
\end{equation}
as time goes to $+\infty$. 
\item The following limits hold as time goes to $+\infty$:
\begin{equation}
\label{dissip_goes_to_zero_prop_relax}
\sup_{x\in [\xHomMinus(t)\, ,\, \xHomPlus(t)]} \int_{x-1}^{x+1} u_t(z,t)^2 \, dz \to 0 
\,,
\end{equation}
and, for every quantity $\varepsilon$ which is positive and smaller than $\abs{\sHomMinus}$ and than $\sHomPlus$,
\begin{equation}
\label{sol_goes_to_m_plus_minus_prop_relax}
\sup_{x\in[\xHomMinus(t),-\varepsilon t]}\abs{u(x,t)-m_-} \to 0
\quad\text{and}\quad
\sup_{x\in[\varepsilon t,\xHomPlus(t)]}\abs{u(x,t)-m_+} \to 0
\,.
\end{equation}
\end{enumerate}
\end{proposition}
\subsection{Relaxation scheme in a standing or almost standing frame}
\label{subsec:relax_sc_stand}
\subsubsection{Principle}
The aim of this \namecref{subsec:relax_sc_stand} is to set up an appropriate relaxation scheme in a standing or almost standing frame. This means defining an appropriate localized energy and controlling the ``flux'' terms occurring in the time derivative of that localized energy. The argument will be quite similar to that of \vref{subsec:relax_sch_tr_fr} (the relaxation scheme in the travelling frame), the main difference being that the speed of the travelling frame will now be either equal or close to zero, and as a consequence the weight function for the localized energy will be defined with a cut-off on the right and another on the left, instead of a single one; accordingly firewall functions will be introduced to control the fluxes along each of these cuts-off.

Let us keep the notation and hypotheses of \cref{subsec:def_hyp_dichot}, and let us assume that hypotheses \cref{hyp_coerc} and \textup{(\hyperlink{hypHom}{\hypHomRef})} and \textup{(\hyperlink{hypNoInv}{\hypNoInvRef})} of \cref{prop:relax} hold. According to \vref{prop:exist_sol_att_ball}, it may be assumed (without loss of generality, up to changing the origin of times) that, for all $t$ in $[0,+\infty)$, 
\begin{align}
\label{hyp_attr_ball_Linfty_relax}
\norm{x\mapsto u(x,t)}_{\Linfty} &\le \Rattinfty \\
\text{and}\qquad
\label{hyp_attr_ball_X_relax}
\norm{x\mapsto \bigl(u(x,t),u_t(x,t)\bigr)}_X &\le \RattX
\,.
\end{align} 
\subsubsection{Notation for the travelling frame}
As in \vref{subsec:relax_sch_tr_fr}, let us introduce as parameters the ``parabolic'' speed $c$ of the travelling frame and its ``physical'' speed $\sigma$ related by
\[
\sigma=\frac{c}{\sqrt{1+\alpha c^2}}
\iff 
c = \frac{\sigma}{\sqrt{1-\alpha \sigma^2}}
\,.
\]
To simplify the notation (that is, to avoid writing absolute values), let us assume that these speeds are nonnegative, namely:
\[
c\ge 0 \,,
\quad\text{or equivalently}\quad
\sigma\ge 0
\,.
\]
By contrast with \cref{subsec:relax_sch_tr_fr}, the other parameters --- namely $\tInit$ and $\xInit$ and $\initCut$ --- are not be required here. 
The relaxation scheme will be applied in the next \cref{subsec:low_bd_en} for a speed $c$ very close or equal to zero. 

Let us introduce the function $(\xi,t)\mapsto v(\xi,t)$, defined for every real quantity $\xi$ and every nonnegative time $t$ by
\[
v(\xi,t) = u(x,t)
\]
where $x$ and $\xi$ are related by
\[
x = \sigma t + \frac{\xi}{\sqrt{1+\alpha c^2}}
\iff
\xi = \sqrt{1+\alpha c^2} x - ct 
\,.
\]
The evolution system for the function $(\xi,t)\mapsto v(\xi,t)$ reads
\[
\alpha v_{tt} + v_t - 2 \alpha c v_{\xi t} = - \nabla V(v) + c v_\xi + v_{\xi\xi}
\,.
\]
\subsubsection{Choice of the parameters and conditions on the speed \texorpdfstring{$c$}{c}}
A localized energy and two firewall functions associated with this solution will now be introduced. Let us denote by $\kappa_0(m_-)$ and by $\kappa_0(m_+)$ the quantities defined in \vref{def_kappaZero} for the two points $m_-$ and $m_+$, and let 
\[
\kappa_0 = \min\bigl(\kappa_0(m_-),\kappa_0(m_+)\bigr)
\quad\text{and}\quad
\eigVmin = \min\bigl(\eigVmin(m_-),\eigVmin(m_+)\bigr)
\,.
\]
Let
\begin{equation}
\label{def_cCutZero}
\cCutZero = \min\biggl( \frac{\sHomPlus}{2}, \frac{\abs{\sHomMinus}}{2}, \frac{1}{4\alpha+2},\frac{\eigVmin}{8\kappa_0\bigl(1+\alpha(\kappa_0+1)\bigr)}\biggr)
\,,
\end{equation}
and let us assume that the (nonnegative) quantity $c$ is small enough so that the following inequalities be satisfied:
\begin{equation}
\label{hyp_param_relax}
c\le\frac{\kappa_0}{6}
\quad\text{and}\quad
c\le\frac{1}{\sqrt{\alpha}}
\quad\text{and}\quad
c\le\frac{\cCutZero}{6}
\,,
\end{equation}
and
\begin{equation}
\label{hyp_c_for_lower_bd_fireZero}
\alpha c(\kappa_0 + c) \le \frac{1}{6}
\,.
\end{equation}
According to \textup{(\hyperlink{hypHom}{\hypHomRef})} and \textup{(\hyperlink{hypNoInv}{\hypNoInvRef})} and to the choice of $\cCutZero$ above, there exists a nonnegative time $T$ such that, for every time $t$ greater than or equal to $T$,  
\begin{equation}
\label{hyp_pos_relax}
\begin{aligned}
\xHomMinus(t)&\le -\frac{11}{6}\cCutZero t
&\quad&\text{and}\quad&
-\frac{1}{6\sqrt{2}}\cCutZero t &\le \xEscMinus(t) \\
\text{and}\quad
\xEscPlus(t)&\le \frac{1}{6\sqrt{2}}\cCutZero t 
&\quad&\text{and}\quad&
\frac{11}{6}\cCutZero t &\le \xHomPlus(t)
\,.
\end{aligned}
\end{equation}
\subsubsection{Notation ``\texorpdfstring{$\pm$}{plus or minus}''}
Let us adopt, for the remaining of this \cref{sec:no_inv_implies_relax} and in the next \cref{sec:convergence}, the following convention: the symbol ``$\pm$'' denotes one the the signs ``$+$'' and ``$-$'', this sign remaining the same along a whole expression, an equality/inequality between two expressions, or a definition. 
\subsubsection{Normalized potential}
Let us introduce the ``normalized'' potential $V^\ddag: \rr^d\to\rr$, $v\mapsto V^\ddag(v)$ defined as
\begin{equation}
\label{def_V_ddag}
V^\ddag (v) = V(v) - \max\bigl(V(m_-),V(m_+)\bigr)
\,.
\end{equation}
As a consequence $\max\bigl(V^\ddag(m_-),V^\ddag(m_+)\bigr) = 0$, and $\nabla V$ and $\nabla V^\ddag$ are equal. With the convention above, it follows from inequalities \cref{posit_pot_around_loc_min,v_nablaV_controls_square_around_loc_min,v_nablaV_controls_pot_around_loc_min} that, for all $v$ in $\rr^d$ satisfying $\abs{v-m_\pm}\le\dEsc(m_\pm)$, 
\begin{align}
\label{v_nablaV_controls_square_around_loc_min_ddag}
(v-m_\pm)\cdot \nabla V^\ddag(v) &\ge \frac{\eigVmin}{2} (v-m_\pm)^2 \,, \\
\text{and}\qquad
\label{v_nablaV_controls_pot_around_loc_min_ddag}
(v-m_\pm)\cdot \nabla V^\ddag(v) &\ge V^\ddag(v) - V^\ddag(m_\pm)
\,.
\end{align}
\subsubsection{Localized energy}
\label{subsubsec:en_dichot}
For every time $t$, let us introduce the three intervals
\[
\begin{aligned}
\iLeft(t) &= ( - \infty , -\cCutZero t ] \,, \\
\text{and}\qquad
\iMain(t) &= [-\cCutZero t, \cCutZero t] \,, \\
\text{and}\qquad
\iRight(t) &= [ \cCutZero t , +\infty) \,,
\end{aligned}
\]
and let us introduce the functions $\chi_0(\xi,t)$ and $\chi(\xi,t)$ (weight function for the localized energy) defined on $\rr\times[0,+\infty)$ as
\[
\chi_0(\xi,t)=
\left\{
\begin{aligned}
&1
&& \quad\text{if}\quad \xi\in \iMain(t) \,, \\
& \exp\bigl(-\kappa_0 (\abs{\xi}-\cCutZero t)\bigr) 
&& \quad\text{if}\quad \xi\not\in \iMain(t) \,,
\end{aligned}
\right.
\quad\text{and}\quad
\chi(\xi,t) = e^{c\xi} \chi_0(\xi,t)
\,,
\]
see \cref{fig:graph_weight_zero_dichot,fig:graph_weight_dichot}.
\begin{figure}[!htbp]
	\centering
    \includegraphics[width=.8\textwidth]{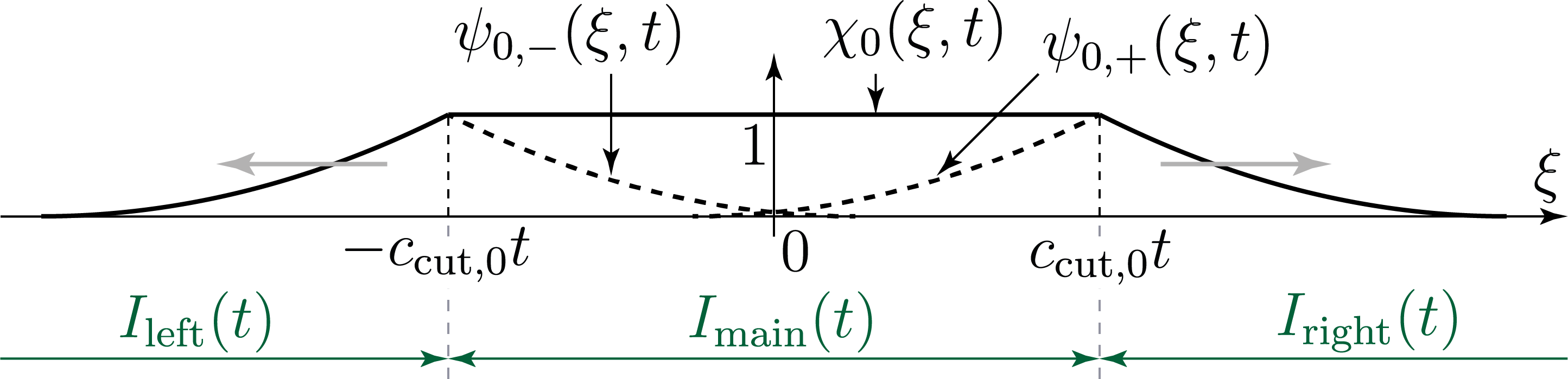}
    \caption{Graphs of functions $\xi\mapsto\chi_0(\xi,t)$ and $\xi\mapsto\psi_{0,+}(\xi,t)$ and $\xi\mapsto\psi_{0,-}(\xi,t)$.}
    \label{fig:graph_weight_zero_dichot}
\end{figure}
\begin{figure}[!htbp]
	\centering
    \includegraphics[width=\textwidth]{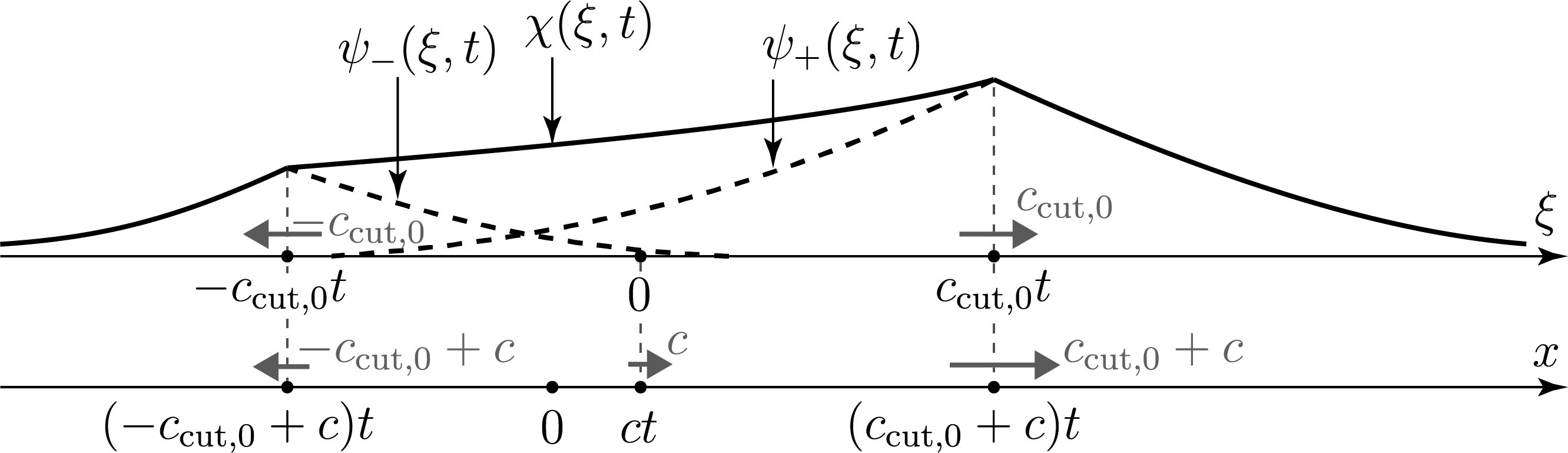}
    \caption{Graphs of the weight functions $\xi\mapsto\chi(\xi,t)$ and $\xi\mapsto\psi_+(\xi,t)$ and $\xi\mapsto\psi_-(\xi,t)$.}
    \label{fig:graph_weight_dichot}
\end{figure}
For all $t$ in $[0,+\infty)$, let us define the ``energy'' $\eee(t)$ by
\[
\eee(t) = \int_{\rr} \chi(\xi,t) E^\ddag(\xi,t) \, d\xi \,,
\quad\text{where}\quad
E^\ddag(\xi,t) = \frac{\alpha}{2}v_t(\xi,t)^2 + \frac{1}{2}v_\xi(\xi,t)^2 + V^\ddag\bigl( v(\xi,t)\bigr)
\,.
\]
The notation $\chi$ and $\eee$ is the same as in \cref{subsubsec:def_loc_en} but the definitions above are slightly different from those introduced in \cref{subsubsec:def_loc_en}. 
\subsubsection{Time derivative of the localized energy}
\label{subsubsec:time_der_loc_en_dichot}
For every nonnegative quantity $t$, let
\begin{equation}
\label{def_dissip_dichot}
\ddd(t) = \int_{\rr} \chi(\xi,t)\, v_t(\xi,t)^2 \, d\xi
\,.
\end{equation}
\begin{lemma}[time derivative of the localized energy]
\label{lem:approx_decr_en_one_dichot}
For every nonnegative time $t$, 
\begin{equation}
\label{upp_bd_dE_tf_dichot}
\begin{aligned}
\eee'(t) \le &-(1+\alpha c^2) \ddd(t) \\
+ \kappa_0&\int_{\rr\setminus\iMain(t)}\chi \biggl[ \frac{\alpha(\cCutZero+2c)+1}{2} v_t^2 + \frac{\cCutZero+1}{2} v_\xi^2 + \cCutZero V^\ddag(v) \biggr] \, d\xi
\,.
\end{aligned}
\end{equation}
\end{lemma}
\begin{proof}
It follows from from expression \vref{ddt_loc_en_tf_second} (time derivative of a localized energy) that for all $t$ in $[0,+\infty)$,
\begin{equation}
\label{dE_preliminary_relax}
\eee'(t)  = - (1+\alpha c^2) \ddd(t) + \int_{\rr}\biggl( \chi_t\Bigl(\frac{\alpha}{2}v_t^2 + \frac{1}{2}v_\xi^2 + V^\ddag(v)\Bigr) + (c\chi-\chi_\xi)(\alpha c v_t^2 + v_\xi\cdot v_t) \biggr) \, d\xi
\,.
\end{equation}
It follows from the definition of $\chi$ that
\[
\chi_t(\xi,t)=e^{c\xi}\partial_t \chi_0(\xi,t) = 
\left\{
\begin{aligned}
& 0 
&& \quad\text{if}\quad \xi\in\iMain(t) \,, \\
&\kappa_0\cCutZero \chi(\xi,t) 
&& \quad\text{if}\quad \xi\not\in\iMain(t) \,, 
\end{aligned}
\right.
\]
and 
\[
(c\chi - \chi_\xi) (\xi,t)= - e^{c\xi} \partial_\xi \chi_0 (\xi,t) = 
\left\{
\begin{aligned}
&0 
&& \quad\text{if}\quad \xi\in\iMain(t) \,, \\
&\sgn(\xi)\, \kappa_0\, \chi(\xi,t) 
&& \quad\text{if}\quad \xi\not\in\iMain(t) \,.
\end{aligned}
\right.
\]
Thus it follows from \cref{dE_preliminary_relax} that, for all $t$ in $[0,+\infty)$,
\[
\begin{aligned}
\eee'(t)  = &- (1+\alpha c^2) \ddd(t) \\
&+ \kappa_0\int_{\rr\setminus\iMain(t)}\chi \biggl(\cCutZero \Bigl(\frac{\alpha}{2}v_t^2 + \frac{1}{2}v_\xi^2 + V^\ddag(v)\Bigr) + \sgn(\xi) (\alpha c v_t^2 + v_\xi\cdot v_t) \biggr) \, d\xi 
\,,
\end{aligned}
\]
and using the inequality
\[
\sgn(\xi) v_\xi\cdot v_t \le \frac{1}{2}v_\xi^2 + \frac{1}{2}v_t^2
\,,
\]
inequality \cref{upp_bd_dE_tf_dichot} follows. \Cref{lem:approx_decr_en_one_dichot} is proved. 
\end{proof}
\subsubsection{Firewall functions}
\label{subsubsec:def_firewall_dichot}
Proceeding as in \vref{subsubsec:def_firewall} two firewall functions will be introduced in order to control the right-hand side of this inequality. 
Let us introduce the functions $\psi_{0,+}(\xi,t)$ and $\psi_{0,-}(\xi,t)$ and $\psi_+(\xi,t)$ and $\psi_-(\xi,t)$ (weight functions for those firewall functions) defined as
\[
\begin{aligned}
\psi_{0,-}(\xi,t) &= \exp\bigl(-\kappa_0\abs{\xi+\cCutZero t}\bigr)
\,,
\\
\text{and}\quad
\psi_{0,+}(\xi,t) &= \exp\bigl(-\kappa_0\abs{\xi-\cCutZero t}\bigr)
\,,
\end{aligned}
\]
and 
\[
\psi_{-}(\xi,t) = e^{c\xi}\psi_{0,-}(\xi,t)
\quad\text{and}\quad
\psi_{+}(\xi,t) = e^{c\xi}\psi_{0,+}(\xi,t)
\,,
\]
see \cref{fig:graph_weight_zero_dichot,fig:graph_weight_dichot}. Observe that 
\[
\chi(\xi,t) = \psi_{-}(\xi,t)
\quad\text{for}\quad
\xi\in\iLeft(t)
\quad\text{and}\quad
\chi(\xi,t) = \psi_{+}(\xi,t)
\quad\text{for}\quad
\xi\in\iRight(t)
\,.
\]
For every nonnegative time $t$, let 
\[
\fff_\pm(t) = \int_{\rr} \psi_\pm(\xi,t) \, F^\ddag_\pm(\xi,t) \, d\xi 
\,,
\]
where
\begin{equation}
\label{def_F_plus_minus}
\begin{aligned}
F^\ddag_\pm(\xi,t) &= 2\alpha \bigl(E^\ddag(\xi,t) - V^\ddag(m_\pm)\bigr) \\
&\qquad + \left(\alpha \bigl(v-m_\pm\bigr)\cdot v_t + \Bigl( \frac{1}{2} + \alpha c \frac{\partial_\xi\psi_\pm}{\psi_\pm} \Bigr) (v-m_\pm)^2\right)(\xi,t) \\
&= \biggl(\alpha^2 v_t^2 + \alpha v_\xi^2 + 2\alpha\bigl( V^\ddag(v)- V^\ddag(m_\pm) \bigr) + \alpha (v-m_\pm)\cdot v_t \\
&\qquad + \Bigl( \frac{1}{2} + \alpha c \frac{\partial_\xi\psi_\pm}{\psi_\pm} \Bigr) (v-m_\pm)^2 \biggr) (\xi,t)
\,.
\end{aligned}
\end{equation}
\subsubsection{Lower bounds on the firewall functions}
\begin{lemma}[lower bounds on the firewall functions]
\label{lem:lower_bound_F_relax}
For every nonnegative quantity $t$,
\begin{equation}
\label{lower_bound_F_relax}
\fff_\pm(t) \ge \int_{\rr}\psi_\pm(\xi,t)\Bigl[ \frac{\alpha^2}{4} v_t(\xi,t)^2 + \alpha v_\xi(\xi,t)^2 + 2 \alpha \Bigl(V^\ddag\bigl(v(\xi,t)\bigr)-V^\ddag(m_\pm)\Bigr) \Bigr] \, d\xi 
\,.
\end{equation}
\end{lemma}
\begin{proof}
Observe that
\[
\partial_\xi \psi_\pm = c \psi_\pm + e^{c\xi}\partial_\xi \psi_{0,\pm}
\quad\text{thus}\quad
\abs{\partial_\xi \psi_\pm} \le (\kappa_0 + c) \psi_\pm
\,.
\]
As a consequence, it follows from the polarization inequality \vref{polarization_inequality} that, for every real quantity $\xi$ and every nonnegative quantity $t$, 
\[
F^\ddag_\pm(\xi,t) \ge \frac{\alpha^2}{4} v_t^2 + \alpha v_\xi^2 + 2 \alpha \bigl(V^\dag(v)-V^\dag(m_\pm)+ \Bigl( \frac{1}{6} - \alpha c(\kappa_0 + c)\Bigr) (v-m_\pm)^2 
\,,
\]
thus inequality \cref{lower_bound_F_relax} follows from condition \vref{hyp_c_for_lower_bd_fireZero} satisfied by $c$. 
\end{proof}
\subsubsection{Energy decrease up to firewalls and pollution}
For every nonnegative time $t$, let 
\[
\SigmaEscPlusMinus(t) = \{\xi\in\rr : \abs{v(\xi,t)-m_\pm} > \dEsc(m_\pm) \} 
\,,
\]
and let
\begin{equation}
\label{def_gggg_pm}
\mathcal{G}_\pm(t) = \int_{\SigmaEscPlusMinus(t)} \psi_\pm (\xi,t) \, d\xi 
\,.
\end{equation}
\begin{lemma}[energy decrease up to firewalls and pollution]
\label{lem:approx_decr_en_dichot}
There exist nonnegative quantities $\KEFZero$ and $\KEEscZero$, depending on $\alpha$ and $V$ and $m_+$ and $m_-$ (only), such that for  every nonnegative time $t$, 
\begin{equation}
\label{approx_decr_en_dichot}
\begin{aligned}
\eee'(t) \le & - (1+\alpha c^2) \ddd(t) + \KEFZero \bigl( \fff_+(t) + \fff_-(t) \bigr) \\
& + \KEEscZero \bigl( \mathcal{G}_-(t) + \mathcal{G}_+(t)\bigr) 
\,.
\end{aligned}
\end{equation}
\end{lemma}
\begin{proof}
For every nonnegative time $t$, since $\chi(\xi,t)=\psi_-(\xi,t)$ for all $\xi$ in $\iLeft(t)$ and $\chi(\xi,t)=\psi_+(\xi,t)$ for all $\xi$ in $\iRight(t)$, it follows from inequality \cref{upp_bd_dE_tf_dichot} that (substituting $\chi$ with $\psi_-$ or $\psi_+$ and adding the nonnegative quantities $-V^\ddag(m_-)$ and $-V^\ddag(m_+)$)
\[
\begin{aligned}
\eee'(t) &+ (1+\alpha c^2) \ddd(t)\le  \\
\kappa_0&\int_{\iLeft(t)}\psi_- \biggl[ \frac{\alpha(\cCutZero+2c)+1}{2} v_t^2 + \frac{\cCutZero+1}{2} v_\xi^2 + \cCutZero \bigl(V^\ddag(v)-V^\ddag(m_-)\bigr) \biggr] \, d\xi \\
+ \kappa_0&\int_{\iRight(t)}\psi_+ \biggl[ \frac{\alpha(\cCutZero+2c)+1}{2} v_t^2 + \frac{\cCutZero+1}{2} v_\xi^2 + \cCutZero \bigl(V^\ddag(v)-V^\ddag(m_+)\bigr) \biggr] \, d\xi
\,.
\end{aligned}
\]
After replacing the quantities $V^\ddag(v)-V^\ddag(m_\pm)$ by their absolute values and extending to $\rr$ the integration domains of these two integrals, the inequality still holds and reads
\[
\begin{aligned}
\eee'(t) &+ (1+\alpha c^2) \ddd(t)\le  \\
\kappa_0&\int_{\rr}\psi_- \biggl[ \frac{\alpha(\cCutZero+2c)+1}{2} v_t^2 + \frac{\cCutZero+1}{2} v_\xi^2 + \cCutZero \abs{V^\ddag(v)-V^\ddag(m_-)} \biggr] \, d\xi \\
+ \kappa_0&\int_{\rr}\psi_+ \biggl[ \frac{\alpha(\cCutZero+2c)+1}{2} v_t^2 + \frac{\cCutZero+1}{2} v_\xi^2 + \cCutZero \abs{V^\ddag(v)-V^\ddag(m_+)} \biggr] \, d\xi
\,.
\end{aligned}
\]
Let $\KEFZero$ be a positive quantity to be chosen below. According to \cref{lower_bound_F_relax}, it follows that, for every nonnegative time $t$,
\[
\begin{aligned}
\eee'(t) &+ (1+\alpha c^2) \ddd(t) - \KEFZero \bigl(\fff_-(t)+\fff_+(t)\bigr) \le \\
&\int_{\rr}\psi_- \biggl[ \Bigl(\frac{\kappa_0\bigl(\alpha(\cCutZero+2c)+1\bigr)}{2} - \frac{\alpha^2 \KEFZero}{4}\Bigr)v_t^2 + \Bigl(\frac{\kappa_0(\cCutZero+1)}{2}-\alpha \KEFZero\Bigr) v_\xi^2 \\
&+\kappa_0\cCutZero \abs{V^\ddag(v)-V^\ddag(m_-)} -2\alpha\KEFZero\bigl(V^\ddag(v)-V^\ddag(m_-)\bigr) \biggr] \, d\xi \\
+&\int_{\rr}\psi_+ \biggl[ \Bigl(\frac{\kappa_0\bigl(\alpha(\cCutZero+2c)+1\bigr)}{2} - \frac{\alpha^2 \KEFZero}{4}\Bigr)v_t^2 + \Bigl(\frac{\kappa_0(\cCutZero+1)}{2}-\alpha \KEFZero\Bigr) v_\xi^2 \\
&+\kappa_0\cCutZero \abs{V^\ddag(v)-V^\ddag(m_+)} -2\alpha\KEFZero\bigl(V^\ddag(v)-V^\ddag(m_+)\bigr) \biggr] \, d\xi
\,.
\end{aligned}
\]
Thus, introducing the quantity $\KEFZero$ as
\[
\KEFZero = \max\biggl[\frac{2\kappa_0\bigl(\alpha(\cCutZero+2)+1\bigr)}{\alpha^2},\frac{\kappa_0(\cCutZero+1)}{2\alpha},\frac{\kappa_0\cCutZero}{2\alpha}\biggr]
\,,
\]
(this quantity depends only on $\alpha$ and $V$), it follows that
\begin{equation}
\label{approx_decr_en_dichot_proof}
\begin{aligned}
\eee'(t) + &(1+\alpha c^2) \ddd(t) - \KEFZero \bigl(\fff_-(t)+\fff_+(t)\bigr) \le \\
&\int_{\rr}\psi_-\Bigl[\kappa_0\cCutZero \abs{V^\ddag(v)-V^\ddag(m_-)} -2\alpha\KEFZero\bigl(V^\ddag(v)-V^\ddag(m_-)\bigr)\Bigr]\, d\xi \\
+&\int_{\rr}\psi_+\Bigl[\kappa_0\cCutZero \abs{V^\ddag(v)-V^\ddag(m_+)} -2\alpha\KEFZero\bigl(V^\ddag(v)-V^\ddag(m_+)\bigr)\Bigr]\, d\xi
\,.
\end{aligned}
\end{equation} 
According to the choice of $\KEFZero$, the integrand of the first (resp. the second) integral of the right-hand side of this inequality is nonpositive as long as $\xi$ is \emph{not} in $\SigmaEscMinus(t)$ (resp. $\SigmaEscPlus(t)$). As a consequence this inequality still holds if the integration domains of these integrals are restricted to $\SigmaEscMinus(t)$ and $\SigmaEscPlus(t)$, respectively. Thus, introducing the quantity $\KEEscZero$ as
\[
\KEEscZero = \bigl(\kappa_0\cCutZero+2\alpha\KEFZero\bigr) \max_{v\in\rr^d,\ \abs{v}\le \Rattinfty,\ m\in\{m_-,m_+\}}\abs{V(v)-V(m)}
\,,
\]
inequality \cref{approx_decr_en_dichot} follows from \cref{approx_decr_en_dichot_proof}. \Cref{lem:approx_decr_en_dichot} is proved.
\end{proof}
\subsubsection{Firewalls upper bounds}
\begin{lemma}[firewalls upper bounds]
\label{lem:fire_upp_bd_relax}
For every nonnegative time $t$, 
\begin{equation}
\label{fire_upp_bd_relax}
\fff_\pm(t) \le \int_{\rr} \psi_\pm \Bigl[\frac{3\alpha^2}{2} v_t^2 + \alpha v_\xi^2 + 2\alpha \bigl(V^\dag(v)-V(m_\pm)\bigr) + \bigl(1+\alpha c(\kappa_0+c)\bigr)(v-m_\pm)^2 \Bigr]\, d\xi
\,.
\end{equation}
\end{lemma}
\begin{proof}
Inequality \cref{fire_upp_bd_relax} follows from the definition \vref{def_F_plus_minus} of $F^\ddag_\pm(\xi,s)$, from the fact that $\partial_\xi\psi_\pm/\psi$ is bounded from above by $\kappa_0+c$, and from the inequality
\[
\alpha (v-m_\pm)\cdot v_t \le \frac{\alpha^2}{2} v_t^2 + \frac{1}{2} (v-m_\pm)^2
\,.
\]
\end{proof}
\subsubsection{Firewalls linear decrease up to pollution}
\label{subsubsec:time_der_fire_dichot}
Let us denote by $\nuFZero(m_-)$ and $\KFZero(m_-)$ ($\nuFZero(m_+)$ and $\KFZero(m_+)$) the quantities denoted by $\nuFZero$ and $\KFZero$ in the proof of \vref{lem:decrease_fire0}, when the minimum point $m$ of \cref{lem:decrease_fire0} is replaced with $m_-$ (with $m_+$).
\begin{lemma}[firewalls linear decrease up to pollution]
\label{lem:fire_decr_dichot}
For every nonnegative quantity $t$, 
\begin{equation}
\label{dt_F_final_dichot}
\fff_\pm'(t) \le - \nuFZero(m_\pm) \fff_\pm(t)+ \KFZero(m_\pm) \mathcal{G}_\pm(t)
\,.
\end{equation}
\end{lemma}
\begin{proof}
The proof is very similar to that of \vref{lem:fire_decr}; however, since the definitions of the various parameters and functions are slightly different, the details of the calculations are provided. Proceeding as in the beginning of the proof of \cref{lem:fire_decr}, it follows that, for all nonnegative time $t$, 
\[
\begin{aligned}
\fff_\pm'(t) = &\int_{\rr} \Biggl[ 
\alpha \bigl( - \psi_\pm - 2\alpha c \partial_\xi\psi_\pm + \alpha \partial_t\psi_\pm \bigr)v_t^2 + \bigl( - \psi_\pm + \alpha \partial_t\psi_\pm \bigr)v_\xi^2 \\
& - \psi_\pm (v-m_\pm) \cdot \nabla V^\ddag (v) - 2 \alpha \partial_\xi\psi_\pm v_\xi \cdot v_t + \frac{\partial_t\psi_\pm + \partial_\xi^2\psi_\pm - c \partial_\xi\psi_\pm}{2} (v-m_\pm)^2 \\
& + \alpha\partial_t\psi_\pm \Bigl( 2 \bigl(V^\ddag(v)-V^\ddag(m_\pm)\bigr) + (v-m_\pm) \cdot v_t - 2  c (v-m_\pm) \cdot v_\xi \Bigr)
\Biggr] \, d\xi
\,.
\end{aligned}
\]
According to the definition of $\psi_\pm$, for all $(\xi,t)$ in $\rr\times[0,+\infty)$ (omitting the arguments $(\xi,t)$ of $\psi_\pm$ and of their derivatives),
\[
\begin{aligned}
\partial_t \psi_\pm = e^{c\xi}\partial_t \psi_{0,\pm}
\quad&\text{thus}\quad
\abs{\partial_t \psi_\pm } = \cCutZero \kappa_0 \psi_\pm \,, \\
c\psi_\pm -\partial_\xi \psi_\pm = -e^{c\xi} \partial_\xi \psi_{0,\pm}
\quad&\text{thus}\quad
\abs{c\psi_\pm -\partial_\xi \psi_\pm} = \kappa_0 \psi_\pm \,, \\
\partial_\xi^2 \psi_\pm - c \partial_\xi \psi_\pm 
= \partial_\xi(e^{c\xi} \partial_\xi \psi_{0,\pm})\qquad\qquad & \\
= e^{c\xi}(c\partial_\xi \psi_{0,\pm} + \partial_\xi^2\psi_{0,\pm})
\quad&\text{thus}\quad
\partial_\xi^2 \psi_\pm - c \partial_\xi \psi_\pm \le \kappa_0(\kappa_0+c)\psi_\pm
\end{aligned}
\]
(compare with the bounds \vref{bounds_psi_psi_s_psi_xi_etc}). Thus, for every nonnegative time $t$, it follows from the previous expression of $\fff_\pm'(t)$ that
\[
\begin{aligned}
&\fff_\pm'(t)  \le \int_{\rr} \psi_\pm \Biggl[ 
\alpha\Bigl( -1 - 2\alpha c \frac{\partial_\xi\psi_\pm}{\psi_\pm} + \alpha \cCutZero\kappa_0 \Bigr) v_t^2+ (-1 + \alpha \cCutZero\kappa_0)v_\xi^2  \\
& - (v-m_\pm) \cdot \nabla V^\ddag (v) - 2 \alpha \frac{\partial_\xi\psi_\pm}{\psi_\pm} v_\xi \cdot v_t + \frac{\kappa_0(\cCutZero+\kappa_0+c)}{2}(v-m_\pm)^2 \\
&+ \alpha\cCutZero\kappa_0 \Bigl( 2 \abs{V^\ddag(v)-V^\ddag(m_\pm)} +  \abs{(v-m_\pm) \cdot v_t} + 2  c \abs{(v-m_\pm) \cdot v_\xi} \Bigr)
\Biggr] \, d\xi
\,.
\end{aligned}
\]
Using the inequalities
\[
\begin{aligned}
- 2 \alpha \frac{\partial_\xi\psi_\pm}{\psi_\pm} v_\xi \cdot v_t &\le \frac{1}{2}v_\xi^2 + 2\alpha^2 \frac{(\partial_\xi\psi_\pm)^2}{\psi_\pm^2} v_t^2 \\
\text{and}\qquad
\abs{(v-m_\pm) \cdot v_t} &\le \frac{1}{2} (v-m_\pm)^2 + \frac{1}{2} v_t^2 \\
\text{and}\qquad
2\abs{(v-m_\pm) \cdot v_\xi} &\le (v-m_\pm)^2 + v_\xi^2
\,,
\end{aligned}
\]
it follows that
\[
\begin{aligned}
\fff_\pm'(t)  \le \int_{\rr} \psi_\pm \Biggl[&
\alpha\Bigl( -1 - 2\alpha c \frac{\partial_\xi\psi_\pm}{\psi_\pm} + \alpha \cCutZero\kappa_0 + 2\alpha\frac{(\partial_\xi\psi_\pm)^2}{\psi_\pm^2} + \frac{\cCutZero\kappa_0}{2} \Bigr) v_t^2 \\
&+\Bigl(-1+\frac{1}{2}+\alpha\cCutZero\kappa_0(c+1)\Bigr)v_\xi^2 - (v-m_\pm) \cdot \nabla V^\ddag (v)  \\
& + \kappa_0 \Bigl(\frac{\cCutZero+\kappa_0+c}{2}+\frac{\alpha \cCutZero}{2}+\alpha c\cCutZero\Bigr)(v-m_\pm)^2 \\
&+ 2 \alpha\cCutZero\kappa_0 \abs{V^\ddag(v)-V^\ddag(m_\pm)}
\Biggr] \, d\xi
\,.
\end{aligned}
\]
Observe that the following equality holds, for all values of argument $\xi$:
\[
- 2\alpha c \frac{\partial_\xi\psi_\pm}{\psi_\pm} + 2\alpha\frac{(\partial_\xi\psi_\pm)^2}{\psi_\pm^2} = - 2\alpha  \frac{\partial_\xi\psi_\pm}{\psi_\pm}\cdot\frac{c\psi_\pm -\partial_\xi \psi_\pm }{\psi_\pm} \le 2\alpha\kappa_0(\kappa_0+c)
\,.
\]
Thus, the previous inequality becomes
\[
\begin{aligned}
&\fff_\pm'(t)  \le \int_{\rr} \psi_\pm \Biggl[
\alpha\Bigl( -1 + \kappa_0\bigl(2\alpha(\kappa_0+c) + \cCutZero(\alpha+1/2)\bigr)\Bigr) v_t^2 \\
&+ \Bigl(-\frac{1}{2}+\alpha\cCutZero\kappa_0(c+1)\Bigr)v_\xi^2 - (v-m_\pm) \cdot \nabla V^\ddag (v) \\
&+ \frac{\kappa_0}{2} \Bigl(\kappa_0+c+\cCutZero\bigl(1+\alpha (2c+1)\bigr)\Bigr)(v-m_\pm)^2 + 2 \alpha\cCutZero\kappa_0 \abs{V^\ddag(v)-V^\ddag(m_\pm)}
\Biggr] \, d\xi
\,.
\end{aligned}
\]
It follows from the definitions \cref{def_kappaZero} of $\kappa_0$ and \cref{def_cCutZero} of $\cCutZero$ and the conditions \cref{hyp_param_relax} on $c$ that
\[
\begin{aligned}
\kappa_0\bigl(2\alpha(\kappa_0+c) + \cCutZero(\alpha+1/2)\bigr) \le \frac{1}{2}
\quad&\text{and}\quad
\alpha\cCutZero\kappa_0(c+1) \le \frac{1}{4} \\
\text{and}\quad
\frac{\kappa_0}{2} \Bigl(\kappa_0+c+\cCutZero\bigl(1+\alpha (2c+1)\bigr)\Bigr) \le\frac{\eigVmin}{8}
\quad&\text{and}\quad
2 \alpha\cCutZero\kappa_0 \le \frac{1}{4}
\,;
\end{aligned}
\]
thus it follows from the previous inequality that
\[
\begin{aligned}
\fff_\pm'(t)  \le \int_{\rr} \psi_\pm \Biggl[&
-\frac{\alpha}{2} v_t^2 -\frac{1}{4} v_\xi^2 - (v-m_\pm) \cdot \nabla V^\ddag (v) + \frac{\eigVmin}{8} (v-m_\pm)^2  \\
&+ \frac{1}{4} \abs{V^\ddag(v)-V^\ddag(m_\pm)}
\Biggr] \, d\xi
\,,
\end{aligned}
\]
and it follows from the upper bound \cref{fire_upp_bd_relax} of \cref{lem:fire_upp_bd_relax} on $\fff(t)$ that
\[
\begin{aligned}
\fff_\pm'(t) +& \nuFZero(m_\pm) \fff_\pm(t) \le \int_{\rr} \psi_\pm \Biggl[
\frac{\alpha}{2}(-1+3\alpha \nuFZero(m_\pm)) v_t^2 +\Bigl(-\frac{1}{4} + \alpha \nuFZero(m_\pm) \Bigr) v_\xi^2  \\
&- (v-m_\pm) \cdot \nabla V^\ddag (v) + \Bigl(\frac{\eigVmin}{8} + \nuFZero(m_\pm)\bigl(1+\alpha c(\kappa_0+c)\bigr)\Bigr) (v-m_\pm)^2 \\
& + \Bigl(\frac{1}{4}+2\alpha\nuFZero(m_\pm)\Bigr) \abs{V^\ddag(v)-V^\ddag(m_\pm)} 
\Biggr] \, d\xi
\,.
\end{aligned}
\]
Thus it follows from the definition \cref{def_nuFireZero} of $\nuFZero(m_\pm)$ that
\[
\begin{aligned}
\fff_\pm'(t) + \nuFZero(m_\pm) \fff_\pm(t) \le \int_{\rr} \psi_\pm \Biggl[ & - (v-m_\pm) \cdot \nabla V^\ddag (v) + \frac{\eigVmin}{4} (v-m_\pm)^2 \\
&+ \frac{1}{2} \abs{V^\ddag(v)-V^\ddag(m_\pm)}
\Biggr] \, d\xi
\,.
\end{aligned}
\]
In view of the $L^\infty$-bound \vref{hyp_attr_ball_Linfty_relax} for the solution, the end of the proof is identical to that of \cref{lem:fire_decr}.
\end{proof}
\subsubsection{Control over pollution}
\label{subsubsec:control_over_pollution}
The following lemma calls upon the notation $T$ introduced for inequalities \cref{hyp_pos_relax}. 
\begin{lemma}[control over pollution]
\label{lem:control_over_pollution}
For every time $t$ greater than or equal to $T$, 
\begin{equation}
\label{control_over_pollution}
\mathcal{G}_\pm(t) \le \frac{5}{2\kappa_0}\exp\Bigl(-\frac{\kappa_0\cCutZero}{2}t\Bigr)
\,.
\end{equation}
\end{lemma}
\begin{proof}
For every nonnegative time $t$, let 
\[
\begin{aligned}
& \xiHomMinus(t) = \sqrt{1+\alpha c^2}\xHomMinus(t) - ct
\quad\text{and}\quad
\xiEscMinus(t) = \sqrt{1+\alpha c^2}\xEscMinus(t) - ct \\
\text{and}\quad
& \xiEscPlus(t) = \sqrt{1+\alpha c^2}\xEscPlus(t) - ct
\quad\text{and}\quad
\xiHomPlus(t) = \sqrt{1+\alpha c^2}\xHomPlus(t) - ct
\,.
\end{aligned}
\]
Assume that the time $t$ is \emph{greater than or equal to $T$}; and observe that according to \cref{hyp_param_relax} the quantity $\sqrt{1+\alpha c^2}$ is not larger than $\sqrt{2}$. Then it follows from hypotheses \cref{hyp_pos_relax} and from the two last hypotheses of \cref{hyp_param_relax} that
\begin{equation}
\label{hyp_relax_in_tf}
\begin{aligned}
\xiHomMinus(t)&\le -\frac{5}{3}\cCutZero t
&\quad&\text{and}\quad&
-\frac{1}{3}\cCutZero t &\le \xiEscMinus(t)\,, \\
\text{and}\quad
\xiEscPlus(t)&\le \frac{1}{3}\cCutZero t 
&\quad&\text{and}\quad&
\frac{5}{3}\cCutZero t &\le \xiHomPlus(t)
\,,
\end{aligned}
\end{equation}
see \cref{fig:setting_relax}. 
\begin{figure}[!htbp]
\centering
\includegraphics[width=\textwidth]{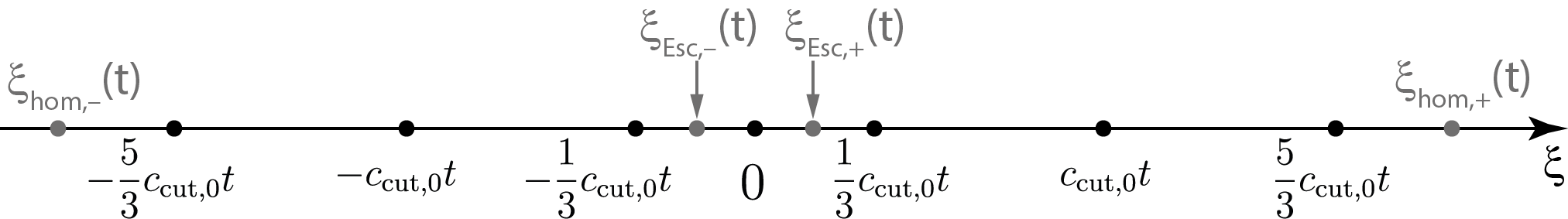}
\caption{Illustration of the notation and assumptions for the proof of \cref{prop:relax}.}
\label{fig:setting_relax}
\end{figure}
According to the definition of $\xEscPlus(t)$ and $\xEscMinus(t)$, 
\[
\begin{aligned}
\SigmaEscMinus(t) &\subset (-\infty,\xiHomMinus(t)] \cup [\xiEscMinus(t),+\infty) \\
\text{and}\quad
\SigmaEscPlus(t) &\subset (-\infty,\xiEscPlus(t)] \cup [\xiHomPlus(t),+\infty)
\,.
\end{aligned}
\]
Let us introduce the quantities
\[
\begin{aligned}
\GfrontMinus(t) &= \int_{-\infty}^{\xiHomMinus(t)} \psi_- (\xi,t) \, d\xi
&\ &\text{and}\  &
\GbackMinus(t) &= \int_{\xiEscMinus(t)}^{+\infty} \psi_- (\xi,t) \, d\xi 
\,,
\\
\text{and}\quad
\GbackPlus(t) &= \int_{-\infty}^{\xiEscPlus(t)} \psi_+ (\xi,t) \, d\xi 
&\ &\text{and}\  &
\GfrontPlus(t) &= \int_{\xiHomPlus(t)}^{+\infty} \psi_+ (\xi,t) \, d\xi 
\,.
\end{aligned}
\]
Then, it follows from the definition \cref{def_gggg_pm} of $\mathcal{G}_\pm(t)$ that
\[
\mathcal{G}_-(t) \le \GfrontMinus(t) + \GbackMinus(t)
\quad\text{and}\quad
\mathcal{G}_+(t) \le \GfrontPlus(t) + \GbackPlus(t)
\,.
\]
According to the definition of $\psi_+$ and $\psi_-$ and according to hypotheses \cref{hyp_param_relax} and inequalities \vref{hyp_relax_in_tf} it follows from explicit calculations that:
\[
\begin{aligned}
\GfrontMinus(t) &\le \frac{1}{\kappa_0+c} \exp\Bigl( \cCutZero \kappa_0 t + (\kappa_0+c) \xiHomMinus(t) \Bigr) 
&&\le \frac{1}{\kappa_0+c} \exp\Bigl(-\frac{\kappa_0\cCutZero}{2}t\Bigr) \,, \\
\GbackMinus(t) &\le \frac{1}{\kappa_0-c} \exp\Bigl(  -\cCutZero \kappa_0 t - (\kappa_0-c) \xiEscMinus(t)  \Bigr) 
&&\le \frac{1}{\kappa_0-c} \exp\Bigl(-\frac{\kappa_0\cCutZero}{2}t\Bigr) \,, \\
\GbackPlus(t) &\le \frac{1}{\kappa_0+c} \exp\Bigl( -\cCutZero \kappa_0 t + (\kappa_0+c) \xiEscPlus(t)  \Bigr) 
&&\le \frac{1}{\kappa_0-c} \exp\Bigl(-\frac{\kappa_0\cCutZero}{2}t\Bigr) \,, \\
\GfrontPlus(t) &\le \frac{1}{\kappa_0-c} \exp\Bigl( \cCutZero \kappa_0 t - (\kappa_0-c) \xiHomPlus(t)  \Bigr) 
&&\le \frac{1}{\kappa_0+c} \exp\Bigl(-\frac{\kappa_0\cCutZero}{2}t\Bigr) 
\,.
\end{aligned}
\]
It follows that
\[
\mathcal{G}_\pm(t) \le \frac{2\kappa_0}{\kappa_0^2-c^2} \exp\Bigl(-\frac{\kappa_0\cCutZero}{2}t\Bigr)
\,,
\]
and since according to the conditions \cref{hyp_param_relax} the (nonnegative) quantity $c$ is not larger than $\kappa_0/6$, inequality \cref{control_over_pollution} follows. \Cref{lem:control_over_pollution} is proved.
\end{proof}
\subsubsection{Energy decrease up to pollution}
\label{subsubsec:dt_en_final_dichot}
\begin{lemma}[firewall linear decrease up to pollution, 2]
\label{lem:der_fire_dichot_next}
There exists a positive quantity $\KFprime$, depending only on $\alpha$ and $V$ and $m_-$ and $m_+$, such that, for every time $t$ greater than or equal to $T$, 
\begin{equation}
\label{der_fire_dichot_next}
\fff_\pm'(t) \le - \nuFZero(m_\pm) \fff_\pm(t) + \KFprime \exp\Bigl(-\frac{\kappa_0\cCutZero}{2}t\Bigr)
\,.
\end{equation}
\end{lemma}
\begin{proof}
Introducing the positive quantity $\KFprime$ defined as
\[
\KFprime = \frac{5}{2\kappa_0}\max\bigl(\KFZero(m_-),\KFZero(m_+)\bigr)
\,,
\]
inequality \cref{der_fire_dichot_next} follows from inequality \cref{dt_F_final_dichot} of \cref{lem:fire_decr_dichot} and inequality \cref{control_over_pollution} of \cref{lem:control_over_pollution}. \Cref{lem:der_fire_dichot_next} is proved. 
\end{proof}

\begin{lemma}[energy decrease up to pollution]
\label{lem:dt_en_final_dichot}
There exist positive quantities $\nuE$ and $\KE$, depending only on $\alpha$ and $V$, such that, for every time $t$ greater than or equal to $T$,
\begin{equation}
\label{dt_en_final_dichot}
\eee'(t) \le -(1+\alpha c^2) \ddd(t) + \KE \exp\bigl(-\nuE (t-T)\bigr)
\,.
\end{equation}
\end{lemma}
\begin{proof}
Let 
\[
\nuE = \min\left(\nuFZero(m_-), \nuFZero(m_+),\frac{\kappa_0\cCutZero}{4}\right)
\,.
\]
According to Grönwall's inequality, it follows from inequalities \cref{der_fire_dichot_next} of \cref{lem:der_fire_dichot_next} that, for every time $t$ greater than or equal to $T$, 
\begin{align}
\fff_\pm(t) &\le \exp\bigl(- \nuFZero(m_\pm)(t-T)\bigr)\fff_\pm(T) \nonumber\\
&\qquad + \KFprime\int_T^t \exp\bigl(- \nuFZero(m_\pm)(t-s)\bigr)\exp\Bigl(-\frac{\kappa_0\cCutZero}{2}s\Bigr) \, ds \nonumber\\
&\le \exp\bigl(- \nuE(t-T)\Biggl(\max\bigl(\fff_\pm(T),0\bigr) + \KFprime\exp\Bigl(-\frac{\kappa_0\cCutZero}{2}T\Bigr)\times \nonumber\\
& \qquad\int_T^t \exp\Bigl(- \bigl(\nuFZero(m_\pm)-\nuE\bigr)(t-s)\Bigr)\exp\biggl(-\Bigl(\frac{\kappa_0\cCutZero}{2}-\nuE\Bigr)(s-T)\biggr)\, ds \Biggr) \nonumber\\
&\le \exp\bigl(- \nuE(t-T)\bigr)\left(\max\bigl(\fff_\pm(T),0\bigr) + \KFprime\int_T^t \exp\Bigl(-\frac{\kappa_0\cCutZero}{4}(s-T)\Bigr)\, ds\right)\nonumber\\
&\le \left(\max\bigl(\fff_\pm(T),0\bigr) + \frac{4 \KFprime}{\kappa_0 \cCutZero}\right)\exp\bigl(- \nuE(t-T)\bigr)
\,.
\label{der_fire_dichot}
\end{align}
According to the $H^1_\text{ul}\times L^2_\text{ul}$-bound \vref{hyp_attr_ball_X_relax} for the solution, there exists a positive quantity $\fffMax$, depending only on $\alpha$ and $V$ and $m_-$ and $m_+$, such that
\[
\fff_+(T)\le \fffMax
\quad\text{and}\quad
\fff_-(T)\le \fffMax
\,.
\]
Thus, introducing the nonnegative quantity
\[
\KE = 2 \KEFZero \Bigl( \fffMax + \frac{4 \KFprime}{\kappa_0 \cCutZero} \Bigr) + \frac{5\KEEscZero}{\kappa_0}
\,,
\]
inequality \cref{dt_en_final_dichot} follows from inequalities \cref{approx_decr_en_dichot} of \cref{lem:approx_decr_en_dichot}, inequality \cref{control_over_pollution} of \cref{lem:control_over_pollution}, and inequality \cref{der_fire_dichot}. \Cref{lem:dt_en_final_dichot} is proved. 
\end{proof}
Inequality \cref{dt_en_final_dichot} of \cref{lem:dt_en_final_dichot} is the key ingredient that will be applied in the next \cref{subsec:low_bd_en}. 
\subsection{Nonnegative asymptotic energy}
\label{subsec:low_bd_en}
Let us keep the notation and hypotheses introduced since the beginning of \cref{sec:no_inv_implies_relax}. For every quantity $c$ close enough to $0$ so that hypotheses \vref{hyp_param_relax} be satisfied, and for every nonnegative time $t$ and real quantity $\xi$, let us denote by 
\[
v^{(c)}(\xi,t)
\quad\text{and}\quad
\chi^{(c)}(\xi,t)
\quad\text{and}\quad
\eee^{(c)}(t)
\quad\text{and}\quad
\ddd^{(c)}(t)
\]
the functions defined as in \cref{subsec:relax_sc_stand}, with the same notation except the ``$(c)$'' superscript that is here to remind that these objects depend on the quantity $c$. For every such $c$, let us introduce the quantity $\eee^{(c)}(+\infty)$ in $\rr\cup\{-\infty\}$ defined as
\[
\eee^{(c)}(+\infty) = \liminf_{t\to+\infty} \eee^{(c)}(t)
\,.
\]
According to estimate \cref{dt_en_final_dichot} on the time derivative of the energy, for every such $c$, 
\begin{equation}
\label{convergence_energy_almost_standing_frame}
\eee^{(c)}(t)\to \eee^{(c)}(+\infty)
\quad\text{as}\quad
t\to + \infty
\,,
\end{equation}
and let us call ``asymptotic energy at the speed $c$'' this quantity.
The aim of this \namecref{subsec:low_bd_en} is to prove the following proposition.
\begin{proposition}[nonnegative asymptotic energy]
\label{prop:nonneg_asympt_en}
The quantity $\eee^{(0)}(+\infty)$ (the asymptotic energy at speed zero) is nonnegative.
\end{proposition}
The proof proceeds through the following lemmas and corollaries, that are rather direct consequences of the relaxation scheme set up in the previous \cref{subsec:relax_sc_stand}, and in particular of the estimate \cref{dt_en_final_dichot} on the time derivative of the energy.

Since according to the definition of $V^\ddag$ the maximum of $V^\ddag(m_+)$ and $V^\ddag(m_-)$ is equal to zero, it may be assumed (without loss of generality) that
\begin{equation}
\label{assumpt_V_ddag_of_m_plus_is_zero}
V^\ddag(m_+)=0
\,.
\end{equation}
\begin{lemma}[nonnegative asymptotic energy in frames travelling at small nonzero speed]
\label{lem:nonneg_en_prel}
For every quantity $c$ close enough to zero so that hypotheses \vref{hyp_param_relax} be satisfied, if in addition $c$ is positive, then
\[
\eee^{(c)}(+\infty) \ge 0
\,.
\]
\end{lemma}
\begin{proof}
See \cite[\GlobalBehaviourLemNonNegativeAsymptoticEnergyTf]{Risler_globalBehaviour_2016}. 
\end{proof}
\begin{corollary}[almost nonnegative energy in a travelling frame]
\label{cor_low_bd_en_c}
For every quantity $c$ close enough to zero so that hypotheses \vref{hyp_param_relax} be satisfied, if in addition $c$ is positive, then, for every time $t$ greater than or equal to $T$, 
\[
\eee^{(c)}(t) \ge -\frac{\KE}{\nuE} \exp\bigl(-\nuE (t-T)\bigr)
\,.
\]
\end{corollary}
\begin{proof}
The proof follows from previous \cref{lem:nonneg_en_prel} and inequality \cref{dt_en_final_dichot}.
\end{proof}
\begin{lemma}[continuity of energy with respect to the speed at $c=0$]
\label{lem:cont_en_c}
For every nonnegative quantity $t$, 
\[
\eee^{(c)}(t)\to \eee^{(0)}(t)
\quad\text{as}\quad
c \to 0
\,.
\]
\end{lemma}
\begin{proof}
For all $t$ in $(0,+\infty)$, 
\[
\eee^{(0)}(t) = \int_{\rr} \chi^{(0)}(x,t)\Bigl(  \frac{\alpha}{2}u_t(x,t)^2 + \frac{1}{2}u_x(x,t)^2 + V^\ddag\bigl( u(x,t)\bigr) \Bigr) \, dx 
\,,
\]
and, for every quantity $c$ close enough to zero so that hypotheses \vref{hyp_param_relax} be satisfied, 
\[
\eee^{(c)}(t) = \int_{\rr} \chi^{(c)}(\xi,t)\Bigl(  \frac{\alpha}{2}v^{(c)}_t(\xi,t)^2 + \frac{1}{2}v^{(c)}_\xi(\xi,t)^2 + V^\ddag\bigl( v^{(c)}(\xi,t)\bigr) \Bigr) \, d\xi 
\,.
\]
Thus, since $v^{(c)}(\cdot,\cdot)$ is related to $u(\cdot,\cdot)$ by 
\[
u(x,t) = v^{(c)}(\xi,t)
\quad\text{where}\quad
\xi = \sqrt{1+\alpha c^2} x - ct
\,,
\]
it follows that
\[
\begin{aligned}
\eee^{(c)}(t) = \int_{\rr} \chi^{(c)}(\sqrt{1+\alpha c^2} x - ct,t)\biggl( & \frac{\alpha}{2} \Bigl( u_t(x,t) + c \frac{u_x(x,t)}{\sqrt{1+\alpha c^2}} \Bigr)^2 + \frac{1}{2(1+\alpha c^2)} u_x(x,t)^2 \\
& + V^\ddag\bigl( u(x,t)\bigr) \biggr) \sqrt{1+\alpha c^2} \, dx 
\,.
\end{aligned}
\]
The result thus follows from the continuity of $\chi^{(c)}(\cdot,\cdot)$ with respect to $c$ and from the on the derivatives of $u(\cdot,\cdot)$ ensured by \vref{prop:exist_sol_att_ball}.  
\end{proof}
\begin{corollary}[almost nonnegative energy in a standing frame]
\label{cor:low_bd_en_0}
For every time $t$ greater than or equal to $T$,
\begin{equation}
\label{low_bd_en_0}
\eee^{(0)}(t) \ge -\frac{\KE}{\nuE} \exp\bigl(-\nuE (t-T)\bigr)
\,.
\end{equation}
\end{corollary}
\begin{proof}
Inequality \cref{low_bd_en_0} follows from \cref{cor_low_bd_en_c} and \cref{lem:cont_en_c}.
\end{proof}
\Cref{prop:nonneg_asympt_en} (``nonnegative asymptotic energy'') follows from \cref{cor:low_bd_en_0}.
\subsection{End of the proof of \texorpdfstring{\cref{prop:relax}}{Proposition \ref{prop:relax}}}
\label{subsec:end_pf_prop_relax}
\begin{lemma}[integrability of dissipation in a standing frame]
\label{lem:dissip_is_integrable}
The function
\[
t\mapsto \ddd^{(0)}(t) 
\]
is integrable on $[0,+\infty)$.
\end{lemma}
\begin{proof}
The statement follows from \cref{prop:nonneg_asympt_en} (``nonnegative asymptotic energy'') and from estimate \cref{dt_en_final_dichot} on the time derivative of energy.
\end{proof}
\begin{corollary}[relaxation --- centre area]
\label{cor_dissip_app_zero}
The following limit holds:
\begin{equation}
\label{ut_app_zero_centre}
\sup_{x\in[-\cCutZero t , \cCutZero t]} \int_{x-1}^{x+1} u_t(y,t)^2 \, dy \to 0
\quad\text{as}\quad
t\to +\infty
\,.
\end{equation}
\end{corollary}
\begin{proof}
Let us proceed by contradiction and assume that the converse holds. Then there exists a positive quantity $\varepsilon$ and a sequence $\bigl((x_n,t_n)\bigr)_{n\in\nn}$ in $\rr\times[0,+\infty)$ such that $t_n$ goes to $+\infty$ as $n$ goes to $+\infty$ and such that, for every $n$ in $\nn$, $x_n$ is in the interval $[-\cCutZero t_n , \cCutZero t_n]$ and 
\begin{equation}
\label{lower_bound_dissip_around_xn_at_time_tn}
\int_{-1}^{+1} u_t(x_n + y,t_n)^2 \, dy \ge \varepsilon
\,.
\end{equation}
According to \vref{prop:asympt_comp} (``asymptotic compactness''), up to replacing the sequence $\bigl((x_n,t_n)\bigr)_{n\in\nn}$ by a subsequence, it may be assumed that the sequence of functions $(u,u_t) (x_n+\cdot, t_n + \cdot)$ converges in the space 
\[
\ccc^0\Bigl([-1,1],H^1\bigl([-1,1],\rr^d\bigr)\times L^2\bigl([-1,1],\rr^d\bigr)\Bigr)
\]
to some limit $(\bar{u},\bar{u}_t)$ that satisfies system \cref{hyp_syst}. It follows from \cref{lower_bound_dissip_around_xn_at_time_tn} that 
\[
\int_{-1}^1 \bar{u}_t(y,0)^2 \, dy \ge \varepsilon
\,,
\quad\text{thus}\quad
\int_{-1}^1 \Bigl( \int_{-1}^1 \bar{u}_t(y,0)^2 \, dy \Bigr) dt >0
\,,
\]
and as a consequence,
\[
\liminf_{n\to+\infty} \int_{-1}^1 \left( \int_{-1}^1 u_t(x_n+y,t_n + t)^2 \, dy \right) dt >0
\,,
\]
a contradiction with the integrability of $t\mapsto\ddd^{(0)}(t)$ (\cref{lem:dissip_is_integrable}). \Cref{cor_dissip_app_zero} is proved. 
\end{proof}
\begin{lemma}[relaxation --- non centre area]
\label{lem:relax_outside_centre_area}
For every positive quantity $\varepsilon$, both quantities
\begin{equation}
\label{relax_outside_centre_area}
\begin{aligned}
& \sup_{x\in[\xHomMinus(t),-\varepsilon t]}\int_{x-1}^{x+1} \Bigl( u_t(y,t)^2 + u_x(y,t)^2 + \bigl( u(y,t)-m_-)^2 \Bigr) \, dy \\
\text{and}\quad
& \sup_{x\in[\xHomPlus(t),\varepsilon t]}\int_{x-1}^{x+1} \Bigl( u_t(y,t)^2 + u_x(y,t)^2 + \bigl( u(y,t)-m_+)^2 \Bigr) \, dy
\end{aligned}
\end{equation}
go to $0$ as time goes to $+\infty$. 
\end{lemma}
\begin{proof}
Since the distance between the interval $[\xHomMinus(t),-\varepsilon t]$ and the set $\SigmaEscMinus(t)$ and the distance between the interval $[\varepsilon t,\xHomPlus(t)]$ and the set $\SigmaEscPlus(t)$ both go to $+\infty$ as time goes to $+\infty$, assertion \cref{relax_outside_centre_area} can be derived (for instance) from inequality \cref{dt_F0_final} of \vref{lem:decrease_fire0} (``firewall decrease up to pollution term'' in the laboratory frame) and inequality \cref{coerc_F0} of \vref{lem:coerc_fire0} (``firewall coercivity up to pollution term'' in the laboratory frame). 
\end{proof}
\begin{lemma}[$V(m_-)$ equals $V(m_+)$]
\label{lem:V_of_m_minus_equals_V_of_m_plus}
The following equalities hold:
\[
V^\ddag(m_-)=V^\ddag(m_+)=0\,,
\quad\text{or in other words}\quad
V(m_-)=V(m_+)
\,.
\]
\end{lemma}
\begin{proof}
It follows from the definition \cref{def_V_ddag} of $V^\ddag$ and from the assumption \cref{assumpt_V_ddag_of_m_plus_is_zero} that $V^\ddag(m_+)$ equals $0$ and that $V^\ddag(m_-)$ is nonpositive. If $V^\ddag(m_-)$ was negative, then, according to \cref{lem:relax_outside_centre_area} above (and according to the bounds \cref{hyp_attr_ball_Linfty_relax} on the solution), the following estimate would hold:
\[
\eee^{(0)}(t)\sim V^\ddag(m_-)\, \cCutZero \, t
\quad\text{as}\quad
t\to+\infty
\,,
\]
a contradiction with \cref{prop:nonneg_asympt_en}. \Cref{lem:V_of_m_minus_equals_V_of_m_plus} is proved.
\end{proof}
\begin{lemma}[convergence towards asymptotic energy]
\label{lem:asympt_energy_for_various_bounds}
For every quantity $\sigma_-$ in $(\sHomMinus,0)$ and every quantity $\sigma_+$ in $(0,\sHomPlus)$, 
\begin{equation}
\label{equ_asympt_energy_for_various_bounds}
\int_{\sigma_- t}^{\sigma_+ t} \Bigl(\frac{\alpha}{2}u_t(x,t)^2+\frac{1}{2}u_x(x,t)^2+V\bigl(u(x,t)\bigr)\Bigr)\, dx \to \eee^{(0)}(+\infty)
\quad\text{as}\quad
t\to+\infty
\,.
\end{equation}
\end{lemma}
\begin{proof}
According to \cref{convergence_energy_almost_standing_frame} the quantity
\[
\eee^{(0)}(t) = \int_{\rr}\chi^{(0)}(x,t)\Bigl(\frac{\alpha}{2}u_t(x,t)^2 + \frac{1}{2}u_x(x,t)^2 + V^\ddag\bigl(u(x,t)\bigr)\Bigr) \, dx
\]
goes to $\eee^{(0)}(+\infty)$ as time goes to $+\infty$, and according to \cref{lem:V_of_m_minus_equals_V_of_m_plus}, $V\ddag(\cdot)$ equals $V(\cdot)-V(m_\pm)$. The fact that the same asymptotic behaviour holds for the integral in \cref{equ_asympt_energy_for_various_bounds} (whatever the values of $\sigma_-$ and $\sigma_+$) can (once again) be derived from inequality \cref{dt_F0_final} of \vref{lem:decrease_fire0} (``firewall decrease up to pollution term'' in the laboratory frame). \Cref{lem:asympt_energy_for_various_bounds} is proved. 
\end{proof}
\begin{proof}[Proof of \cref{prop:relax}]
All statements of \cref{prop:relax} have been proved: 
\begin{enumerate}
\item equality between $V(m_-)$ and $V(m_+)$ is stated in \cref{lem:V_of_m_minus_equals_V_of_m_plus};
\item limits \cref{dissip_goes_to_zero_prop_relax,sol_goes_to_m_plus_minus_prop_relax} are stated in \cref{cor_dissip_app_zero,lem:relax_outside_centre_area,lem:V_of_m_minus_equals_V_of_m_plus};
\item according to \cref{prop:nonneg_asympt_en} the quantity $\eee^{(0)}(+\infty)$ is nonnegative, and, denoting by $\eeeResAsympt[u]$ this quantity, the limit \cref{asympt_energy_prop_relax} is stated in \cref{lem:asympt_energy_for_various_bounds}. 
\end{enumerate}
\Cref{prop:relax} is proved. 
\end{proof}
\section{Convergence}
\label{sec:convergence}
The aim of this \namecref{sec:convergence} is to prove \cref{prop:approach_set_bistable_stat_sol_stand_terrace} below. This statement extends \cref{prop:relax} under additional hypotheses. 
\subsection{Set-up}
\subsubsection{Hypotheses}
As everywhere else, let us consider a function $V$ in $\ccc^2(\rr^d,\rr)$ satisfying the coercivity hypothesis \cref{hyp_coerc}. Let us consider two points $m_-$ and $m_+$ of $\mmm$, and a solution $(x,t)\mapsto u(x,t)$ of system \cref{hyp_syst}. Let us assume that hypotheses \textup{(\hyperlink{hypHom}{\hypHomRef})} and \textup{(\hyperlink{hypNoInv}{\hypNoInvRef})} of \cref{prop:relax} hold, and let us keep all the notation of \cref{sec:no_inv_implies_relax}, together with the notation $\usmooth$ and $\usmall$ introduced after \cref{lem:smooth_plus_small}, ensuring (for all $(x,t)$ in $\rr\times[0,+\infty)$) the decomposition
\begin{equation}
\label{usmooth_plus_usmall}
u(x,t) = \usmooth(x,t) + \usmall(x,t)
\,.
\end{equation}
\subsubsection{Notation}
According to \cref{prop:relax}, the quantities $V(m_-)$ and $V(m_+)$ are equal. 
\begin{notation}
Let
\[
\valueOfV = V(m_-) = V(m_+)
\,,
\]
and
\[
\mmm_{\valueOfV} = \mmm\cap V^{-1}(\{\valueOfV\}) = \{m\in\mmm: V(m)=\valueOfV\}
\,,
\]
and let $\Phi_0(\valueOfV)$ denote the union, for all ordered pairs $(m_1,m_2)$ of points of $\mmm_{\valueOfV}$, of the sets $\Phi_0(m_1,m_2)$ defined in \cref{subsubsec:stat_sol_trav_fronts}:
\begin{equation}
\label{definition_PhiZero_of_h}
\Phi_0(\valueOfV)=\bigsqcup_{(m_1,m_2)\in \mmm_{\valueOfV}^2}\Phi_0(m_1,m_2) 
\,.
\end{equation}
For every function $\xi\mapsto \phi(\xi)$ in $\Phi_0(\valueOfV)$, let
\[
I(\phi)=\bigcup_{\xi\in\rr} \ \bigl\{\bigl(\phi(\xi),\phi'(\xi)\bigr)\bigr\}
\]
denote the ``image'' of $\phi$, and let $I(\bigl(\Phi_0(\valueOfV)\bigr)$ denote the union of all images of bistable stationary solutions connecting minimum points in the level set $V^{-1}(\{\valueOfV\})$:
\[
I(\bigl(\Phi_0(\valueOfV)\bigr) = \bigcup_{\phi\in\Phi_0(\valueOfV)} \  \ I(\phi)
\,.
\]
\end{notation}
\subsubsection{Additional hypotheses}
Let us introduce the following hypotheses. 
\begin{description}
\item[\hypOnlyMinOfLabel{\valueOfV}]\hypertarget{hypOnlyMinOf} All critical points of $V$ in the level set $V^{-1}(\{\valueOfV\})$ are nondegenerate minimum points. In other words, for every $v$ in $\rr^d$, 
\[
V(v)=\valueOfV
\quad\text{and}\quad
\nabla V(v)=0
\implies
D^2V(v)
\text{ is positive definite.}
\]
\item[\hypDiscStatOfLabel{\valueOfV}]\hypertarget{hypDiscStatOf} For every $m_1$ in $\mmm_{\valueOfV}$, the set
\[
\bigsqcup_{m_2\in\mmm_{\valueOfV}}\bigl\{ \bigl( \phi(0), \phi'(0) \bigr) : \phi \in \PhiZeroNorm(m_1,m_2) \bigr\}
\]
is totally disconnected in $\rr^{2d}$ (that is, its connected components are singletons). Equivalently, the set 
\begin{equation}
\label{def_phiZeronorm_of_mmm_h}
\PhiZeroNorm(\valueOfV) = \bigcup_{(m_1,m_2)\in \mmm_{\valueOfV}^2}\PhiZeroNorm(m_1,m_2)
\end{equation}
is totally disconnected for the topology of compact convergence (uniform convergence on compact subsets of $\rr$). 
\end{description}
\subsection{Statement}
\begin{proposition}[approach to the set of bistable stationary solutions / to a standing terrace of bistable solutions]
\label{prop:approach_set_bistable_stat_sol_stand_terrace}
Assume that the potential $V$ satisfies the coercivity hypothesis \cref{hyp_coerc} and that hypotheses \textup{(\hyperlink{hypHom}{\hypHomRef})} and \textup{(\hyperlink{hypNoInv}{\hypNoInvRef})} hold for the solution $(x,t)\mapsto u(x,t)$ under consideration. Then, in addition to the conclusions of \cref{prop:relax}, the following conclusions hold.
\begin{enumerate}
\item If hypothesis \textup{(\hyperlink{hypOnlyMinOf}{\hypOnlyMinOfRef{\valueOfV}})} holds, then the quantity
\[
\sup_{x\in\iMain(t)}\dist\biggl(\Bigl(u(x,t),\partial_x\usmooth(x,t)\Bigr) \,, \,I\bigl(\Phi_0(\valueOfV)\bigr)\biggr)
\]
goes to $0$ as time goes to $+\infty$. 
\label{item:approach_set_bist_stat_sol}
\item If both hypotheses \textup{(\hyperlink{hypOnlyMinOf}{\hypOnlyMinOfRef{\valueOfV}})} and \textup{(\hyperlink{hypDiscStatOf}{\hypDiscStatOfRef{\valueOfV}})} hold, then there exists a standing terrace of bistable stationary solutions $(x,t)\mapsto \ttt(x,t)$, connecting $m_-$ to $m_+$, such that the quantity
\begin{equation}
\label{approach_standing_terrace}
\sup_{x\in\iMain(t)}\abs{u(x,t) - \ttt(x,t)}
\end{equation}
goes to $0$ as time goes to $+\infty$. In addition, the residual asymptotic energy $\eeeResAsympt[u]$ of the solution equals the energy $\eee[\ttt]$ of this standing terrace. 
\label{item:approach_standing_terrace_bist_stat_sol}
\end{enumerate}
\end{proposition}
\subsection{Approach to normalized Hamiltonian level set zero for a sequence of times}
\label{subsec:approach_normalized_Ham_level_set_zero_seq_times}
Let us introduce the normalized potential function $V^\ddag$ defined as in \vref{def_V_ddag}, and the normalized Hamiltonian $H^\ddag$ defined as
\[
H^\ddag:\rr^d\times\rr^d \to \rr, \quad (u,v)\mapsto \frac{1}{2}v^2 - V^\ddag(u)
\,.
\]
The goal of this \namecref{subsec:approach_normalized_Ham_level_set_zero_seq_times} is to prove the following lemma. 
\begin{lemma}[approach to normalized Hamiltonian level set zero for a sequence of times]
\label{lem:lim_inf_sup_H}
Assume that hypotheses \cref{hyp_coerc} and \textup{(\hyperlink{hypHom}{\hypHomRef})} and \textup{(\hyperlink{hypNoInv}{\hypNoInvRef})} hold. Then the following limit holds:
\begin{equation}
\label{lim_inf_sup_H}
\liminf_{t\to+\infty}\, \sup_{x\in\iMain(t)} \, \abs{ H^\ddag\bigl(\usmooth(x,t),\partial_x\usmooth(x,t)\bigr)} =0 
\,. 
\end{equation}
\end{lemma}
Since this lemma does not require hypotheses \textup{(\hyperlink{hypOnlyMinOf}{\hypOnlyMinOfRef{\valueOfV}})} and \textup{(\hyperlink{hypDiscStatOf}{\hypDiscStatOfRef{\valueOfV}})}, let us ignore these two additional hypotheses throughout this \namecref{subsec:approach_normalized_Ham_level_set_zero_seq_times}. They will be introduced, when necessary, in the forthcoming \namecrefs{subsec:approach_normalized_Ham_level_set_zero_seq_times}.

Let us introduce the function $\hat{\ddd}^{(0)}(\cdot)$ defined, for every nonnegative time $t$, as
\[
\hat{\ddd}^{(0)}(t) = \int_{\rr}\chi_0(x,t) \partial_t^2\usmooth(x,t)^2 \, dx
\,,
\]
where $\chi_0$ is the function defined in \cref{subsubsec:en_dichot}. In the parabolic case, \cref{lem:lim_inf_sup_H} can be derived from the integrability of the function $t\mapsto \ddd^{(0)}(t)$ on $[0,+\infty)$, see \cite[\GlobalRelaxationLemApproachZeroHamSequenceTimes]{Risler_globalRelaxation_2016}. In the hyperbolic case considered here, the integrability on $[0,+\infty)$ of $\hat{\ddd}^{(0)}(t)$ will also be needed. 
\subsubsection{Integrability of \texorpdfstring{$t\mapsto \hat{\ddd}^{(0)}(t)$}{t->hat Dzero(t)}}
\label{subsubsec:integrability_of_hatDzero_of_t}
The aim of this \namecref{subsubsec:integrability_of_hatDzero_of_t} is to prove the following lemma (the proof will require several steps). 
\begin{lemma}[integrability of the square integral of $u_{tt}$]
\label{lem:integrability_of_square_integral_of_utt}
The function 
\[
t\mapsto \hat{\ddd}^{(0)}(t)
\]
is integrable on $[0,+\infty)$. 
\end{lemma}
For every real quantity $x$ and nonnegative time $t$, let 
\[
w(x,t) = \partial_t\usmooth(x,t)
\,, \quad\text{so that}\quad
\hat{\ddd}^{(0)}(t) = \int_{\rr}\chi_0(x,t) w_t(x,t)^2 \, dx
\,.
\]
According to its definition \cref{def_Usmooth_Usmall,def_usmooth_usmall}, the function $\usmall$ satisfies the system
\[
\alpha\partial_t^2\usmall + \partial_t\usmall = - \usmall + \partial_x^2\usmall
\,,
\]
so that, according to system \cref{hyp_syst} and the decomposition \cref{usmooth_plus_usmall}, the function $\usmooth$ satisfies the system
\begin{equation}
\label{system_u_smooth}
\alpha\partial_t^2\usmooth + \partial_t\usmooth = -\nabla V(u) + \usmall + \partial_x^2\usmooth
\,,
\end{equation}
and its time derivative $w$ satisfies the system
\begin{equation}
\label{hyp_syst_ut}
\alpha w_{tt} + w_t = - D^2V(u)\cdot u_t + \partial_t\usmall + w_{xx}
\,.
\end{equation}
For every nonnegative time $t$, let 
\[
\begin{aligned}
\dddSmallZero(t) &= \int_{\rr} \chi_0(x,t)\left(\partial_t\usmall(x,t)\right)^2 \, dx \,, \\
\text{and}\quad
\hat{\eee}^{(0)}(t) &= \int_{\rr} \chi_0(x,t)\Bigl(\frac{\alpha}{2} w_t(x,t)^2 + \frac{1}{2} w_x(x,t)^2\Bigr)\, dx
\,.
\end{aligned}
\]
\begin{lemma}[time derivative of localized $w$-energy]
\label{lem:time_derivative_localized_w_energy}
There exists a positive quantity $\KhatEzeroDzero$, depending only on $V$ and $\alpha$, such that, for every nonnegative time $t$, 
\begin{equation}
\label{time_derivative_localized_w_energy}
\begin{aligned}
\frac{d}{dt}\hat{\eee}^{(0)}(t) &\le -\frac{1}{2} \hat{\ddd}^{(0)}(t) + \KhatEzeroDzero \ddd^{(0)}(t) + 4 \dddSmallZero(t)  \\
&\qquad + \frac{\kappa_0(1+\cCutZero)}{2} \int_{\rr\setminus\iMain(t)}\chi_0(x,t) w_x(x,t)^2\, dx
\,.
\end{aligned}
\end{equation}
\end{lemma}
\begin{proof}[Proof of \cref{lem:time_derivative_localized_w_energy}]
It follows from the hyperbolic system \cref{hyp_syst_ut} satisfied by $w$ that, for every nonnegative time $t$, 
\[
\begin{aligned}
\frac{d}{dt}\hat{\eee}^{(0)}(t) &= \int_{\rr} \biggl[\partial_t\chi_0\Bigl(\frac{\alpha}{2} w_t^2 + \frac{1}{2} w_x^2\Bigr)+\chi_0(\alpha w_t \cdot w_{tt} + w_x \cdot w_{xt}) \biggr] \, dx \\
&= \int_{\rr} \biggl[\partial_t\chi_0\Bigl(\frac{\alpha}{2} w_t^2 + \frac{1}{2} w_x^2\Bigr)+\chi_0\, w_t\cdot(\alpha w_{tt} -w_{xx})- \partial_x\chi_0\, w_x\cdot w_t\biggr] \, dx \\
&= - \hat{\ddd}^{(0)}(t) + \int_{\rr} \biggl[\partial_t\chi_0\Bigl(\frac{\alpha}{2} w_t^2 + \frac{1}{2} w_x^2\Bigr) +\chi_0\, w_t \cdot \bigl(-D^2V(u)\cdot u_t + \partial_t\usmall \bigr) \\
&\qquad - \partial_x\chi_0\, w_x \cdot w_t\biggr] \, dx 
\,.
\end{aligned}
\]
The following inequalities hold:
\[
\begin{aligned}
\partial_t\chi_0 &\le \kappa_0\cCutZero  \chi_0 \,,\\
\text{and}\quad
\abs{w_t \cdot D^2V(u)\cdot u_t} &\le \frac{1}{4}w_t^2 + \abs{D^2V(u)\cdot u_t}^2 \,, \\
\text{and}\quad
\abs{w_t \cdot \partial_t\usmall} &\le \frac{1}{16}w_t^2 + 4\abs{\partial_t\usmall}^2 \,, \\
\text{and}\quad
\abs{\partial_x\chi_0\, w_x\cdot w_t} &\le \chi_0\frac{\kappa_0}{2}( w_x^2 + w_t^2)
\,,
\end{aligned}
\]
and according to the definitions \cref{def_kappaZero,def_cCutZero} of $\kappa_0$ and $\cCutZero$,
\[
\frac{\alpha\kappa_0\cCutZero}{2}\le \frac{1}{32}
\quad\text{and}\quad
\frac{\kappa_0}{2}\le \frac{1}{8}
\,.
\]
It follows that
\begin{equation}
\label{time_derivative_localized_w_energy_proof}
\begin{aligned}
\frac{d}{dt}\hat{\eee}^{(0)}(t) &\le -\frac{1}{2} \hat{\ddd}^{(0)}(t) + \int_{\rr} \chi_0 \abs{D^2V(u)\cdot w}^2 \, dx + 4 \dddSmallZero(t) \\
&\qquad + \frac{\kappa_0(1+\cCutZero)}{2} \int_{\rr\setminus\iMain(t)}\chi_0\, w_x^2 \, dx
\,.
\end{aligned}
\end{equation}
Thus, introducing the quantities 
\[
\begin{aligned}
\eigVmax &= \max \bigl\{ \text{eigenvalues of } D^2V(v): v\in\rr^d, \quad \abs{v}\le \Rattinfty\bigr\} \,, \\
\text{and}\quad
\KhatEzeroDzero &= \eigVmax^2
\,,
\end{aligned}
\]
inequality \cref{time_derivative_localized_w_energy} follows from inequality \cref{time_derivative_localized_w_energy_proof}. \Cref{lem:time_derivative_localized_w_energy} is proved. 
\end{proof}
For every real quantity $x$ and nonnegative time $t$, let
\[
\begin{aligned}
\hat F^{(0)}(x,t) &= 2\alpha\Bigl(\frac{\alpha}{2}w_t(x,t)^2 + \frac{1}{2}w_x(x,t)^2\Bigr) + \Bigl(\alpha w (x,t)\cdot w_t(x,t) + \frac{1}{2}w(x,t)^2\Bigr) \\
&= \alpha^2 w_t(x,t)^2 + \alpha w_x(x,t)^2 + \alpha w (x,t)\cdot w_t(x,t) + \frac{1}{2}w(x,t)^2
\,,
\end{aligned}
\]
see the discussion in \cref{subsubsec:funct_sf} and the definition \cref{def_fireZero} of $F^\dag_0(x,t)$. It follows from this definition that
\begin{equation}
\label{coercivity_hat_F_zero}
\hat F^{(0)}(x,t) \ge \alpha w_x(x,t)^2 + \frac{1}{4} w(x,t)^2
\,.
\end{equation}
For every nonnegative time $t$, let
\[
\hat{\fff}^{(0)}_{\pm}(t) = \int_{\rr} \psi_{0,\pm}(x,t) \hat F^{(0)}(x,t)\, dx 
\,.
\]
According to inequality \cref{coercivity_hat_F_zero}, both quantities $\fff^{(0)}_{-}(t)$ and $\fff^{(0)}_{+}(t)$ are nonnegative, and, since
\[
\chi_0(x,t) = \psi_{0,-}(x,t)
\quad\text{for $x$ in $\iLeft(t)$, and} \quad
\chi_0(x,t) = \psi_{0,+}(x,t)
\quad\text{for $x$ in $\iRight(t)$,}
\]
it follows that
\[
\hat{\fff}^{(0)}_{-}(t) + \hat{\fff}^{(0)}_{+}(t) \ge \alpha\int_{\rr\setminus\iMain(t)}\chi_0(x,t) w_x(x,t)^2\, dx
\,,
\]
and thus, in view of \cref{time_derivative_localized_w_energy}, that
\begin{equation}
\label{time_derivative_localized_w_energy_with_hat_fff_zero}
\begin{aligned}
\frac{d}{dt}\hat{\eee}^{(0)}(t) &\le -\frac{1}{2} \hat{\ddd}^{(0)}(t) + \KhatEzeroDzero \ddd^{(0)}(t) + 4 \dddSmallZero(t) \\
&\qquad + \frac{\kappa_0(1+\cCutZero)}{2\alpha} \bigl(\hat{\fff}^{(0)}_{-}(t) + \hat{\fff}^{(0)}_{+}(t)\bigr)
\,.
\end{aligned}
\end{equation}
\begin{lemma}[linear decrease up to pollution for $\hat{\fff}^{(0)}_{\pm}(t)$]
\label{lem:time_derivative_hat_fff_zero}
There exist positive quantities $\nuHatFZero$ and $\KHatFZero$ such that, for every nonnegative time $t$, 
\begin{equation}
\label{time_derivative_hat_fff_zero}
\frac{d}{dt}\hat{\fff}^{(0)}_{\pm}(t) \le -\nuHatFZero \hat{\fff}^{(0)}_{\pm}(t) + \KHatFZero \ddd^{(0)}(t) + \left(4\alpha+\frac{1}{2}\right)\dddSmallZero(t)
\,.
\end{equation}
The quantity $\nuHatFZero$ depends only on $\alpha$, and the quantity $\KHatFZero$ depends only on $\alpha$ and $V$. 
\end{lemma}
\begin{proof}[Proof of \cref{lem:time_derivative_hat_fff_zero}]
It follows from the hyperbolic system \cref{hyp_syst_ut} satisfied by $w$ that, for every nonnegative time $t$, 
\[
\begin{aligned}
\frac{d}{dt}\hat{\fff}^{(0)}_{\pm}(t) =& \int_{\rr}\biggl[\partial_t \psi_{0,\pm} \hat F^{(0)} + \psi_{0,\pm}\Bigl(-\alpha w_t^2 -w_x^2 \\
+ (2\alpha w_t+w)&\cdot\left(-D^2V(u)\cdot u_t + \partial_t\usmall\right) \Bigr) - 2\alpha \partial_x \psi_{0,\pm} w_x\cdot w_t + \frac{1}{2}\partial_{xx}\psi_{0,\pm} w^2 \biggr]\, dx
\,.
\end{aligned}
\]
Since
\[
\abs{\partial_t \psi_0}\le \kappa_0\cCutZero \psi_{0,\pm}
\quad\text{and}\quad
\abs{\partial_x \psi_{0,\pm}} = \kappa_0 \psi_{0,\pm}
\quad\text{and}\quad
\partial_{xx}\psi_{0,\pm} \le \kappa_0^2 \psi_{0,\pm} 
\,,
\]
and since
\[
\begin{aligned}
\abs{2\alpha w_t \cdot D^2V(u)\cdot u_t} &\le \frac{\alpha}{4}w_t^2 + 4\alpha\abs{D^2V(u)\cdot u_t}^2 \,,\\
\text{and}\quad
\abs{2 w_x\cdot w_t} &\le w_x^2 + w_t^2 \,, \\
\text{and}\quad
\abs{2\alpha w_t\cdot\partial_t\usmall} &\le \frac{\alpha}{4} w_t^2 + 4\alpha\abs{\partial_t\usmall}^2 \,,\\
\text{and}\quad
\abs{w\cdot D^2V(u)\cdot u_t} &= \abs{\bigl(u_t - \partial_t\usmall\bigr)\cdot D^2V(u)\cdot u_t} \\
&\le \abs{u_t\cdot D^2V(u)\cdot u_t} + \frac{1}{2}\abs{D^2V(u)\cdot u_t}^2 + \frac{1}{2} \abs{\partial_t\usmall}^2 \,, \\
\text{and}\quad
\abs{w\cdot\partial_t\usmall}&\le \frac{1}{2}w^2 + \frac{1}{2}\abs{\partial_t\usmall}^2
\,,
\end{aligned}
\]
it follows  that
\[
\begin{aligned}
\frac{d}{dt}\hat{\fff}^{(0)}_{\pm}(t) &\le \kappa_0\cCutZero \hat{\fff}^{(0)}_{\pm}(t)+ \int_{\rr}\psi_{0,\pm}\biggl[\alpha\Bigl(-1 +\kappa_0+ \frac{1}{4}+ \frac{1}{4}\Bigr) w_t^2 +(-1+\alpha\kappa_0)w_x^2 \\
 + \left(4\alpha+\frac{1}{2}\right)&\abs{D^2V(u)\cdot u_t}^2  + \abs{u_t\cdot D^2V(u)\cdot u_t} + \frac{1+\kappa_0}{2} w^2 + (4\alpha + 1)\abs{\partial_t\usmall}^2\biggr]\, dx
\,.
\end{aligned}
\]
Let $\nuHatFZero$ be a (small) positive quantity to be chosen below. Since
\[
\hat F^{(0)}(x,t) \le \frac{3}{2}\alpha^2 w_t(x,t)^2 + \alpha w_x(x,t)^2 + w(x,t)^2
\,,
\]
it follows that
\begin{equation}
\label{d_over_dt_hat_fff_zero_step}
\begin{aligned}
\frac{d}{dt}\hat{\fff}^{(0)}_{\pm}(t) &+ \nuHatFZero \hat{\fff}^{(0)}_{\pm}(t) \le  \int_{\rr}\psi_{0,\pm}\biggl[ \alpha\Bigl(-\frac{1}{2}+\kappa_0 + \frac{3\alpha}{2}\left(\kappa_0\cCutZero + \nuHatFZero\right)\Bigr) w_t^2 \\
& + \Bigl(-1+\alpha\kappa_0 + \alpha\bigl(\kappa_0\cCutZero + \nuHatFZero\bigr)\Bigr)w_x^2+ 4\alpha\abs{D^2V(u)\cdot u_t}^2  \\
& + \abs{u_t\cdot D^2V(u)\cdot u_t} + \Bigl(\frac{1+\kappa_0}{2}+ \kappa_0\cCutZero + \nuHatFZero\Bigr) w^2  \\
& + \left(4\alpha + \frac{1}{2}\right)\abs{\partial_t\usmall}^2 \biggr]\, dx
\,.
\end{aligned}
\end{equation}
Let us choose 
\[
\nuHatFZero = \frac{1}{16\alpha}
\,.
\]
According to the definitions \cref{def_kappaZero,def_cCutZero} of $\kappa_0$ and $\cCutZero$, 
\[
\kappa_0\le\frac{1}{4}
\quad\text{and}\quad
\alpha\kappa_0\le \frac{1}{4}
\quad\text{and}\quad
\alpha\kappa_0\cCutZero \le \frac{1}{16}
\,,
\]
so that, according to the choice of $\nuHatFZero$ above, 
\[
-\frac{1}{2}+\kappa_0 + \frac{3\alpha}{2}(\kappa_0\cCutZero + \nuHatFZero) \le -\frac{5}{16}\le 0
\quad\text{and}\quad
-1+\alpha\kappa_0 + \alpha(\kappa_0\cCutZero + \nuHatFZero) -\frac{5}{8}\le 0
\,.
\]
Thus, introducing the quantity
\[
\KHatFZero = 4\alpha \eigVmax^2 + \eigVmax + \frac{1+\kappa_0}{2}+ \kappa_0\cCutZero + \nuHatFZero
\,,
\]
inequality \cref{time_derivative_hat_fff_zero} follows from inequality \cref{d_over_dt_hat_fff_zero_step} (using the fact that $\psi_{0,\pm}$ is less than or equal to $\chi_0$). \Cref{lem:time_derivative_hat_fff_zero} is proved. 
\end{proof}
\begin{proof}[Proof of \cref{lem:integrability_of_square_integral_of_utt}]
It follows from inequalities \cref{time_derivative_localized_w_energy_with_hat_fff_zero,time_derivative_hat_fff_zero} that, for every nonnegative time $t$, 
\[
\begin{aligned}
&\hat{\ddd}^{(0)}(t) \le -2 \frac{d}{dt}\hat{\eee}^{(0)}(t) + 2 \Bigl(\KhatEzeroDzero + \frac{\kappa_0(1+\cCutZero)\KHatFZero}{\alpha\nuHatFZero}\Bigr)\ddd^{(0)}(t) \\
&\quad + \left(8+\frac{\kappa_0(1+\cCutZero)}{\alpha\nuHatFZero}\left((4\alpha + \frac{1}{2}\right)\right)\dddSmallZero(t)
- \frac{\kappa_0(1+\cCutZero)}{\alpha\nuHatFZero}\frac{d}{dt}\left(\hat{\fff}^{(0)}_-(t)+\hat{\fff}^{(0)}_-(t)\right)
\,.
\end{aligned}
\]
As a consequence, for every nonnegative time $T$, since the quantities $\hat{\eee}^{(0)}(T)$ and $\hat{\fff}^{(0)}_-(T)$ and $\hat{\fff}^{(0)}_+(T)$ are nonnegative, it follows that
\[
\begin{aligned}
\int_0^T \hat{\ddd}^{(0)}(t) \, dt &\le  2 \hat{\eee}^{(0)}(0) + 2 \left(\KhatEzeroDzero + \frac{\kappa_0(1+\cCutZero)\KHatFZero}{\alpha\nuHatFZero}\right)\int_0^{+\infty} \ddd^{(0)}(t) \, dt \\
&\quad + \left(8+\frac{\kappa_0(1+\cCutZero)}{\alpha\nuHatFZero}\left((4\alpha + \frac{1}{2}\right)\right) \int_0^{+\infty}\dddSmallZero(t) \, dt \\
&\quad + \frac{\kappa_0(1+\cCutZero)}{\alpha\nuHatFZero}\left(\hat{\fff}^{(0)}_-(0)+\hat{\fff}^{(0)}_-(0)\right)
\,.
\end{aligned}
\]
According to \cref{lem:dissip_is_integrable}, and since according to \cref{lem:smooth_plus_small} the quantity $\norm{\Usmall(t)}_X$ goes to $0$ at an exponential rate as $t$ goes to $+\infty$, the quantity to the right of this inequality is finite; since this quantity does not depend on $T$, \cref{lem:integrability_of_square_integral_of_utt} is proved. 
\end{proof}
\subsubsection{Proof of \texorpdfstring{\cref{lem:lim_inf_sup_H}}{Lemma \ref{lem:lim_inf_sup_H}}}
\begin{proof}[Proof of \cref{lem:lim_inf_sup_H}]
Let us proceed by contradiction and assume that the converse is true. Then there exists a positive quantity $\delta$ such that, for every large enough positive time $t$, 
\begin{equation}
\label{hyp_H_not_small}
\sup_{x\in\iMain(t)}\abs{H^\ddag\bigl(\usmooth(x,t),\partial_x\usmooth(x,t)\bigr)}\ge\delta 
\,. 
\end{equation}
For every $(x,t)$ in $\rr\times[0,+\infty)$, let 
\[
\nnn(x,t) = \nabla V^\ddag\bigl(u(x,t)\bigr)-\nabla V^\ddag\bigl(\usmooth(x,t)\bigr)-\usmall(x,t)-\partial_t\usmall(x,t)
\,;
\]
it follows from system \cref{system_u_smooth} satisfied by $\usmooth$ that
\[
\partial_x \Bigl( H^\ddag\bigl(\usmooth,\partial_x\usmooth\bigr) \Bigr) = \partial_x\usmooth\cdot \left(\alpha \partial_t^2\usmooth + u_t + \nnn\right) 
\,
\]
where the arguments of $\usmooth$ and its partial derivatives and of $\nnn$ are $(x,t)$ everywhere. 
As a consequence, it follows from \cref{hyp_H_not_small} that
\[
\liminf_{t\to+\infty}\int_{\iMain(t)}\abs{\partial_x\usmooth(x,t)\cdot \left(\alpha \partial_t^2\usmooth(x,t) + u_t(x,t)+ \nnn(x,t)\right)}\, dx\ge 2\delta
\,.
\]
Since according to \cref{lem:smooth_plus_small} the quantity $\norm{\Usmall(t)}_X$ goes to $0$ at an exponential rate as $t$ goes to $+\infty$, it follows that the previous limit still holds if the term $\nnn(x,t)$ is dropped, that is,
\[
\liminf_{t\to+\infty}\int_{\iMain(t)}\abs{\partial_x\usmooth(x,t)\cdot \left(\alpha \partial_t^2\usmooth(x,t) + u_t(x,t)\right)}\, dx\ge 2\delta
\,.
\]
Thus it follows from Cauchy--Schwarz inequality and from the bound \cref{hyp_attr_ball_X_relax} on the $\Ltwoul$-norm of $u_x$ that the limit
\[
\liminf_{t\to+\infty}\, t \int_{\iMain(t)} \bigl(\alpha \partial_t^2\usmooth(x,t) + u_t(x,t) \bigr)^2\, dx
\]
is positive. Since 
\[
\begin{aligned}
\int_{\iMain(t)} \bigl(\alpha \partial_t^2\usmooth(x,t) + u_t(x,t) \bigr)^2\, dx &\le 
\int_{\rr}\chi_0(x,t)\bigl(\alpha \partial_t^2\usmooth(x,t) + u_t(x,t) \bigr)^2\, dx \\
&\le 2\int_{\rr}\chi_0(x,t)\bigl(\alpha^2 \partial_t^2\usmooth(x,t)^2 + u_t(x,t)^2 \bigr)\, dx \\
&= 2\alpha^2 \hat{\ddd}^{(0)}(t) + 2\ddd^{(0)}(t) 
\,,
\end{aligned}
\]
it follows that the limit
\[
\liminf_{t\to+\infty}\, t \left(\ddd^{(0)}(t) + \hat{\ddd}^{(0)}(t)\right)
\]
is positive, a contradiction with \cref{lem:dissip_is_integrable,lem:integrability_of_square_integral_of_utt}. \Cref{lem:lim_inf_sup_H} is proved.
\end{proof}
\subsection{Approach to normalized Hamiltonian level set zero for all times}
\begin{lemma}[approach to normalized Hamiltonian level set zero for all times]
\label{lem:H_small}
Assume that, in addition to hypotheses \cref{hyp_coerc} and \textup{(\hyperlink{hypHom}{\hypHomRef})} and \textup{(\hyperlink{hypNoInv}{\hypNoInvRef})}, hypothesis \textup{(\hyperlink{hypOnlyMinOf}{\hypOnlyMinOfRef{\valueOfV}})} holds. Then the following limit holds:
\[
\sup_{x\in\iMain(t)} \abs{H^\ddag\left(\usmooth(x,t),\partial_x\usmooth(x,t)\right)}\to 0
\quad\text{as}\quad t\to+\infty 
\,. 
\]
\end{lemma}
\begin{proof}
See the proof of \cite[\GlobalRelaxationLemApproachZeroHamAllTimes]{Risler_globalRelaxation_2016}.
\end{proof}
\subsection{Approach to the set of bistable stationary solutions in the normalized Hamiltonian level set zero}
The following lemma completes the proof of conclusion \cref{item:approach_set_bist_stat_sol} of \cref{prop:approach_set_bistable_stat_sol_stand_terrace}.
\begin{lemma}[approach to bistable stationary solutions in the normalized Hamiltonian level set zero]
\label{lem:app_bist}
Assume that, in addition to hypotheses \cref{hyp_coerc} and \textup{(\hyperlink{hypHom}{\hypHomRef})} and \textup{(\hyperlink{hypNoInv}{\hypNoInvRef})}, hypothesis \textup{(\hyperlink{hypOnlyMinOf}{\hypOnlyMinOfRef{\valueOfV}})} holds. Then the following limit holds:
\[
\sup_{x\in\iMain(t)}\dist\Bigl(\bigl(\usmooth(x,t),\partial_x\usmooth(x,t)\bigr) \,, \,I\bigl(\Phi_0(\valueOfV)\bigr)\Bigr)\to 0
\quad\text{as}\quad
t\to +\infty
\,.
\]
\end{lemma}
\begin{proof}
See the proof of \cite[\GlobalRelaxationLemApproachSetBistableStationarySolutions]{Risler_globalRelaxation_2016}.
\end{proof}
In view of \cref{lem:app_bist}, conclusion \cref{item:approach_set_bist_stat_sol} of \cref{prop:approach_set_bistable_stat_sol_stand_terrace} is proved. 
\subsection{Approach to a standing pattern of bistable stationary solutions}
The proof of conclusion \cref{item:approach_standing_terrace_bist_stat_sol} of \cref{prop:approach_set_bistable_stat_sol_stand_terrace} is identical to the proof of the same result in the parabolic case, see \cite[\GlobalRelaxationSecApproachSetBistableStatSolAndConvergenceStandTerr]{Risler_globalRelaxation_2016}. To keep track of the Escape points, the same method as the one used for travelling fronts in \cref{subsec:convergence} (again the ``smooth plus small'' decomposition) can be called upon. Once the standing terrace $\ttt(x,t)$ is defined and the convergence towards $0$ of the quantity \cref{approach_standing_terrace} is proved, the equality between the residual asymptotic energy $\eeeResAsympt[u]$ of the solution and the energy $\eee[\ttt]$ of the standing terrace can be proved by the same arguments as those of of \cite[\GlobalRelaxationSubsecValueAsymptoticEnergy]{Risler_globalRelaxation_2016}. 
\section{Proof of \texorpdfstring{\cref{thm:1,prop:resid_asympt_energy}}{Theorem \ref{thm:1} and Proposition \ref{prop:resid_asympt_energy}}}
\label{sec:proof_thm_1}
As everywhere else, let us consider a function $V$ in $\ccc^2(\rr^d,\rr)$ satisfying the coercivity hypothesis \cref{hyp_coerc}. Let us assume in addition that the generic hypotheses \cref{hyp_gen} hold for $V$, and let us consider a bistable solution $(x,t)\mapsto u(x,t)$ of system \cref{hyp_syst}. The conclusions of \vref{thm:1,prop:resid_asympt_energy} can be split into two parts.
\begin{enumerate}
\item The approach to the propagating terrace of bistable fronts travelling to the right, and to the one travelling to the left. 
\item On the remaining ``centre'' spatial domain, the fact that the time derivative of the solution goes to zero, and the fact that the ``residual asymptotic energy'' is nonnegative. 
\end{enumerate}
Concerning the first part, it is a rather direct consequence of \vref{prop:inv_cv} (``invasion implies convergence''), and the derivation of this first part from this proposition is unchanged with respect to the parabolic case; it is explained in details in \cite[\GlobalBehaviourSecProofTheorem]{Risler_globalBehaviour_2016}. 

As far as the second part is concerned, it may be assumed that between the ``last'' fronts travelling to the right and to the left, the hypotheses (and thus the conclusions) of \vref{prop:relax} (``no invasion implies relaxation'') hold. Then the conclusions of \cref{thm:1,prop:resid_asympt_energy} concerning the behaviour of the solution in this centre area follow from the conclusions of \cref{prop:relax,prop:approach_set_bistable_stat_sol_stand_terrace}. \Cref{thm:1,prop:resid_asympt_energy} are proved. 
\section{Spatial asymptotics of the profiles of travelling waves}
\label{sec:prop_stand_trav}
Let us assume that $V$ satisfies hypothesis \cref{hyp_coerc}, let $c$ denote a nonnegative quantity, and let us consider the differential system governing the profiles of waves travelling at the speed $c$ (or ``standing'' if $c$ equals $0$): 
\begin{equation}
\label{syst_trav_front_bis}
\phi''=-c\phi'+\nabla V(\phi) \,.
\end{equation}
A proof of the following lemma can be found, for instance, in \cite{Risler_globalBehaviour_2016}.
\begin{lemma}[spatial asymptotics of the profiles of travelling waves]
\label{lem:asympt_behav_tw_2}
Let $m$ be in $\mmm$, and let $\xi\mapsto \phi(\xi)$ be a global solution of the differential system \cref{syst_trav_front_bis} satisfying 
\[
\abs{\phi(\xi)-m}\le \dEsc(m)
\quad\text{for every }
\xi \text{ in }[0,+\infty)
\quad\text{and}\quad
\phi(\cdot) \not\equiv m
\,.
\]
Then the following conclusions hold. 
\begin{enumerate}
\item Both quantities $\abs{\phi(\xi)-m}$ and $\phi'(\xi)$ go to $0$ as $\xi$ goes to $+\infty$.
\label{item:cv_spatial_asymptotics_tw}
\item For all $\xi$ in $[0,+\infty)$, the scalar product $\bigl(\phi(\xi)-m\bigr) \cdot \phi'(\xi)$ is negative. 
\label{item:transv_spatial_asymptotics_tw}
\item For all $\xi$ in $(0,+\infty)$, the quantity $\abs{\phi(\xi)-m}$ is smaller than $\dEsc(m)$. 
\label{item:closer_spatial_asymptotics_tw}
\item The supremum $\sup_{\xi\in\rr}\abs{\phi(\xi)-m}$ is larger than $\dEsc(m)$. 
\label{item:escape_spatial_asymptotics_tw}
\item In addition to assertion \cref{item:cv_spatial_asymptotics_tw} above, the quantities
\[
e^{c\xi}\abs{\phi(\xi)-m}
\quad\text{and}\quad
e^{c\xi}\abs{\phi'(\xi)}
\]
go to $0$ at an exponential rate when $\xi$ goes to $+\infty$. 
\label{item:exp_cv_spatial_asymptotics_tw}
\end{enumerate}
\end{lemma}
\subsubsection*{Acknowledgements} 
I am indebted to Thierry Gallay and Romain Joly for their help and interest through numerous fruitful discussions. 
\emergencystretch=1em
\printbibliography 
\bigskip
\mySignature
\end{document}